\pdfoutput=1
\newif\ifpersonal
\personaltrue 
\RequirePackage[l2tabu,orthodox]{nag} 
\documentclass[10pt,a4paper,reqno]{amsart} 
\usepackage{amsmath,amsthm,amssymb,mathrsfs,mathtools,bm,eucal,tensor} 
\usepackage{microtype,fixltx2e,lmodern} 
\usepackage[utf8]{inputenc} 
\usepackage[T1]{fontenc} 
\usepackage{enumerate,comment,braket,xspace,tikz-cd,csquotes} 
\usepackage[all,cmtip]{xy} 
\usepackage[centering,vscale=0.7,hscale=0.8]{geometry}
\usepackage[hidelinks]{hyperref}
\usepackage[capitalize]{cleveref}

\theoremstyle{plain}
\newtheorem{thm-intro}{Theorem}
\newtheorem{thm}{Theorem}[section]
\newtheorem*{thm*}{Theorem}

\newtheorem{lem}[thm]{Lemma}
\newtheorem{prop}[thm]{Proposition}

\newtheorem{cor}[thm]{Corollary}

\theoremstyle{definition}
\newtheorem{defin}[thm]{Definition}
\newtheorem{notation}[thm]{Notation}
\newtheorem{eg}[thm]{Example}

\theoremstyle{remark}
\newtheorem{rem}[thm]{Remark}
\numberwithin{equation}{section}
\newtheorem{construction}[thm]{Construction}

\ifpersonal
\newcommand{\personal}[1]{\textcolor[rgb]{0,0,1}{(Personal: #1)}}
\newcommand{\todo}[1]{\textcolor{red}{(Todo: #1)}}
\else
\newcommand*{\personal}[1]{\ignorespaces}
\newcommand*{\todo}[1]{\ignorespaces}
\fi

\newcommand{\rB}{\mathrm B}

\newcommand{\rR}{\mathrm R}
\newcommand{\rT}{\mathrm T}

\newcommand{\rt}{\mathrm t}

\newcommand{\cC}{\mathcal C}
\newcommand{\cD}{\mathcal D}

\newcommand{\cF}{\mathcal F}

\newcommand{\cG}{\mathcal G}
\newcommand{\cI}{\mathcal I}
\newcommand{\cJ}{\mathcal J}
\newcommand{\cK}{\mathcal K}

\newcommand{\cO}{\mathcal O}

\newcommand{\cS}{\mathcal S}
\newcommand{\cT}{\mathcal T}

\newcommand{\cX}{\mathcal X}

\newcommand{\cZ}{\mathcal Z}
\DeclareFontFamily{U}{BOONDOX-calo}{\skewchar\font=45 }
\DeclareFontShape{U}{BOONDOX-calo}{m}{n}{<-> s*[1.05] BOONDOX-r-calo}{}
\DeclareFontShape{U}{BOONDOX-calo}{b}{n}{<-> s*[1.05] BOONDOX-b-calo}{}
\DeclareMathAlphabet{\mathcalboondox}{U}{BOONDOX-calo}{m}{n}

\newcommand{\bbA}{\mathbb A}

\newcommand{\bbG}{\mathbb G}
\newcommand{\bbL}{\mathbb L}

\newcommand{\bbT}{\mathbb T}
\newcommand{\bbZ}{\mathbb Z}

\newcommand{\bA}{\mathbf A}
\newcommand{\bB}{\mathbf B}
\newcommand{\bD}{\mathbf D}

\newcommand{\bT}{\mathbf T}

\newcommand{\bDelta}{\bm{\Delta}}

\makeatletter
\let\save@mathaccent\mathaccent
\newcommand*\if@single[3]{%
	\setbox0\hbox{${\mathaccent"0362{#1}}^H$}%
	\setbox2\hbox{${\mathaccent"0362{\kern0pt#1}}^H$}%
	\ifdim\ht0=\ht2 #3\else #2\fi
}
\newcommand*\rel@kern[1]{\kern#1\dimexpr\macc@kerna}
\newcommand*\widebar[1]{\@ifnextchar^{{\wide@bar{#1}{0}}}{\wide@bar{#1}{1}}}
\newcommand*\wide@bar[2]{\if@single{#1}{\wide@bar@{#1}{#2}{1}}{\wide@bar@{#1}{#2}{2}}}
\newcommand*\wide@bar@[3]{%
	\begingroup
	\def\mathaccent##1##2{%
		\let\mathaccent\save@mathaccent
		\if#32 \let\macc@nucleus\first@char \fi
		\setbox\z@\hbox{$\macc@style{\macc@nucleus}_{}$}%
		\setbox\tw@\hbox{$\macc@style{\macc@nucleus}{}_{}$}%
		\dimen@\wd\tw@
		\advance\dimen@-\wd\z@
		\divide\dimen@ 3
		\@tempdima\wd\tw@
		\advance\@tempdima-\scriptspace
		\divide\@tempdima 10
		\advance\dimen@-\@tempdima
		\ifdim\dimen@>\z@ \dimen@0pt\fi
		\rel@kern{0.6}\kern-\dimen@
		\if#31
		\overline{\rel@kern{-0.6}\kern\dimen@\macc@nucleus\rel@kern{0.4}\kern\dimen@}%
		\advance\dimen@0.4\dimexpr\macc@kerna
		\let\final@kern#2%
		\ifdim\dimen@<\z@ \let\final@kern1\fi
		\if\final@kern1 \kern-\dimen@\fi
		\else
		\overline{\rel@kern{-0.6}\kern\dimen@#1}%
		\fi
	}%
	\macc@depth\@ne
	\let\math@bgroup\@empty \let\math@egroup\macc@set@skewchar
	\mathsurround\z@ \frozen@everymath{\mathgroup\macc@group\relax}%
	\macc@set@skewchar\relax
	\let\mathaccentV\macc@nested@a
	\if#31
	\macc@nested@a\relax111{#1}%
	\else
	\def\gobble@till@marker##1\endmarker{}%
	\futurelet\first@char\gobble@till@marker#1\endmarker
	\ifcat\noexpand\first@char A\else
	\def\first@char{}%
	\fi
	\macc@nested@a\relax111{\first@char}%
	\fi
	\endgroup
}
\makeatother

\newcommand{\oS}{\widebar S}

\newcommand{\tS}{\widetilde S}

\newcommand{\tZ}{\widetilde Z}

\newcommand{\infcat}{$\infty$-category\xspace}
\newcommand{\infcats}{$\infty$-categories\xspace}

\newcommand{\inftopos}{$\infty$-topos\xspace}
\newcommand{\inftopoi}{$\infty$-topoi}

\newcommand{\Mod}{\mathrm{Mod}}

\newcommand{\Coh}{\mathrm{Coh}}

\newcommand{\dSt}{\mathrm{dSt}}
\newcommand{\Afd}{\mathrm{Afd}}
\newcommand{\Top}{\mathcal T\mathrm{op}}

\newcommand{\dR}{\mathrm{dR}}

\newcommand{\loc}{\mathrm{loc}}

\newcommand{\dAn}{\mathrm{dAn}}

\newcommand{\dAnk}{\mathrm{dAn}_k}
\newcommand{\Ank}{\mathrm{An}_k}

\newcommand{\cTank}{\cT_{\mathrm{an}}(k)}

\newcommand{\Str}{\mathrm{Str}}
\newcommand{\Strloc}{\mathrm{Str}^\mathrm{loc}}
\newcommand{\RTop}{\tensor*[^\rR]{\Top}{}}

\newcommand{\dAfd}{\mathrm{dAfd}}
\newcommand{\dAfdk}{\mathrm{dAfd}_k}

\newcommand{\trunc}{\mathrm{t}_0}

\newcommand{\CAlg}{\mathrm{CAlg}}

\newcommand{\Catinf}{\mathrm{Cat}_\infty}
\newcommand{\Catst}{\mathrm{Cat}_{\infty}^{\mathrm{st}}}

\newcommand{\anPreStk}{\mathrm{AnPreStk}}

\newcommand{\fib}{\mathrm{fib}}

\newcommand{\dAff}{\mathrm{dAff}}

\newcommand{\bfMap}{\mathbf{Map}}

\newcommand{\dAnSt}{\mathrm{dAnSt}}

\newcommand{\anNil}{\mathrm{AnNil}}
\newcommand{\anFMP}{\mathrm{AnFMP}}
\newcommand{\anFGrpd}{\mathrm{AnFGrpd}}

\newcommand{\Ind}{\mathrm{Ind}}

\newcommand{\laft}{\mathrm{laft}}

\newcommand{\pro}{\mathrm{Pro}}

\newcommand{\QCoh}{\mathrm{QCoh}}

\newcommand{\an}{^\mathrm{an}}
\newcommand{\alg}{^\mathrm{alg}}

\newcommand{\id}{\mathrm{id}}

\newcommand{\red}{\mathrm{red}}

\newcommand{\op}{^\mathrm{op}}

\newcommand{\cl}{\mathrm{cl}}

\usetikzlibrary{decorations.markings} 
\tikzset{
  closed/.style = {decoration = {markings, mark = at position 0.5 with { \node[transform shape, xscale = .8, yscale=.4] {/}; } }, postaction = {decorate} },
  open/.style = {decoration = {markings, mark = at position 0.5 with { \node[transform shape, scale = .7] {$\circ$}; } }, postaction = {decorate} }
}

\DeclareMathOperator{\Fun}{Fun}

\DeclareMathOperator{\Map}{Map}

\DeclareMathOperator{\Spec}{Spec}

\DeclareMathOperator{\Sym}{Sym}

\DeclareMathOperator*{\colim}{colim}

\begin{document}

\title{Spreading out the Hodge filtration in non-archimedean geometry}

\author{Jorge ANT\'ONIO}
\address{Jorge ANT\'ONIO, IRMA, UMR 7501
7 rue René-Descartes
67084 Strasbourg Cedex}
\email{jorgeantonio@unistra.fr}
\maketitle
\begin{abstract}
    The goal of the current text is to study non-archimedean analytic derived de Rham cohomology 
    by means of formal completions. Our approach is inspired by the deformation to the normal cone provided in \cite{Gaitsgory_Study_II}.
    More specifically, given a morphism $f \colon X \to Y$ of (derived) $k$-analytic spaces we construct the \emph{non-archimedean deformation to the normal cone}
    associated to $f$. The latter can be thought as an $\bA^1_k$-parametrized deformation whose fiber at $1 \in \bA^1_k$ coincides with the formal completion of $f$
    and the fiber at $0 \in \bA^1_k$ with the (derived) normal cone associated to $f$.
    We further show that such deformation can be endowed with a natural filtration which spreads out the usual Hodge filtration on the (completed shifted) analytic
    tangent bundle to the formal completion. Such filtration agrees with the $I$-adic filtration in the case where $f$ is a locally complete intersection morphism
    between (derived) $k$-affinoid spaces.
    Along the way we develop the theory of (ind-inf)-$k$-analytic spaces or in other words $k$-analytic formal moduli problems.
\end{abstract}

\tableofcontents

\section{Introduction}
\subsection{Background} Let $k$ be a field of characteristic zero and $X$ a variety over $k$. Following Illusie
we can associate to $X$ its \emph{derived de Rham
cohomology} as follows: we consider the
(Hodge completed) derived de Rham algebra
    \[
        \dR_{X/ k} \coloneqq  \Sym(\bbL_{X/ k}[-1])^\wedge 
    \]
where $(-)^\wedge$ stands for the completion of the usual symmetric algebra 
    \[\Sym(\bbL_{X/ Y}[-1]) \in \CAlg_k,\]
with respect to its natural Hodge filtration. Furthermore, we can identify
the graded pieces of the latter with exterior powers
    \[
        \wedge^i \bbL_{X/ Y} [-i], \quad \mathrm{for} \ i \ge 0,
    \]
c.f. \cite[Vol. II, \S VIII]{Illusie_Complexe_1971}.
The Hodge completed derived de Rham algebra is thus a natural generalization of the usual de Rham complex for singular varieties over $k$,
where we replace the sheaf of differentials
    \[
        \Omega^1_{X/ k},  
    \]
by the (algebraic) cotangent complex $\bbL_{X/ k}$ and force a spectral sequence of the form Hodge-to-de Rham to converge.

On the other hand, Hartshorne introduced in \cite{hartshorne1975rham} the \emph{algebraic de Rham cohomology} for singular varieties
$X/ k$ in terms
of formal completions: suppose we are given a closed immersion
$i \colon X \hookrightarrow Y$, where $Y$ is a smooth algebraic variety over $k$. Then one can form the \emph{formal completion}, $Y^\wedge_X$, of $Y$ at $X$ along $i$. Inspired by the theory of tubular neighborhoods, Hartshorne defined an algebraic cohomology theory
by taking global sections of $Y_X^\wedge$
    \[
        \Gamma(Y^\wedge_X, \cO_{Y^\wedge_X}).    
    \]
Moreover, the cohomology groups so obtained are independent of the choice of the closed embedding $i \colon  X\to Y$.
In \cite{Bhatt_Derived_Completions},
Bhatt proved a general comparison between the two mentioned notions. More precisely, the author proves a natural multiplicative
equivalence of complexes
    \[
        \dR_{X/k} \simeq  \Gamma(Y^\wedge_X, \cO_{Y^\wedge_X}),
    \]
c.f. \cite[Proposition 4.16]{Bhatt_Derived_Completions}.
In particular, the results of Bhatt imply that
the algebraic de Rham cohomology is equipped with a \emph{Hodge filtration}, which is finer than the standard infinitesimal Hodge filtration
and the Deligne-Hodge filtration.
Futhermore, when $f \colon \Spec A \to \Spec B$ is a closed immersion of affine (derived) schemes, the Hodge filtration
is finer than the $I$-adic filtration on the completion 
    \[
        Y^\wedge_X \simeq \mathrm{Spf} (B^\wedge_I),
    \]
where $I \coloneqq \fib(B \to A)$.
In the particular case where $f$ is a locally complete intersection morphism both the $I$-adic filtration
and the Hodge filtration coincide (c.f. \cite[Remark 4.13]{Bhatt_Derived_Completions}).

\subsection{Main results}
The main goal of the present text is to prove a non-archimedean analogue of the above results. Namely, let $k$ be a non-archimedean field of characteristic $0$,
we prove the following:

\begin{thm} \label{intro-thm1}
    Let $f \colon X \to Y$ be a morphism of (locally geometric) derived $k$-analytic stacks. Consider
    the \emph{Hodge completed derived de Rham complex associated to $f$}, defined as 
        \[\dR_{X/ Y} \an \coloneqq \Sym\an(\bbL_{X/ Y} \an[-1])^\wedge,\]
    in $\CAlg_k$.
    Then one has a natural multiplicative equivalence
        \[
            \dR_{X/ Y} \an \simeq \Gamma(Y^\wedge_X, \cO_{Y^\wedge_X}),
        \]
    of derived $k$-algebras,
    where $Y^\wedge_X$ denotes the derived $k$-analytic stack corresponding to the formal completion of $X$ in $Y$ along the morphism $f$.
\end{thm}

Unfortunately, the methods of B. Bhatt are algebraic in nature and cannot be directly extrapolated to the non-archimedean setting. Instead, we follow closely the
approach of Gaitsgory-Rozemblyum given in \cite[\S 9]{Gaitsgory_Study_II}. The main idea consists of defining the \emph{non-archimedean deformation
to the normal cone} as a derived $k$-analytic stack
    \[
        \cD\an_{X/ Y} \in \dAnSt_k,  
    \]
mapping naturally to the $k$-analytic affine line $\bA^1_k$. Denote by $p \colon \cD_{X/ Y}\an \to \bA^1_k$ the natural morphism. Then
\cref{intro-thm1} follows immediately from the following more precise result (see \cref{thm:main_thm}):

\begin{thm} \label{intro-thm2}
    Let $f \colon X \to Y$ be a morphism of locally geometric derived $k$-analytic stacks. Then the non-archimedean deformation to the normal cone satisfies the following
    assertions:
    \begin{enumerate}
        \item The fiber of $p \colon \cD_{X/ Y}\an \to \bA^1_k$ at $\{ 0\} \subseteq \bA^1_k$ identifies naturally with the completion of the shifted analytic
        tangent bundle
            \[
                \bT\an_{X/ Y}[1]^\wedge,  
            \]
        completed along the zero section $s_0 \colon X \to \bT\an_{X/ Y}[1]$. Moreover for $\lambda \neq 0$ we have a natural identification
            \[
                (\cD_{X/ Y}\an)_\lambda  \simeq Y^\wedge_X,
            \]
        where $Y^\wedge_X$ denotes the formal completion of $Y$ at $X$ along $f$.
        \item There exists a natural sequence of morphisms admitting a deformation theory
            \[
                X \times \bA^1_k = X^{(0)} \to X^{(1)} \hookrightarrow \dots \hookrightarrow X^{(n)} \hookrightarrow \dots \to Y \times \bA^1_k, 
            \]
        such that the colimit of derived $k$-analytic stacks identifies naturally
            \[
                \colim_{n \ge 0} X^{(n)} \simeq \cD_{X/ Y}\an.  
            \]
        In particular, the latter induces the usual Hodge filtration on (global sections of) $\bT\an_{X/ Y}[1]^\wedge$ and it induces a further filtration
        on (global sections of) the formal completion $Y^\wedge_X$, which we shall also refer to as the \emph{Hodge filtration} on the formal completion $Y^\wedge_X$;
        \item When $f \colon X \to Y$ is a locally complete intersection morphism of derived $k$-affinoid spaces, the Hodge filtration on $Y^\wedge_X$ identifies with the $I$-adic filtration on global sections.
        More precisely, let
            \[A \coloneqq \Gamma(X, \cO_X\alg) \quad \mathrm{and} \quad B \coloneqq \Gamma(Y, \cO_Y\alg),\]
        denote the corresponding derived global sections derived $k$-algebras. Then the global sections of
        $Y^\wedge_X$ agree with the completion
            \[
                B^\wedge_I \in \CAlg_k, \quad I \coloneqq \fib(B \to A),  
            \]
        and the Hodge filtration on $B^\wedge_I$ coincides with the usual $I$-adic filtration.
    \end{enumerate}
\end{thm}

For applications, it is desirable to establish \cref{intro-thm2} in such generality. As an example, it is interesting to
have at our disposal a natural Hodge filtration for the de
Rham cohomology of the moduli of $\ell$-adic continuous representations of an \'etale fundamental group, see
\cite{antonio2017moduli} and \cite{antonio2019moduli}. The latter is a locally geometric derived $k$-analytic stack.
Notwithstanding, the reader can safely assume that $f$ is simply a morphism of (derived) $k$-analytic spaces.

In order to construct the deformation to the normal cone $\cD_{X/ Y}\an$ we shall need two main ingredients. Namely, a general theory of $k$-analytic
formal moduli problems and the relative analytification functor introduced in \cite{Holstein_Analytification_of_mapping_stacks} (see also \cite[\S 6]{Porta_Yu_NQK}).
So far, a general theory of analytic formal moduli problems was lacking in the literature. For this reason, we devote the whole \S 2
to develop such framework. Let $X \in \dAnk$ be a (derived) $k$-analytic space. Denote by $\anNil_{X/ }^\cl$ the \infcat of closed embeddings
    \[
        g\colon X \to S,  
    \]
where $S \in \dAnk$, such that $g_\red \colon X_\red \to S_\red$ is an isomorphism of $k$-analytic spaces. We shall define a $k$-analytic formal moduli problem
under $X$ (resp. over $X$) as a functor
    \[
        Y \colon (\anNil^{\mathrm{cl}}_{X/ })^ \mathrm{op} \to \cS,  \quad \textrm{(resp. $Y \colon (\anNil^\cl_{/ X}) \op \to \cS$)}
    \]
satisfying some additional requirements, see \cref{defin:analytic_formal_moduli_problems_under} (resp. \cref{defin:analytic_formal_problems_over_X}).
It turns out that analytic formal moduli problems can be described in very explicit terms as follows:
\begin{thm} \label{intro-thm3}
    Let $X \in \dAn_k$ denote a derived $k$-analytic space. Then the following assertions hold:
    \begin{enumerate}
        \item The \infcat $\anNil_{X/ /Y}^\cl$ (resp. $(\anNil_{/ X}^\cl)_{/ Y}$) is filtered and we have a natural equivalence
            \[
                \colim_{S \in \anNil^\cl_{X/ / Y}} S \simeq Y \quad \textit{resp.} \quad \colim_{X \in (\anNil_{/ X}^\cl)_{/ Y}} S \simeq Y; 
            \]
        \item The natural morphism $f \colon X \to Y$ (resp. $f \colon Y \to X$)
        admits a relative (pro-)cotangent complex $\bbL_{X/ Y} \in \pro(\Coh^+(X))$, which completely controls the deformation
        theory of the morphism $f$;
        \item In the case where $X$ is a derived $k$-affinoid space, the \infcat of analytific formal moduli problems $\anFMP_{X/ }$ is presentable and
        monadic over the usual \infcat of ind-coherent sheaves on 
            \[A \coloneqq \Gamma(X, \cO_X\alg).\]
    \end{enumerate}
\end{thm}

We refer the reader to \cref{cor:formal_moduli_problems_over_X_are_ind_inf_objects}, \cref{prop:analytic_FMP_under_X_are_ind_inf_schemes},
\cref{cor:universal_property_of_relative_cotangent_complex_for_morphisms_between_analytic_FMP} and \cref{prop:conservativity_and_preservation_of_sifted_colimits_of_tangent_complex},
for a detailed treatment of \cref{intro-thm3}. 

Granting a general deformation theory we are able to construct the deformation $\cD_{X/ Y}\an$
as a $k$-analytic formal moduli problem
under $X$. Furthermore, we are able to extrapolate the main formal properties of the derived deformation to the normal cone in the algebraic setting, proved in \cite[\S 9]{Gaitsgory_Study_II}.
The latter is done via a careful analysis of the behaviour of the relative analytification functor, denoted $(-)\an_Y$, and by means of
Noetherian approximation.

In particular, when $f \colon X \to Y$ is a morphism of derived $k$-affinoid spaces, we have a natural identification
    \[
        \cD_{X/ Y}\an \simeq (\cD_{X\alg/ Y\alg} )\an_Y,  
    \]
where $X\alg \coloneqq \Spec A$ and $Y\alg \coloneqq \Spec B$, with $A$ and $B$ are as in \cref{intro-thm2}. 

Along the way, we establish a precise comparison between the works of \cite{Bhatt_Derived_Completions} and \cite[\S 9]{Gaitsgory_Study_II},
which we consider to be of relevance for experts in derived algebraic geometry. Moreover, we shall also explicit mention that the techniques employed in the current
text, namely Noetherian approximation, allows to construct the deformation to the normal cone and its Hodge filtration
as in \cite[\S 9]{Gaitsgory_Study_II} outside the the almost of finite presentation scope, provided that we work under Noetherian assumptions. The latter
might be helpful even in purely algebraic contexts.

\subsection{Relation with other works and future questions} As previously explained, the content of this text is build upon two major prior works,
namely \cite[\S 9]{Gaitsgory_Study_II} and \cite{Bhatt_Derived_Completions}. It follows essentially from \cref{lem:identification_of_Hodge_and_adic_filtrations}
that
Bhatt's construction of the Hodge filtration on algebraic de Rham cohomology with the approach of Gaitsgory-Rozemblyum (c.f. \cite[\S 9]{Gaitsgory_Study_II}).
Such comparison result is not completely obvious. Indeed, in \cite{Bhatt_Derived_Completions} the author produces the Hodge filtration in the general case, by left Kan extending the adic
filtration from the case of local complete intersection morphisms, whereas the construction of \cite[\S 9]{Gaitsgory_Study_II} is completely canonical. Therefore,
there are two main advantages provided by our comparison result: first it proves the canonicity of Bhatt's approach and secondly it enlightens many
important aspects of \cite[\S 9]{Gaitsgory_Study_II}, which were not made completely explicitly.

In \cite{grosse2002finiteness}, Grosse-Kl\"onne introduces an \emph{analytic} de Rham cohomology
theory for a very large class of rigid $k$-analytic spaces. Let $X$ be a rigid $k$-analytic space,
Grosse-Kl\"onne's construction is related to \emph{over-convergent} global sections of formal completions
along closed immersion $i \colon X \hookrightarrow Y$, where $Y$ is a smooth rigid $k$-analytic space.
Furthermore, the author proves finiteness of over-convergent analytic de Rham cohomology under certain finiteness conditions on $X$.

We expect that our derived de Rham cohomology theory, $\dR_{X/k}\an$, can be compared to the construction in \cite{grosse2002finiteness}. In particular,
this would provide a proof of finiteness for derived (over-convergent) analytic de Rham cohomology.
Such comparison result would equip us with a Hodge-to-de Rham
spectral sequence for over-convergent de Rham cohomology, in great generality. 
In order to achieve such a comparison one needs to extend
the results in this text to the over-convergent setting. We hope to pursue this venue in a future project.

Ideally, the deformation to the normal bundle would provide an important tool to study the relationship between de Rham cohomology and \'etale cohomology
of $k$-analytic spaces, similar to the algebraic case (see \cite[\S 1, Corollary]{Bhatt_Derived_Completions}). When $k$ is a field of equi-characteristic zero, we expect that a non-smooth extension of the main construction of \cite{achinger2019betti}, in the non-archimedean setting, could be related to the formal
completion $Y^\wedge_X$. More interestingly, would be to study the $p$-adic setting. Even though the author currently does not have an idea of how to accomplish
this, it might be useful to extrapolate certain results in \cite{bras2018overconvergent} to the singular case, via derived geometry.

Another future goal closely related to this work is the study of a rigid cohomology theory for singular algebraic varieties and possibly algebraic stacks over perfect fields of
positive characteristic. Our hope is that techniques from derived geometry can be useful to reduce questions about the rigid cohomology of singular algebraic varieties to the free case, via simplicial polynomial resolutions
of general $k$-algebras.

Finally, we expect that the results in this text provide a first step in a consistent study of $\cD_X$-modules and curved modules outside the smooth case, in the non-archimedean setting.
In order to develop such a framework, one has to set the main ingredients of \cite{ben2012loop} in non-archimedean geometry. One such ingredient is the deformation to the normal cone done
in the current text. Another important ingredient is a $k$-analytic HKR theorem, which is a work in progress with M. Porta and F. Petit (see the author's thesis \cite{ferreira2019modeles}
for a sketch of proof of the latter statement).

\subsection{Organization of the paper}
In \S 2, we develop a theory of formal moduli problems in the non-archimedean setting. We devote 2.1 to a brief review of the main notions in non-archimedean
geometry that will be important for us. We then proceed to \S 2.2 where a careful study of the structure of \emph{nil-isomorphisms} is performed.
The notion of nil-isomorphism is basic building block of the theory of analytic formal moduli problems \emph{under} (resp. \emph{over}) a given base.
The latter notions are introduced and study at length in \S 2.3 (resp. \S 2.4). In \S 2.4, we establish \emph{pseudo}-nil-descent for pro-almost perfect complexes
or in other words pro-coherent complexes. Given a morphism $X \to Y$ which exhibits $X$ as an analytic formal moduli problem, we prove that
the \infcat of $\pro(\Coh^+(Y))$ does satisfy descent along the usual $*$-upper functoriality, that it is, it can be realized as the totalization of
    \[\pro(\Coh^+(X^\bullet)),\]
where $X^\bullet$ denotes the \v{C}ech nerve associated to $X \to Y$. The main novelty compared with the main result in \cite[\S 7.2]{Gaitsgory_Study_I}
is that we use the entire $\Coh^+$ and the $*$-upper functoriality. Our nil-descent statement is based in the work of Halpern-Leistner and Preygel, c.f.
see \cite[Theorem 3.3.1]{preygel_Leistner_Mapping_stacks_properness}. Section 2.6 is devoted to the study of the notion of formal groupoids over a base $X$.
We will prove the existence of a natural equivalence between the
\infcats of analytic formal moduli problems under $X$ and analytic formal groupoids over $X$ (see \cref{thm:Phi_is_an_equivalence}).
The content of this equivalence can be interpreted as: every analytic formal moduli problem under $X$ arises as a \emph{classifying formal derived $k$-analytic stack}
associated to an analytic formal groupoid (which corresponds to the associated \v{C}ech nerve).
This equivalence of \infcats is based on the pseudo-nil-descent results, proved earlier. Notice that \cref{thm:Phi_is_an_equivalence},
is a direct non-archimedean analogue
of \cite[\S 5, Theorem 2.3.2]{Gaitsgory_Study_II}. In \S 2.6, we study the tangent complex associate to an analytic formal moduli problem
under $X$. In particular, we prove that such a functor is monadic, see \cref{prop:conservativity_and_preservation_of_sifted_colimits_of_tangent_complex}.

The main bulk of the text lies in \S 3, where we introduce the deformation to the normal cone, $\cD_{X/ Y}$, in the non-archimedean setting. In
\S 3.1 we review the algebraic situation. We shall summarize the main algebraic formal properties of $\cD_{X/ Y}$ and then proceed to
give a simplification of the formalism developped in \cite[\S 9]{Gaitsgory_Study_II}
in the case of closed immersions of derived schemes, see \cref{lem:identification_of_Hodge_and_adic_filtrations}. This will enable us to prove a comparison statement between Bhatt's construction of the Hodge filtration in algebraic de Rham
cohomology, c.f. \cite{Bhatt_Derived_Completions} and the work of Gaitsgory-Rozemblyum.

In \S 3.2 we introduce the non-archimedean deformation in the local case, via the relative analyticification functor and Noetherian approximation. We
devote \S 3.3 to gluing the deformation to the class of morphisms between locally geometric derived $k$-analytic stacks. In particular, our results hold for any derived $k$-analytic space.
We then introduce the construction of the \emph{non-archimedean Hodge filtration} in the local, in \S 3.4. This is possible by extending the usual Hodge filtration of
Gaitsgory-Rozemblyum to the non-archimedean setting as follows: we first extrapolate the main results in \cite[\S 9.5]{Gaitsgory_Study_II} to the Noetherian setting
not necessarily of almost of finite presentation (which was the only case treated in \cite[\S 9]{Gaitsgory_Study_II}) and then to the non-archimedean setting via the relative analytification functor.
Finally, in \S 3.4 we glue the non-archimedean Hodge filtration for general morphisms between locally geometric derived $k$-analytic stacks.

\subsection{Notations and Conventions} In this text, $k$ denotes a non-archimedean field of characteristic zero. We let $\dAff_k$ and $\dAff_k^\laft$ denote
the \infcats of affine derived schemes (resp. affine derived schemes almost of finite presentation).
We shall denote by $\dSt_k$ the \infcat of derived $k$-stacks. 
We shall further denote by $\dSt^\laft_k \subseteq \dSt_k$ the full subcategory spanned by derived $k$-stacks locally almost of finite presentation.
Given derived $k$-stacks $X, Y \in \dSt_k$ we shall denote by
    \[
        \Map(X, Y) \in \dSt_k,    
    \]
the corresponding mapping stack. If $X$ and $Y$ live over $\bbA_k^1$, we then shall denote by
    \[
        \Map_{/ \bbA^1_k}(X, Y) \in (\dSt_k)_{/ \bbA^1_k}.  
    \]
We denote $\dAnSt_k$ the \infcat of derived $k$-analytic stacks introduced in \cite{Porta_Yu_Derived_non-archimedean_analytic_spaces}.
We shall denote by $\dAnk$ and $\dAfd_k$ the \infcats of derived $k$-analytic spaces (resp., derived $k$-affinoid spaces). We will
further denote by $\Ank$ the usual ordinary category of (discrete) $k$-analytic spaces as in \cite{Berkovich_Etale_1993}.
We shall instead use bold letters for the corresponding notions in the non-archimedean setting. Therefore, we shall denote by
    \[
        \bfMap(X, Y) \in \dAnSt_k,
    \]
the corresponding $k$-analytic mapping stack, whenever $X, Y \in \dAnSt_k$.
Let $\cC$ be a stable \infcat. We shall denote by $\cC^{\mathrm{fil}}$ the associated \infcat of \emph{filtered objects in $\cC$}.
Similarly, we denote by $\cC^{\mathrm{gr}}$ the \infcat of graded objects in $\cC$. We shall employ the notations
    \[
        \cC^{\mathrm{gr}, \ge 0} \quad \mathrm{and} \quad \cC^{\mathrm{gr} , =n}, 
    \]
the respective full subcategories of positively indexed graded objects in $\cC$ and objects concentrated in a single degree $n$. Notice that the latter
\infcat is equivalent to $\cC$ itself.

\subsection{Acknowledgments} The author would like to express his deep gratitude to Mauro Porta for the encouragement to pursue this project and
for many commentaries and suggestions during the realization of the present work. The author would also like to thank Marco Robalo for a clarifying
remark.

\section{Analytic formal moduli problems}

\subsection{Preliminaries}
Recall the notions of $k$-affinoid/analytic spaces introduced
in \cite[Definition 7.3 and Definition 2.5.]{Porta_Yu_Derived_non-archimedean_analytic_spaces}, respectively.


\begin{defin}
    Let $f \colon X \to Y$ be a morphism in the \infcat $\dAnk$. We shall say that $f$ is an \emph{affine morphism} if
    for every morphism $Z \to Y$ in $\dAnk$ such that $Z$ is a derived $k$-affinoid space, the pullback
        \[
            Z' \coloneqq Z \times_X Y \in \dAnk,  
        \]
    is a derived $k$-affinoid space.
\end{defin}

\begin{defin}
    Let $f \colon X \to Y$ be a morphism in $\dAnk$. We shall say that $f$ is an \emph{admissible open immersion} if the induced morphism on $0$-th truncations
        \[
            \trunc(f) \colon \trunc(X) \to \trunc(Y),  
        \]
    is an admissible open immersion in the sense of \cite[\S 1.3]{Berkovich_Etale_1993}.
\end{defin}

\begin{defin}
    Let $f \colon X \to Y$ be a morphism in $\Ank$. We shall say that $f$ is a \emph{finite morphism} if $f$ is affine and for every admissible open covering
        \[
            \coprod_{j \in J} V_j \to Y,  
        \]
    the base change morphism
        \[
            U_j \coloneqq V_j \times_Y Y \to V_j,  
        \]
    exihibits the $k$-affinoid algebra $B_j \coloneqq \Gamma(U_j, \cO_{U_j})$ as a finite $A = \Gamma(V_j, \cO_{V_j})$-module.
    Given $f \colon X \to Y$ in the \infcat $\dAnk$, we say that $f$ is \emph{finite} if its truncation $\trunc(f)$ is a finite
    morphism.
\end{defin}

\begin{defin} Let $X \in \Ank$ denote an (ordinary) $k$-analytic space.
    We denote by $\cJ_X \subseteq \cO_X$, the \emph{nil-radical ideal} of $\cO_X$, that is the sheaf ideal spanned by nilpotent
    sections on $X$. We denote by $X_\red$ the $k$-analytic space obtained as the pair $(X, \cO_{X}/ \cJ_X)$, which we shall refer to as the \emph{underlying reduced $k$-analytic space
    associated to $X$}.       
\end{defin}

\begin{rem}
    Let $X \in \Ank$, then the underlying reduced $X_\red \in \Ank$ is a reduced $k$-analytic space, by construction.
\end{rem}

\begin{notation}
    We denote by $\Ank^\mathrm{red} \subseteq \Ank$ the full subcategory spanned by reduced $k$-analytic spaces. We further denote by
    \[(-)_\red \colon \dAnk \to \Ank^\red,\]
    the functor obtained by the association
        \[
            Z \in \Ank \mapsto Z_\red \in \Ank^\mathrm{red}.
        \]
\end{notation}

We now extend the construction $(-)_\red \colon \Ank \to \Ank^\mathrm{red}$ to the \infcat of derived $k$-analytic spaces:

\begin{defin}
    Consider the functor $\dAnk \to \Ank^\mathrm{red}$ defined as the composite
        \[
            (-)_\red \colon \dAnk \xrightarrow{\trunc(-)} \Ank \xrightarrow{(-)_\red} \Ank^\mathrm{red}.
        \]
    Given $X \in \dAnk$ a derived $k$-analytic space we denote by $X_\red \in \Ank^\red$ the \emph{underlying reduced $k$-analytic space}
    associated with $X$. 
\end{defin}

\begin{lem}
    Let $f \colon X \to Y$ be an admissible open immersion of derived $k$-analytic spaces. Then $f_\red \colon X_\red \to Y_\red$ is
    an admissible open immersion, as well.
\end{lem}

\begin{proof}
    By the definitions, it is clear that the truncation
        \[
            \trunc(f) \colon \trunc(X) \to \trunc(Y),  
        \]
    is an admissible open immersion of ordinary $k$-analytic spaces. In the case of ordinary $k$-analytic spaces the question is local and we reduce ourselves to the affinoid case.
    Then it is clear that localization in the ordinary category $\Afd_k$ commutes with quotients (as the former is defined via a colimit).
\end{proof}

\begin{defin}
    In \cite[Definition 5.41]{Porta_Yu_Representability} the authors introduced the notion of a square-zero extension between derived
    $k$-analytic spaces. In particular, given a morphism $f \colon Z \to Z'$ in $\dAnk$, we shall say that $f$ \emph{has the structure of
    a square-zero extension} if $f$ exhibits $Z'$ as a square-zero extension of $Z$.
\end{defin}

\begin{rem} \label{rem:construction_of_nilpotent_extensions_as_square_zero_extensions}
    Let $X \in \Ank$. Let $\cJ \subseteq \cO_X$ be an ideal satisfying $\cJ^2 = 0$. Consider the fiber sequence
        \[
            \cJ \to \cO_X \to \cO_X/\cJ,  
        \]
    in the \infcat $\Coh^+(X)$. We have an induced natural fiber sequence of the form
        \[
            \bbL_{\cO_X}\an \to \bbL_{\cO_X / \cJ}\an \to \bbL_{(\cO_X / \cJ) / \cO_X}\an,
        \]
    and we have a further identification $\tau_{\le 1} (\bbL_{(\cO_X/ \cJ) / \cO_X}) \simeq \cJ[1]$. For this reason, we obtain a well defined morphism
        \[
            d \colon \bbL_{\cO_X/ \cJ} \to \cJ[1],
        \]
    in the derived \infcat $\Mod_{\cO_X/ \cJ}$. This derivation classifies a square-zero extension of $\cO_X/ \cJ$ by $\cJ[1]$ which can be identified with the
    object $\cO_X$ itself. In particular, we deduce that $X$ is a square-zero extension of the ordinary $k$-analytic space $(X, \cO_{X}/\cJ)$.
\end{rem}

\begin{rem}
    Recall the \infcat $\RTop(\cTank)$ of $\cTank$-structured \inftopoi defined in
    \cite[Definition 2.4]{Porta_Yu_Derived_non-archimedean_analytic_spaces}, where $\cTank$ denotes the $k$-analytic pre-geometry (see for instance \cite[Construction 2.2]{Porta_Yu_Derived_non-archimedean_analytic_spaces}).
    Let  $\cO \in \Strloc_{\cTank}(\cX)$ be a local $\cTank$-structure on $\cX$ (c.f. \cite[Definition 2.4]{Porta_Yu_Derived_non-archimedean_analytic_spaces}). Since the pregeometry $\cTank$ is compatible
    with $n$-truncations, c.f. \cite[Theorem 3.23]{Porta_Yu_Derived_non-archimedean_analytic_spaces}, it follows that
    $\pi_0(\cO) \in \Strloc_{\cTank}(\cX)$, as well.
    
    Moreover, if $(\cX, \cO) \in \RTop(\cTank)$ is a $\cTank$-structured \inftopos, we define its \emph{underlying reduced $\cTank$-structured \inftopos} as the $\cTank$-structured space
        \[  
            (\cX, \pi_0(\cO)/ \cJ) \in \RTop(\cTank),  
        \]
    where $\cJ \subseteq \pi_0(\cO)$ denotes the ideal sheaf spanned by nilpotent sections on $\pi_0(\cO)$. Moreover, the quotient $\pi_0(\cO)/ \cJ$
    is considered in the \infcat of local structures on $\cX$
        \[
            \pi_0(\cO) / \cJ \in \Strloc_{\cTank}(\cX).
        \]
\end{rem}

Recall the notion of the underlying algebra functor $(-)\alg \colon \Strloc_{\cTank}(\cX) \to \CAlg_k(\cX)$ introduced in
\cite[Lemma 3.13]{Porta_Yu_Derived_non-archimedean_analytic_spaces}.

\begin{lem} \label{lem:derived_k_analytic_space_whose_reduction_is_affinoid_is_also_affinoid}
    Let $Z \coloneqq (\cZ, \cO_Z) \in \RTop(\cTank)$ denote a $\cTank$-structure \inftopos such that $\pi_0(\cO_Z\alg) $ is a Noetherian derived $k$-algebra on $\cZ$.
    Suppose that the reduction
    $Z_\red$ is equivalent to a derived $k$-affinoid space. Then the truncation $\trunc(Z)$ is isomorphic to an ordinary $k$-affinoid space.
    If we assume further that for every $i>0$, the homotopy sheaves $\pi_i(\cO_Z)$ are
    coherent $\pi_0(\cO_Z)$-modules, then $Z$ itself is equivalent to a derived $k$-affinoid space.
\end{lem}

\begin{proof} The second claim of the Lemma follows readily from the first assertion together with the definitions.
    We are thus reduced to prove that $\trunc(Z)$ is isomorphic to an ordinary $k$-affinoid space.
    Let $\cJ \subseteq \pi_0(\cO_Z)$, denote the coherent ideal sheaf associated to the closed immersion $Z_\red \hookrightarrow Z$. Notice that the ideal $\cJ$
    agrees with the nilradical ideal of $\pi_0(\cO_Z)$. Thanks to our assumption that $\pi_0(\cO_Z)$ is a Noetherian derived $k$-algebra on $\cZ$, it follows that there exists
    a sufficiently large integer $n \ge 2$ such that
        \[
            \cJ^n = 0.  
        \]
    Arguing by induction, we can suppose that $n = 2$, that is to say that
        \[\cJ^2 = 0.\]
    In particular, \cref{rem:construction_of_nilpotent_extensions_as_square_zero_extensions} implies that the
    the natural morphism $Z_\red \to Z$ has the structure of a square-zero extension.
    The assertion now follows from \cite[Proposition 6.1]{Porta_Yu_Representability}
    and its proof.
\end{proof}

\begin{rem}
    We observe that the converse of \cref{lem:derived_k_analytic_space_whose_reduction_is_affinoid_is_also_affinoid} holds true.
    Indeed, the natural morphism $Z_\red \to Z$ is a closed immersion. In particular, if $Z \in \dAfd_k$ we deduce readily
    that $Z_\red \in \Afd_k$, as well.
\end{rem}

\begin{lem} \label{lem:affine_morphisms_are_compatible_with_Zariski_localization_on_the_base}
    Let $f \colon X \to Y$ be an affine morphism in $\dAn_k$. Suppose we are given an admissible open covering
        \[
            g \colon \coprod_{j \in J} U_j \to Y,  
        \]
    where for each $j \in J$, $U_j \in \dAfd_k$. For each $j \in J$, let 
        \[
            V_j \coloneqq U_j \times_X Y \in \dAfd_k,  
        \]
    then $\coprod_{j \in J} V_j \to Y$ is an admissible open covering by derived $k$-affinoid spaces.
\end{lem}

\begin{proof} It is clear from our assumption that $f$ is an affine morphism that for every index $j \in J$, the objects $V_j \in \dAfd_k$. The claim of the Lemma
    follows immediately from the observation that
    both the classes of effective epimorphisms of \inftopoi \ and admissible open immersions of derived $k$-analytic spaces are stable under pullbacks,
    cf. \cite[Proposition 6.2.3.15]{HTT} and \cite[Corollary 5.11, Proposition 5.12]{Porta_Yu_Representability}, respectively.
\end{proof}

\subsection{Non-archimedean differential geometry} In this \S, we introduce the notion of a nil-isomorphism between derived
$k$-analytic spaces and study certain important features of such class of morphisms. The results contained in this paragraph will be of
crucial importance to the study of analytic formal moduli problems.

\begin{defin}
    Let $f \colon X \to Y$ be a morphism in $\dAnk$. We say that $f$ is a \emph{nil-isomorphism} if $f$ is almost of finite presentation
    and furthermore
    $f_\red \colon X_\red \to
    Y_\red$ is an isomorphism of ordinary $k$-analytic spaces. 
    We will denote by $\anNil_{X/ }$ the full subcategory of $(\dAnk)_{X/}$
    spanned by nil-isomorphisms $X \to Y$.
\end{defin}

\begin{lem} \label{lem:nil-isos_are_affine_morphisms}
    Let $f \colon X \to Y$ be a nil-isomorphism of derived $k$-analytic spaces. Then:
    \begin{enumerate}
        \item Given any morphism $Z \to Y$ in $\dAnk$, the base change morphism
            \[
                Z \times_X Y \to Z,  
            \]
        is an nil-isomoprhism, as well;
        \item $f$ is an affine morphism;
        \item $f$ is a finite morphism.
    \end{enumerate}
\end{lem}

\begin{proof} To prove (1), it suffices to prove that
    the functor $(\textrm{-})_\red \colon \dAnk \to \Ank^{\red}$ commutes with finite limits. The truncation functor
        \[
            \trunc \colon \dAnk \to \Ank,  
        \]
    commutes with finite limits, c.f. \cite[Proposition 6.2 (5)]{Porta_Yu_Derived_non-archimedean_analytic_spaces}. It suffices then to prove that the usual underlying reduced functor
        \[
            (-)_\red \colon \Ank \to \Ank^\red,
        \]
    commutes with finite limits. By construction,
    the latter assertion is equivalent to the claim that
    the usual complete tensor product of ordinary $k$-affinoid algebras commutes with the operation of taking the quotient by the Jacobson radical, which is immediate.

    We now prove (2). Let $Z \to Y$ be an admissible open immersion such that $Z$ is a derived $k$-affinoid space. We claim that the pullback
    $Z \times_X Y$ is again a derived $k$-affinoid space. Thanks to \cref{lem:derived_k_analytic_space_whose_reduction_is_affinoid_is_also_affinoid}
    we reduce ourselves to prove that $(Z \times_X Y)_\red$ is equivalent to an
    ordinary $k$-affinoid space. Thanks to (1), we deduce that the induced morphism
        \[
            (Z \times_X Y)_\red \to Z_\red,  
        \]
    is an isomorphism of ordinary $k$-analytic spaces. In particular, $(Z \times_X Y)_\red$ is a $k$-affinoid space. The result now follows by invoking
    \cref{lem:derived_k_analytic_space_whose_reduction_is_affinoid_is_also_affinoid} again.

    To prove (3), we shall show that the induced morphism on the underlying $0$-th truncations $\trunc(X) \to \trunc(Y)$ is a finite morphism of ordinary $k$-affinoid spaces.
    But this follows immediately from the fact that both $\trunc(X)$ and $\trunc(Y)$ can be obtained from the reduced $X_\red$ by means of a finite sequence of square-zero extensions
    as in \cref{rem:construction_of_nilpotent_extensions_as_square_zero_extensions}.
\end{proof}

\begin{defin}
    A morphism $X \to Y$ be a morphism in $\dAnk$ is called a \emph{nil-embedding} if the induced morphism of ordinary $k$-analytic spaces
    $\trunc(X) \to \trunc(Y)$ is a closed immersion with nilpotent ideal of definition. 
\end{defin}

\begin{prop} \label{prop:filtered_colimit_for_nil-embeddings}
    Let $f \colon X \to Y$ be a nil-embedding of derived $k$-analytic spaces. Then there exists a sequence of morphisms
        \[X = X_0^0 \hookrightarrow X_0^1 \hookrightarrow \dots \hookrightarrow X_0^n = X_0 
        \hookrightarrow X_1 \dots X_n \hookrightarrow \dots \hookrightarrow Y,\]
    such that for each $0 \le i \le n$ and $j \ge 0$, the morphisms $X_0^i \hookrightarrow X_0^{i+1}$ and $X_j \to X_{j+1}$ have the structure of square-zero extensions.
    Furthermore, the induced morphisms $\mathrm{t}_{\le j}(X_j) \to \mathrm{t}_{\le j}(Y)$ are equivalences of derived
    $k$-analytic spaces, for every $j \ge 0$.
\end{prop}

\begin{proof}
    The proof follows the same scheme of reasoning as \cite[Proposition 5.5.3]{Gaitsgory_Study_II}. For the sake of completeness we present the complete argument here.
    Consider the induced morphism on the underlying truncations
        \[
            \trunc(f) \colon \trunc(X) \to \trunc(Y).   
        \]
    By construction, there exists a sufficiently large integer $n \ge 0$ such that
        \[
            \cJ^{n+1} = 0,  
        \]
    where $\cJ \subseteq \pi_0(\cO_Y)$ denotes the ideal associated to the nil-embedding $\trunc(f)$.
    Therefore, we can factor the latter as a finite sequence of square-zero extensions of ordinary $k$-analytic spaces
        \[
            \trunc(X) \hookrightarrow X_0^{\mathrm{ord}, 0} \hookrightarrow \dots \hookrightarrow X^{\mathrm{ord}, n}_0 = \trunc(Y),
        \]
    as in the proof of \cref{lem:derived_k_analytic_space_whose_reduction_is_affinoid_is_also_affinoid}. For each $0 \le i \le n$, we set
        \[
            X_0^i \coloneqq X \bigsqcup_{\trunc(X)} X_0^{\mathrm{ord}, i}.
        \]
    By construction, we have that the natural morphism $\trunc(X_0^n) \to \trunc(Y)$ is an isomorphism of ordinary $k$-analytic spaces.
    We now argue by induction on the Postnikov tower associated to the morphism $f \colon X \to Y$.
    Suppose that for a certain integer $i \ge 0$, we have constructed a derived $k$-analytic space $X_i$ together with morphisms $g_i \colon
    X \to X_i$ and $h_i \colon X_i \to Y$ such that $f \simeq h_i \circ g_i$
    and the induced morphism
        \[
            \mathrm{t}_{\le i}(X_i) \to \mathrm{t}_{\le i}(Y)
        \]
    is an equivalence of derived $k$-analytic spaces. We shall proceed as follows: by the assumption that $h_i$ is $(i+1)$-connective, we deduce from
    \cite[Proposition 5.34]{Porta_Yu_Representability} the existence of a natural equivalence
        \[
            \tau_{\le i}(\bbL_{X_i/Y}\an) \simeq 0,
        \]
    in $\Mod_{\cO_{X_i}}$. Consider the natural fiber sequence
        \[
            h_i^* \bbL\an_{Y} \to \bbL\an_{X_i} \to \bbL\an_{X_i/Y},
        \]
    in $\Mod_{\cO_{X_i}}$. The natural morphism
        \[
            \bbL\an_{X_i/ Y} \to \pi_{i+1}(\bbL\an_{X_i/Y})[i+1],  
        \]
    induces a morphism $\bbL\an_{X_i} \to \pi_{i+1}(\bbL\an_{X_i/Y})[i+1]$, such that the composite
        \begin{equation} \label{eq:fiber_sequence_of_cotangent_complexes_to_produce_the_existence_of_the_desired_square_zero_extension_approximating_Y_in_degree_i+1}
            h_i^* \bbL\an_{Y} \to \bbL\an_{X_i} \to \pi_{i+1}(\bbL\an_{X_i/Y})[i+1],  
        \end{equation}
    is null-homotopic, in $\Mod_{\cO_{X_i}}$. By the universal property of the relative analytic cotangent complex,
    \eqref{eq:fiber_sequence_of_cotangent_complexes_to_produce_the_existence_of_the_desired_square_zero_extension_approximating_Y_in_degree_i+1}
    produces a square-zero extension
        \[
            X_i  \to X_{i+1},
        \]
    together with a morphism $h_{i+1} \colon X_{i+1} \to Y$, factoring $h_i \colon X_i \to Y$. We  are reduced to show that the morphism
        \[
            \cO_Y \to h_{i+1, *}(\cO_{X_{i+1}}), 
        \]
    is $(i+2)$-connective. Consider the commutative diagram
        \begin{equation} \label{eq:commutative_diagram_of_fiber_sequences_exhibiting_O_X_i+1_as_an_approximation_of_level_i+1_of_Y}
        \begin{tikzcd}
            h_{i, *}(\pi_{i+1}(\bbL\an_{X_i/Y}))[i+1] \ar{r} & h_{i+1, *}(\cO_{X_{i+1}}) \ar{r} & h_{i, *}(\cO_{X_i}) \\
            \cI \ar{r} \ar{u}{s_i} & \cO_Y \ar{r} \ar{u} & h_*(\cO_{X_i}) \ar{u} \\
            \cJ \ar{r}{=} \ar{u} & \cJ \ar{r} \ar{u} & 0 \ar{u}
        \end{tikzcd},
        \end{equation}
    in $\mathrm{Mod}_{\cO_Y}$, where both the vertical and horizontal composites are fiber sequences. By our inductive hypothesis, $\cI$ is $(i+1)$-connective.
    Moreover, thanks to \cite[Proposition 5.34]{Porta_Yu_Representability} we can identify the natural
    morphism    
        \[
           s_i \colon \cI \to h_{i, *}(\pi_{i+1}(\bbL\an_{X_i/Y}))[i + 1]
        \]
    with the natural morphism $\cI \to \tau_{\le i + 1}(\cI)$. We deduce that the fiber of the morphism $s_i$ must be necessarily $(i+2)$-connective. The latter observation
    combined with the structure of \eqref{eq:commutative_diagram_of_fiber_sequences_exhibiting_O_X_i+1_as_an_approximation_of_level_i+1_of_Y}
    implies that $h_{i+1} \colon X_{i+1} \to Y$ induces an equivalence of derived $k$-analytic spaces
        \[
            \rt_{\le i+1}(X_{i+1}) \to \rt_{\le i+1}(Y),  
        \]
    as desired.
\end{proof}

\begin{cor}
    Let $X \in \dAnk$. Then the following assertions hold:
    \begin{enumerate}
        \item The natural morphism
            \[
                X_\red \to X,  
            \]
         in $\dAnk$, can be \emph{approximated} by successive square-zero extensions;
        \item For each $n \ge 0$, the natural morphism
            \[
                X_\red \to \rt_{\le n}(X),  
            \]
        can be \emph{approximated} by a finite number of square-zero extensions.
    \end{enumerate}
\end{cor}

\begin{proof}
    Both the assertions of the Corollary follow readily from \cref{prop:filtered_colimit_for_nil-embeddings} by observing that the
    canonical morphisms $X_\red \to X$ and $X_\red \to \rt_{\le n}(X)$ have the structure of nil-embeddings and that in the latter case the finiteness assumption on the Postkinov
    tower forces the finiteness of the approximation
    sequence.
\end{proof}

\begin{lem} \label{lem:pullbacks_of_derived_affinoid_spaces_along_finite_morphisms_are_algebraic}
    Let $f \colon X \to Y$ be a finite morphism of derived $k$-affinoid spaces. Let $Z \to Y$ be an admissible open immersion and denote by
        \[
            A \coloneqq \Gamma(X, \cO_X\alg), \quad B \coloneqq \Gamma(Y, \cO_Y\alg), \quad C \coloneqq \Gamma(Z, \cO_Z\alg),  
        \]
    the corresponding derived $k$-algebras of derived global sections. Consider the base change
        \[
            Z' \coloneqq Z \times_Y X \in \dAfd_k.
        \]
    Then one has a natural equivalence
        \[
            \Gamma(Z', \cO_{Z'}\alg) \simeq A \otimes_B C,  
        \]
    in the \infcat $\CAlg_k$.
\end{lem}

\begin{proof} 
    The Lemma is a direct consequence of \cite[Proposition 3.17]{Porta_Yu_Derived_non-archimedean_analytic_spaces} (iii).
\end{proof}

\begin{lem} \label{lem:Beck_Chevalley_natural_transformation_equivalence}
    Let $f \colon X \to Y$ be a finite morphism of derived $k$-affinoid spaces and $g \colon Z \to Y$ an admissible open immersion in $\dAfd_k$. Form the pullback diagram
        \[
        \begin{tikzcd}
            Z' \ar{r}{g'}  \ar{d}{f'} & X \ar{d}{f} \\
            Z \ar{r}{g} & Y,
        \end{tikzcd}
        \]
    in the \infcat $\dAfd_k$.
    Then the commutative diagram
        \[
        \begin{tikzcd}
            \Coh^+(Y) \ar{r}{f^*} \ar{d}{g^*} & \Coh^+(X) \ar{d}{(g')^*} \\
            \Coh^+(Z) \ar{r}{(f')^*} & \Coh^+(Z'),
        \end{tikzcd}
        \]
    is right adjointable. In other words, the Beck-Chevalley natural transformation
        \[
            \alpha \colon g^* \circ f_* \to (f')_* \circ g'_* 
        \]
    is an equivalence of functors.
\end{lem}

\begin{proof}
    Since $f$ is assumed to be a finite morphism of derived $k$-affinoid spaces, it follows from the derived Tate aciclicity theorem, c.f. \cite[Theorem 3.1]{Porta_Yu_Derived_Hom_spaces}
    that the right adjoint $f_* \colon \Coh^+(X) \to \Coh^+(Y)$ is well defined. The assertion of the Lemma is now an immediate consequence of \cref{lem:pullbacks_of_derived_affinoid_spaces_along_finite_morphisms_are_algebraic}
    together with \cite[Proposition 2.5.4.5]{Lurie_SAG} and the derived Tate aciclicity theorem.
\end{proof}

\begin{prop} \label{lem:f^*_admits_a_right_adjoint_whenever_f_is_nil-iso}
    Let $f \colon S \to S'$ be a nil-isomorphism between derived $k$-analytic spaces. Then the pullback functor
        \[
            f^* \colon \Coh^+(S') \to \Coh^+(S),  
        \]
    admits a well defined right adjoint, $f_*$.
\end{prop}

\begin{proof}
    Since $f \colon S \to S'$ is a nil-isomorphism, we conclude from \cref{lem:nil-isos_are_affine_morphisms} that $f$ is an affine morphism
    between derived $k$-analytic spaces. By admissible descent of $\Coh^+$, cf. \cite[Theorem 3.7]{Antonio_Porta_Nonarchimedean_Hilbert},
    together with \cref{lem:affine_morphisms_are_compatible_with_Zariski_localization_on_the_base} and \cref{lem:Beck_Chevalley_natural_transformation_equivalence} we reduce the statement of the Lemma to the case
    where both $S$ and $S'$ are equivalent to derived $k$-affinoid spaces. In this case, the result follows by our assumptions on $f$ and \cref{lem:nil-isos_are_affine_morphisms}.
\end{proof}

\begin{prop} \label{prop:existence_of_pushouts_along_closed_nil-isomorphisms}
    Let $f \colon X \to Y$ be a nil-embedding of derived $k$-analytic spaces. Let
    $g \colon X \to Z$ be a finite morphism in $\dAnk$. The the diagram
        \[
        \begin{tikzcd}
            X \ar{r}{f} \ar{d}{g} & Y \\
            Z
        \end{tikzcd}  
        \]
    admits a pushout in $\dAnk$, denoted $Z'$. Moreover, the natural morphism
    $Z \to Z'$ is also a nil-embedding.
\end{prop}

\begin{proof} The \infcat of $\cTank$-structured \inftopoi \ $\RTop(\cTank)$ is a presentable \infcat. Consider the pushout diagram
        \[
        \begin{tikzcd}
            X \ar{r}{f} \ar{d}{g} & Y \ar{d} \\
            Z \ar{r} & Z',
        \end{tikzcd}
        \]
    in the \infcat $\RTop(\cTank)$. By construction, the underlying \inftopos of $Z'$ can be computed as the pushout in the \infcat $\RTop$ of
    the induced diagram on the underlying \inftopoi \ of $X$, $Z$ and $Y$. Moreover, since $g$ is a nil-isomorphism it induces an equivalence on underlying \inftopoi \
    of both $X$ and $Y$. It follows that the induced morphism $Z \to Z'$ in $\RTop(\cTank)$ induces an equivalence on the underlying \inftopoi.
    Moreover, it follows essentially by construction that we have a natural equivalence
        \begin{align*}
            \cO_{Z'}  & \simeq g_*(\cO_Y) \times_{g_*(\cO_X)} \cO_Z \\
                      & \in \Str_{\cTank}^\loc(Z).
        \end{align*}
    As the morphism $g_*(\cO_Y) \to g_*(\cO_X)$ is an effective epimorphism and effective epimorphisms are preserved under fiber products in an \inftopos, c.f. \cite[Proposition 6.2.3.15]{HTT},
    it follows that the natural morphism
        \[
            \cO_{Z'} \to \cO_Z,  
        \]
    is an effective epimorphism, as well.
    Consider now the commutative diagram of fiber sequences
        \[
        \begin{tikzcd}
            \cJ ' \ar{r} \ar{d} & \cO_{Z'} \ar{r} \ar{d}  & \cO_Z \ar{d} \\
            \cJ \ar{r} & g_*(\cO_Y) \ar{r} & g_*(\cO_X),
        \end{tikzcd}
        \]
    in the stable \infcat $\Mod_{\cO_Z'}$. Since the right commutative square is a pullback square it follows that the morphism
        \[
            \cJ' \to \cJ,  
        \]
    is an equivalence. In particular, $\pi_0(\cJ')$ is a finitely generated
    nilpotent ideal of $\pi_0(\cO\alg_{Z'})$. Indeed, finitely generation follows from our assumption that $g$ is a finite morphism.
    Thanks to \cref{lem:derived_k_analytic_space_whose_reduction_is_affinoid_is_also_affinoid},
    it follows that $\trunc(Z')$ is an ordinary $k$-analytic space and the morphism $\trunc(Z') \to \trunc(Z)$ is a nil-embedding. We are thus reduced to show that
    for every $i>0$, the homotopy sheaf $\pi_i(\cO_{Z'}) \in \Coh^+(\trunc(Z'))$. But this follows immediately from the existence of a fiber sequence
        \[
            \cO_{Z'} \to g_*(\cO_Y) \oplus \cO_Z \to g_*(\cO_X),  
        \]
    in the \infcat $\Mod_{\cO_{Z'}}$ together with the fact that $g_*(\cO_Y)$ and $g_*(\cO_Z)$ have coherent homotopy sheaves, by our assumption that $g$ is a
    finite morphism combined with \cref{lem:nil-isos_are_affine_morphisms}.
\end{proof}

\begin{cor} \label{lem:pushouts_of_square_zero_extensions_have_the_structure_of_a_square_zero_extension}
    Let $f \colon S \to S'$ be a square-zero extension and $g \colon S \to T$ a nil-isomorphism in $\dAnk$. Suppose we are given a pushout diagram
        \[
        \begin{tikzcd}
            S \ar{r}{f} \ar{d} & S' \ar{d} \\
            T \ar{r} & T'  
        \end{tikzcd},
        \]
    in $\dAnk$. Then the induced morphism $T \to T'$ is a square-zero extension.
\end{cor}

\begin{proof}
    Since $g$ is a nil-isomorphism of derived $k$-analytic spaces, \cref{lem:f^*_admits_a_right_adjoint_whenever_f_is_nil-iso}
    implies that the pullback functor $g^* \colon \Coh^+(T) \to \Coh^+(S)$ admits a well defined right adjoint
        \[
            g_* \colon \Coh^+(S ) \to \Coh^+(T) .
        \]
    Let $\cF \in \Coh^+(S)^{\ge 0}$ and $d \colon \bbL\an_S \to \cF[1]$ be a derivation
    associated with the morphism $f \colon S \to S'$. Consider now the natural composite
        \[
            d ' \colon \bbL\an_T \to g_* (\bbL\an_S) \xrightarrow{g_*(d)} g_* (\cF)[1],  
        \]
    in the \infcat $\Coh^+(T)$. By the universal property of the relative analytic cotangent complex, we deduce the existence of a square-zero extension
        \[
            T \to T',  
        \]
    in the \infcat $\dAnk$. Let $X \in \dAnk$ together with morphisms $S' \to X$ and $T \to X$ compatible with both $f$ and $g$. By the universal property of
    the relative analytic cotangent complex, the morphism $S' \to X$ induces a uniquely defined (up to a contractible indeterminacy space) morphism
        \[
            \bbL\an_{S/X} \to \cF[1],
        \]
    in $\Coh^+(S)$, such that the compositve $\bbL\an_S \to \bbL\an_{S/X} \to \cF[1]$ agrees with $d$. By applying the right adjoint $g_*$ above we obtain a
    commutative diagram
        \[
        \begin{tikzcd}
            \bbL\an_T \ar{r}{\mathrm{can}} \ar{d} & \bbL\an_{T/X} \ar{d} \ar{rd}{d''} & \\
            g_*(\bbL\an_S) \ar{r}{g_*(\mathrm{can})} & g_*(\bbL\an_{S/X}) \ar{r} & g_*(\cF)[1],
        \end{tikzcd}
        \]
    in the \infcat $\Coh^+(T)$. From this, we conclude again by the universal property of the relative analytic cotangent complex the existence
    of a uniquely defined natural morphism $T' \to X$ extending both $T \to X$ and $S' \to X$ and compatible with the restriction to $S$. The latter assertion is equivalent to state
    that the commutative square
        \[
        \begin{tikzcd}
            S \ar{r} \ar{d} & S' \ar{d} \\
            T \ar{r} & T',
        \end{tikzcd}
        \]
    is a pushout diagram in $\dAnk$. The proof is thus concluded.
\end{proof}

\begin{construction} \label{const:leq_n_construction}
    Let $n \ge 0$. We have a well defined functor
        \[
            \anNil_{\mathrm{t}_{\le n}(X)/ } \to \anNil_{X/ },  
        \]
    given by the formula
        \[
            (\rt_{\le n}(X) \to S) \in \anNil_{\rt_{\le n}(X)/ } \mapsto ( X \to S \bigsqcup_{\rt_{\le n}(X)} X ).
        \]
    Given any functor $F \colon \anNil_{X/ }\to \cS$,
    we define $F^{\le n} \colon \anNil_{\rt_{\le n}(X)/ } \to \cS$ as the functor given on objects by the association
        \[
            (\rt_{\le n}(X) \to S) \in \anNil_{\rt_{\le n}(X)/ } \mapsto  F^{\le n}(S \bigsqcup_{\rt_{\le n}(X)} X) \in \cS.
        \]
    It follows from the construction that, for each $n \ge 0$ and $S \in \anNil_{X/ }$, we have a natural morphism
        \[
            F(S) \to F^{\le n}(\rt_{\le n}(S)),  
        \]
    in the \infcat $\cS$.
\end{construction}

\subsection{Analytic formal moduli problems under a base} We study the \infcat of $k$-analytic formal moduli problems
under a base $X \in \dAnk$ and explore certain important features of such. The
results presented here will prove to be crucial for the study of the deformation to the normal cone in the $k$-analytic
setting, which we treat in the next section.
We start with the following central definition:

\begin{defin} \label{defin:analytic_formal_moduli_problems_under}
    An \emph{analytic formal moduli problem under $X$} corresponds to the datum of a functor
        \[F \colon (\anNil_{X/})\op \to \cS,\]
    satisfying the following two conditions:
    \begin{enumerate}
        \item $F(X) \simeq *$ in $\cS$;
        \item Given any $S \in \anNil\op_{X/ }$, the natural morphism
            \[
                F(S) \to \lim_{n \ge 0} F^{\le n}(\rt_{\le n}(S)),  
            \]
        is an equivalence in $\cS$ (see \cref{const:leq_n_construction}). In other words, $Y$ is \emph{convergent};
        \item Given any pushout diagram
            \[\begin{tikzcd}
                S \ar{r}{f} \ar{d} & S' \ar{d} \\
                T \ar{r} & T',
            \end{tikzcd}\]
        in the \infcat $\anNil_{X/}$, such that $f$ has the structure of a square-zero extension, the induced morphism
            \[F(T') \to F(T) \times_{F(S)}F(S),\]
        is an equivalence in $\cS$.
    \end{enumerate}
    We shall denote by $\anFMP_{X/}$ the full subcategory of $\Fun((\anNil_{X/ })\op, \cS)$ spanned by analytic formal moduli problems
    under $X$.
\end{defin}

\begin{rem}
    In the previous definition we remark that the functor
        \[
            F^{\le n} \colon \anNil\op_{\rt_{\le n}(X)/ } \to \cS,  
        \]
    is itself an analytic formal moduli problem under $\rt_{\le n}(X)$.
\end{rem}

We shall give some important examples of formal moduli problems under $X$:

\begin{eg}
    \begin{enumerate}
        \item Let $X \in \dAnk$. As in the algebraic case, we can consider the \emph{de Rham pre-stack associated to $X$}, $X_\mathrm{dR} \colon \dAfd_k\op \to \cS$,
        determined by the formula
            \[
                X_{\mathrm{dR}}(Z) \coloneqq X(Z_\red), \quad Z \in \dAfd_k.  
            \]
        We have a natural morphism $X \to X_\mathrm{dR}$ induced from the natural morphism $Z_\red \to Z$.
        The restriction of $X_\dR$ to $\Fun(\anNil\op_{X/ }, \cS)$ via the functor
            \[
                h_* \colon \dAnSt_k \to \anFMP_{X/ },
            \]
        induced by the inclusion $h \colon \anNil_{X/ } \subseteq \dAnk$, is an analytic formal moduli problem under $X$.
        Moreover, it follows from the construction of $X_\dR$ that the object $h_* (X_\dR)$ is equivalent to a final object in $\anFMP_{X/ }$.
        \item Let $f \colon X \to Y$ be a morphism in the \infcat $\dAnk$. We define the \emph{formal completion of $X$ in $Y$ along $f$} as the pullback 
            \begin{align*}
                Y^\wedge_X  & \coloneqq Y \times_{Y_\mathrm{dR}} X_\mathrm{dR} \\
                            & \in \dAnSt_k.
            \end{align*}
        By construction we have a natural factorization $X \to Y^\wedge_X \to Y$ in $\dAnSt_k$, and moreover the restriction of $X \to Y^\wedge_X$ to the \infcat
        $\Fun(\anNil\op_{X/}, \cS)$ along the natural functor
            \[
                h_* \colon  (\dAnSt_k)_{X/ } \to \anFMP_{X/ },  
            \]
        exhibits $Y^\wedge_X$ as a formal moduli problem under $X$.
        \item Let $f \colon X \to Y$ be a closed immersion in the \infcat $\dAnk$. Consider the shifted tangent bundle associated to $f$ together with the zero section
            \[
            \begin{tikzcd}
                X \ar{r}{s_0} \arrow[rd, equal] & \bT\an_{X/ Y}[1] \ar{d}{p} \\
                    & X.
            \end{tikzcd}
            \]
        The completion $\bT\an_{X/Y} [1]^\wedge \in \anFMP_{X/ }$ along $s_0$ will play an important role in what follows.
    \end{enumerate}
\end{eg}

\begin{notation}
    We set $\anNil_{X/}^\cl \subseteq \anNil_{X/}$ to be the full subcategory spanned by those objects corresponding to nil-embeddings
        \[
            X \to S,  
        \]
    in $\dAnk$.
\end{notation}

\begin{prop} \label{prop:analytic_FMP_under_X_are_ind_inf_schemes}
    Let $Y \in \anNil_{X/}$. The following assertions hold:
    \begin{enumerate}
        \item The inclusion functor
            \[
              \anNil^{\cl}_{X//Y} \hookrightarrow \anNil_{X//Y} , 
            \]
        is cofinal.
        \item The natural morphism
            \[
               \colim_{Z \in \anNil_{X//Y}^\cl} Z \to Y,  
            \]
        is an equivalence in $\Fun(\anNil_{X/}\op, \cS)$.
        \item The \infcat $\anNil_{X//Y}^\cl$ is filtered.
    \end{enumerate}
\end{prop}

\begin{proof}
    We start by proving claim (1). Let $n \ge 0$, and consider the usual restriction along the natural morphism $X_\red \to \rt_{\le n}(X)$ functor 
        \[\mathbf{res}^{\le n} \colon \anNil_{\rt_{\le n}(X)/} \to \anNil_{X_\red/}.\]
    Such functor admits a well defined left adjoint
        \[\mathbf{push}^{\le n} \colon \anNil_{X_\red/} \to \anNil_{\rt_{\le n}(X)/},\]
    which is determined by the formula
        \[
            (X_\red \to T) \in \anNil_{X_\red /} \mapsto (\rt_{\le n}(X) \to T') \in \anNil_{X/},   
        \]
    where we have set
        \begin{equation} \label{eq:set_definition_of_T'_as_pushout_of_T_along_the_inclusion_X_red_to_X}
            T' \coloneqq \rt_{\le n}(X) \bigsqcup_{X_\red} T \in \anNil_{X/}.  
        \end{equation}
    We claim that $T' \in \anNil_{\rt_{\le n}(X)/}$ belongs to the full subcategory $\anNil^\cl_{\rt_{\le n}(X)/} \subseteq \anNil_{\rt_{\le n}(X)/}$.
    Indeed, since the structural morphism
        $X_\red \to T$,
    is necessarily a nil-embedding we deduce the claim from \cref{prop:existence_of_pushouts_along_closed_nil-isomorphisms}.
    We shall denote by
        \[
            \mathbf{res}^{\le n}_!(Y) \colon \anNil\op_{X_\red /} \to \cS,
        \]
    the left Kan extension of $Y$ along the functor $\mathbf{res}^{\le n}$ above. By the colimit formula
    for left Kan extensions, c.f. \cite[Lemma 4.3.2.13]{HTT}, it follows that $\mathbf{res}^{\le n}_!(Y)$ is given by the formula
        \[
            (X_\red \to T) \in  \anNil_{X_\red /} \mapsto Y^{\le n}(T') \in \cS,
        \]
    where $T'$ is as in \eqref{eq:set_definition_of_T'_as_pushout_of_T_along_the_inclusion_X_red_to_X}. 
    We thus have a diagram of functors
        \[
            \mathbf{res}^{\le n} \colon \anNil_{\rt_{\le n}(X)/ /Y^{\le n}} \rightleftarrows \anNil_{X_\red // \mathbf{res}^{\le n}_!(Y)} \colon \mathbf{push}^{\le n},
        \]
    where $\mathbf{res}^{\le n}$ is given on objects by the formula
        \[
            (\rt_{\le n}(X) \to S \to Y^{\le n}) \in \anNil_{\rt_{\le n}(X)/ /Y^{\le n}} \mapsto (X_\red \to S \to \mathbf{res}_!^{\le n}(Y) )  \in \anNil_{X_\red // \mathbf{res}^{\le n}_!(Y)} 
        \]
    and the functor $\mathbf{push}^{\le n}$ is given by the association
        \[
            (X_\red \to T \to Y^{\le n}) \in \anNil_{X_\red/ /Y^{\le n}} \mapsto (\rt_{\le n}(X) \to T \sqcup_{X_\red} \rt_{\le n}(X) \to Y^{\le n}) \in \anNil_{\rt_{\le n}(X)/ /Y^{\le n}}.  
        \]
    We claim that the pair $(\mathbf{res}^{\le n}, \mathbf{push}^{\le n})$ forms an adjunction. Indeed, a morphism
        \[
            (X_\red \to S \to \mathbf{res}^{\le n}_!(Y)) \to \mathbf{res}_!^{\le n}(\rt_{\le n}(X) \to T \to Y^{\le n}),  
        \]
    corresponds to a commutative diagram
        \[
        \begin{tikzcd}
            X_\red \ar{r} \ar{d}{=} & S \ar{r} \ar{d} & \mathbf{res}_!^{\le n}(Y) \ar{d}{=} \\
            X_\red \ar{r} & T \ar{r} & \mathbf{res}_!^{\le n}(Y),
        \end{tikzcd}
        \]
    in the \infcat $\Fun(\anNil_{X_\red/ }\op, \cS)$. The latter datum is equivalent to the datum of a commutative diagram
        \begin{equation} \label{eq:com_diagram_expliciting_datum_of_res_push_adjunctions}
        \begin{tikzcd}
                                 &                 & \rt_{\le n}(X) \ar{d}   \ar{rd}    &   \\
            X_\red \ar{r} \ar{d} & S \ar{r} \ar{d} & S' \ar{r} \ar{d}                   & Y^{\le n} \ar{d}{=} \\
            X_\red \ar{r}        & T \ar{r}        & T' \ar{r}                          & Y^{\le n}  
        \end{tikzcd}.
        \end{equation}
    Since the morphism $\rt_{\le n}(X) \to T'$ factors through the structural map
        \[
            \rt_{\le n}(X) \to T,  
        \]
    we deduce that the datum of \eqref{eq:com_diagram_expliciting_datum_of_res_push_adjunctions} is equivalent to the datum of a commutative diagram
        \[
        \begin{tikzcd}
            \rt_{\le n}(X) \ar{r} \ar{d} & S' \ar{r} \ar{d} & Y^{\le n} \ar{d}{=} \\
            \rt_{\le n}(X) \ar{r}        & T \ar{r}        & Y^{\le n}, 
        \end{tikzcd}
        \]
    which corresponds to a uniquelly well defined morphism
        \[
            \mathbf{push}^{\le n}(X_\red \to S \to \mathbf{res}^{\le n}_!(Y)) \to (\rt_{\le n}(Y) \to T \to Y^{\le n}),
        \]
    in the \infcat $\anNil_{\rt_{\le n}(X)/ /Y^{\le n}}$. We further observe that for every $n \ge m \ge 0$, the objects
        \[
            \mathbf{res}^{\le n}_!(Y) \quad \mathrm{and} \quad \mathbf{res}^{\le m}_!(Y),  
        \]
    are equivalent as functors $\anNil_{X_\red/ } \op \to \cS$, we shall denote this functor simply by $\mathbf{res}_!(Y)$.
    
    Passing to the limit over $n \ge 0$ we obtain a commutative diagram of the form
        \[
        \begin{tikzcd}
            \anNil_{X_\red/ /\mathbf{res}_!(Y)} \ar{rr} \ar{rd} & & \lim_{n \ge 0} \anNil_{\rt_{\le n}(X)/ /Y^{\le n}} \\
                                                & \anNil_{X/ /Y} \ar{ur}.
        \end{tikzcd}
        \]  
    The horizontal morphism is cofinal since it fits into an adjunction, by our previous considerations. Thanks to \cite[
        Corollary 4.1.1.9]{HTT} in order to show that the natural morphism
        \[
            \anNil_{X_\red/ /\mathbf{res}_!(Y)} \to \anNil_{X/ /Y},
        \]
    is cofinal, it suffices to prove that $\anNil_{X/ /Y} \to \lim_{n \ge 0} \anNil_{\rt_{\le n}(X)/ /Y^{\le n}}$ is itself cofinal. But the latter is an immediate consequence of
    the fact that derived $k$-analytic spaces are nilcomplete, c.f. \cite[Lemma 7.7]{Porta_Yu_Representability}, combined with assumption (3) in \cref{defin:analytic_formal_moduli_problems_under}.
    Assertion (1) of the Proposition now follows from the observation that the functor
        \[
            \anNil_{X_\red/ /\mathbf{res}_!(Y)} \to \anNil_{X/ /Y},
        \]
    factors through the full subcategory $\anNil_{X/ /Y}^{\mathrm{cl}} \subseteq \anNil_{X/ /Y}$.
    
    Claim (2) follows immediately from (1) combined with Yoneda Lemma. To prove (3) we shall make use of \cite[Lemma 5.3.1.12]{HTT}. Let
        \[
            F \colon \partial \Delta^n \to \anNil^\cl_{X//Y}.  
        \]
    For each $[m] \in \Delta^{n}$, denote by $S_m \coloneqq F([m]) $ in $\anNil^\cl_{X//Y}$. The pushout
        \[
            S_n \bigsqcup_X S_{n-1},  
        \]
    exists in $\anNil^\cl_{X/}$. We wish to show that $S_n \bigsqcup_X S_{n-1}$ admits a morphism
        \[S_n \bigsqcup_X S_{n-1} \to Y,\]
    compatible with the diagram $F$. In order to prove the latter assertion, we observe that \cref{prop:filtered_colimit_for_nil-embeddings} can filter the diagram $F$ by diagrams $F_i \to F$ such that $X \to F_0$ is formed by square-zero
    extensions and so are each transition morphisms $F_i \to F_{i+1}$. This implies that for every $j \ge 0$, we can find morphisms
        \[
            \rt_{\le j}(S_n) \bigsqcup_{\rt_{\le j}(X)} \rt_{\le j}(S_n-1) \to Y^{\le j}, 
        \]
    which are compatible for varying $j \ge 0$.
    Moreover, the fact that $Y$ satisfies condition (2) in \cref{defin:analytic_formal_moduli_problems_under}
    implies that we can find a well defined morphism
        \[
            S_n \bigsqcup_X S_{n-1} \to Y,  
        \]
    which is compatible with $F$, as desired.
\end{proof}

As a Corollary we deduce the following important result:

\begin{cor}
    Let $X \in \dAnk$, then the \infcat $\anFMP_{X/ }$ is presentable. In particular, the latter admits all small colimits.
\end{cor}

\begin{proof} Consider the fully faithful functor
        \[
            \anNil_{X/ }^\cl \hookrightarrow \anNil_{X/ } \to \anFMP_{X/ }.
        \]
    It follows from the definitions that the \infcat $\anFMP_{X/ }$ admits filtered colimits. In particular, we obtain
    a well defined functor
        \[
            F \colon \Ind(\anNil_{X/ } ^\cl) \to \anFMP_{X/ },
        \]
    which is further fully faithful, since every the image of every $(X \to S) \in \anNil_{X/ }$ is compact in $\anFMP_{X/ }$.
    It follows from \cref{prop:analytic_FMP_under_X_are_ind_inf_schemes} that $F$ itself is essentially surjective. Since the
    \infcat $\anNil_{X/ }^\cl$ is essentially small we are reduced to show that $\anFMP_{X/ }$ admits all small colimits. 
    We already know that $\anFMP_{X/ }$ admits filtered colimits. We shall prove that $\anFMP_{X/ }$ admits finite colimits as well.
    As a consequence of \cref{prop:existence_of_pushouts_along_closed_nil-isomorphisms}, we deduce that
    the \infcat $\anNil_{X/ }^\cl$ admits all finite colimits. It is now clear that $\anFMP_{X/ }$ admits finite colimits as well,
    and the proof is concluded.
\end{proof}

\begin{defin}
    Let $Y \in \anFMP_{X/}$ denote an analytic formal moduli problem under $X$. The \emph{relative pro-analytic cotangent complex of $Y$ under $X$} is defined as the pro-object
        \begin{align*}
            \bbL\an_{X/Y} &  \coloneqq \{ \bbL\an_{X/Z} \}_{Z \in \anNil^\cl_{X//Y}} \\
                          &  \in \pro(\Coh^+(X)),
        \end{align*}
    where, for each $Z \in \anNil^\cl_{X//Y}$
        \[\bbL\an_{X/Z} \in \Coh^+(X),\]
    denotes the usual relative analytic cotangent complex associated to
    the structural morphism $X \to Z$ in $\anNil^\cl_{X//Y}$.
\end{defin}

\begin{rem}
    Let $Y \in \anFMP_{X/}$. For a general $Z \in \dAnk$, there exists a natural morphism
        \[
            \bbL\an_X \to \bbL\an_{X/Z} ,  
        \]
    in the \infcat $\Coh^+(X)$. Passing to the limit over $Z \in \anNil^\cl_{X//Z}$, we obtain a natural map
        \[
            \bbL\an_X \to \bbL\an_{X/Y},  
        \]
    in $\pro(\Coh^+(X))$, as well.
\end{rem}

The following result justifies our choice of terminology for the object $\bbL\an_{X/Y} \in \pro(\Coh^+(X))$:

\begin{lem} \label{lem:pro_cot_complex_classifies_nil_extensions_for_analytic_moduli_problems}
    Let $Y \in \anFMP_{X/}$. Let $X \hookrightarrow S$ be a square-zero extension associated to an analytic derivation
        \[
            d \colon \bbL\an_S \to \cF [1] ,  
        \]
    where $\cF \in \Coh^+(X)^{\ge 0}$. Then there exists a natural morphism
        \[
            \Map_{\anFMP_{X/}}(S, Y) \to \Map_{\pro(\Coh^+(X))}(\bbL\an_{X/Y}, \cF) \times_{\Map_{\Coh^+(X)}(\bbL\an, \cF)} \{ d \}
        \]
    which is furthermore an equivalence in the \infcat $\cS$.
\end{lem}

\begin{proof}
    Thanks to \cref{prop:analytic_FMP_under_X_are_ind_inf_schemes} combined with the Yoneda Lemma
    we can identify the space of liftings of the map $X \to Y$ along $X \to S$ with the mapping space
        \[
            \Map_{\anFMP_{X/}}(S, Y) \simeq \colim_{Z \in \anNil_{X//Y}} \Map_{\anNil_{X/}} (S, Z).  
        \]
    Fix $Z \in \anNil_{X//Y}^\cl$. Then we have a natural identification of mapping spaces
        \begin{align} \label{eq:identification_of_universal_property_of_rel_an_cotagent_complex}
            \Map_{\anNil_{X/}} (S, Z) & \simeq \Map_{(\dAnk)_{X/}} (S, Z) \\
                                      & \simeq \Map_{\Coh^+(X)}(\bbL\an_{X/Z}, \cF) \times_{\Map_{\Coh^+(X)}(\bbL\an_X, \cF)} \{ d \},
        \end{align}
    see \cite[\S 5.4]{Porta_Yu_Representability} for a justification of the latter assertion.
    Passing to the colimit over $Z \in \anNil_{X//Y}^\cl$, we conclude thanks to the formula for mapping spaces in pro-\infcats that we have
    a natural equivalence
        \[
            \Map_{\anFMP_{X/}}(S, Y) \simeq \Map_{\pro(\Coh^+(X))}(\bbL\an_{X/Y}, \cF) \times_{\Map_{\Coh^+(X)}(\bbL\an, \cF)} \{ d \},
        \]
    as desired.
\end{proof}

\begin{construction} \label{rem:morphisms_of_AnFMP_induce_transition_morphisms_on_relative_analytic_cot_complexes}
    Let $f \colon Y \to Z$ denote a morphism in $\anFMP_{X/}$. Then, for every $S \in \anNil_{X/ /Y}^\cl$, the induced morphism
        \[
            S \to Z,   
        \]
    in $\anFMP_{X/}$ factors necessarily through some $S' \in \anNil^\cl_{X/ /Z}$. For this reason, we obtain a natural morphism
        \[
            \bbL\an_{X/S'} \to \bbL\an_{X/S},  
        \]
    in the \infcat $\Coh^+(X)$. Passing to the limit over $S \in \anNil^\cl_{X//Y}$ we obtain a canonically defined morphism
        \[
            \theta (f) \colon \bbL\an_{X/Z} \to \bbL\an_{X/Y},  
        \]
    in $\pro(\Coh^+(X))$. Moreover, this association is functorial and thus we obtain a well defined functor
        \[
            \bbL\an_{X/ \bullet} \colon \anFMP_{X/ } \to \pro(\Coh^+(X)),  
        \]
    given by the formula
        \[
            (X \to Y) \in \anFMP_{X/ } \mapsto \bbL\an_{X/ Y} \in \pro(\Coh^+(X)).  
        \]
\end{construction}

\begin{prop} \label{prop:conservativity_of_relative_an_cot_complex}
    Let $X \in \dAnk$ be a derived $k$-analytic space. Then the functor
    \[
        \bbL\an_{X/ \bullet} \colon \anFMP_{X/} \to \pro(\Coh^+(X)),
    \]
    obtained via \cref{rem:morphisms_of_AnFMP_induce_transition_morphisms_on_relative_analytic_cot_complexes,rem:morphisms_of_AnFMP_induce_transition_morphisms_on_relative_analytic_cot_complexes},
    is conservative.
\end{prop}

\begin{proof} Let $f \colon Y \to Z $ be a morphism in $\anFMP_{X/}$.
    Thanks to \cref{prop:analytic_FMP_under_X_are_ind_inf_schemes} we are reduced to show that given any
        \[
            S \in \anNil_{X//Z}^\cl,  
        \]
    the structural morphism $g_S \colon X \to S$ admits a unique extension $S \to Y$ which factors the structural morphism $X \to Z$. Thanks to
    \cref{prop:filtered_colimit_for_nil-embeddings} we can reduce ourselves to the case where $X \to S$ has the structure of
    a square-zero extension. In this case, the result follows from \cref{lem:pro_cot_complex_classifies_nil_extensions_for_analytic_moduli_problems}
    combined with our hypothesis.
\end{proof}

Our goal now is to give an alternative description of analytic formal moduli problems under $X \in \dAnk$, in terms of derived $k$-analytic stacks: 

\begin{construction} \label{const:anFMP_as_ind_inf_schemes} Consider the \infcat of derived $k$-analytic stacks, $\dAnSt_k$.
    We have a natural functor
        \[
            h \colon \anNil_{X/} \to \dAnk \hookrightarrow \dAnSt_k.
        \]
    Therefore, given any derived $k$-analytic stack $Y$ equipped with a morphism $X \to Y$, one can consider its restriction to the \infcat
    $\anNil_{X/}$:
        \[
            Y \circ h \colon \anNil_{X/} \op \to \cS.      
        \]
    We have thus a natural restriction functor
        \[
            h_* \colon \dAnSt_k \to \Fun(\anNil_{X/} \op, \cS).  
        \]
On the other hand, \cref{prop:analytic_FMP_under_X_are_ind_inf_schemes} allows us to define a natural functor
    \[
        F \colon \anNil_{X/} \op \to \dAnSt_k
    \]
via the formula
    \[
        (X \to Y) \in \anFMP_{X/ } \mapsto \colim_{S \in \anNil_{X/ /Y}^\cl} S,  
    \]
the colimit being computed in the \infcat $\dAnSt_k$. The latter agrees with the left Kan extension of the functor 
    \[
        h \colon \anNil_{X/ } \to \dAnSt_k,  
    \]
along the natural inclusion functor $\anNil_{X/ } \hookrightarrow \anFMP_{X/ }$.
In particular, any analytic formal moduli under $X$ when regarded as a derived $k$-analytic stack can be realized
as an \emph{ind}-\emph{inf}-object, i.e. it can be written as a filtered colimit of nil-embeddings $X \to Z$.
We refer the reader to \cite[\S1]{Gaitsgory_Study_II} for a precise meaning
of the latter notion in the algebraic setting.
\end{construction}

\begin{defin}
    Let $f \colon X \to Y$ be a morphism in the \infcat $\dAnSt_k$. We shall say that $f$ has a \emph{deformation theory} if it satisfies the following conditions:
    \begin{enumerate}
        \item Both $X$ and $Y$ is \emph{nilcomplete}, c.f. \cite[Definition 7.4]{Porta_Yu_Representability};
        \item $Y$ is \emph{infinitesimally cartesian} if it satisfies \cite[Definition 7.3]{Porta_Yu_Representability};
        \item The morphism $f \colon X \to Y$ admits a relative \emph{analytic pro-cotangent complex}, i.e., if it satisfies \cite[Definition 7.6]{Porta_Yu_Representability} under the weaker assumption
        that the corresponding derivation functor is pro-corepresentable.
    \end{enumerate}
\end{defin}

\begin{prop} \label{prop:sufficient_conditions_for_a_prestack_to_be_equiv_to_an_analytic_FMP}
    Let $Y \in (\dAnSt_k)_{X/}$. Assume further that $Y$ admits a deformation theory.
    Then $Y$ is equivalent to an analytic formal moduli problem under $X$.
\end{prop}

\begin{proof}
    We must prove that given a pushout diagram
        \[
        \begin{tikzcd}
            S \ar{r}{f} \ar{d}{g} & S' \ar{d} \\
            T \ar{r} & T'  
        \end{tikzcd}
        \]
    in the \infcat $\anNil_{X/}$, where $f$ has the structure of a square-zero extension, then the natural morphism
        \[
            Y(T') \to Y(T) \times_{Y(S)} Y(S'),  
        \]
    is an equivalence in the \infcat $\cS$. Suppose further that $S \hookrightarrow S'$ is associated to some analytic derivation
        \[
            d \colon \bbL\an_S \to \cF[1],
        \]
    for some $\cF \in \Coh^+(S)^{\ge 0}$.
    Thanks to \cref{lem:pushouts_of_square_zero_extensions_have_the_structure_of_a_square_zero_extension} we deduce
    that the induced morphism $T \to T'$ admits a structure of a square-zero extension, as well.
    Then, by our assumptions that $Y$ is infinitesimally cartesian and it admits a relative pro-cotangent complex, we obtain a chain of natural equivalences of the form
        \begin{align*}
            Y(T') & \simeq \bigsqcup_{f \colon T \to Y} \Map_{T/ }(T', Y) \\
                  & \simeq \bigsqcup_{f \colon T \to Y} \Map_{\pro(\Coh^+(T))_{\bbL\an_T/ }}(\bbL\an_{T/ Y}, g_*(\cF)[1]) \\
                  & \simeq \bigsqcup_{f \colon T \to Y} \Map_{\pro(\Coh^+(S))_{g^* \bbL\an_T/ }}(g^* \bbL\an_{T/ Y}, \cF[1]) \\
                  & \simeq \bigsqcup_{f \colon T \to Y} \Map_{\pro(\Coh^+(S))_{\bbL\an_S/ }}(\bbL\an_{S/ Y}, \cF[1]) \\
                  & \simeq \bigsqcup_{f \colon T \to Y} \Map_{S/ }(S', Y) \\
                  & \simeq Y(T) \times_{Y(S)} Y(S'),
        \end{align*}
    where only the fourth equivalence requires an additional justification, namely: it follows from the existence of a commutative diagram between fiber sequences
        \[
        \begin{tikzcd}    
            g^* f^*\bbL\an_Y \ar{r} \ar{d}{=} & g^* \bbL\an_{T} \ar{r} \ar{d} &  g^* \bbL\an_{T/Y} \ar{d} \\
            (f\circ g)^*\bbL\an_Y \ar{r} & \bbL\an_{S} \ar{r} & \bbL\an_{S/Y},
        \end{tikzcd}
        \]
    in the \infcat $\pro(\Coh^+(S))$ combined with the fact that the derivation $d_T \colon \bbL\an_T \to g_*(\cF)[1]$ is induced from
        \[
            d \colon \bbL\an_S \to \cF[1],
        \]
    as in the proof of \cref{lem:pushouts_of_square_zero_extensions_have_the_structure_of_a_square_zero_extension}. The result now follows.
\end{proof}

\begin{prop} \label{prop:stacks_under_X_with_deformation_theory_are_analytic_FMP}
    Let $Z \in (\dAnSt_k)_{X/ }$ such that the structural morphism $X \to Z$ is a nil-isomorphism and assume that
    $Z$ admits a deformation theory. Then the natural morphism
        \[
            \colim_{S \in \anNil_{X/ /Z}^\cl} S \to Z,  
        \]
    in the \infcat $(\dAnSt_k)_{/X}$, is an equivalence.
\end{prop}

\begin{proof}
    We shall prove that for every derived $k$-analytic space $T \in \dAnk$, any diagram of the form
        \[
        \begin{tikzcd}
            X \ar{rr}{f} \ar{rd}{g} & & Z \\
                & T \ar{ru}{h} &  
        \end{tikzcd},
        \]
    factors through an object
        \[
            X \to S \to Z,  
        \]
    in $\anNil_{X/ /Z}^\cl$. Consider the commutative diagram
        \begin{equation} \label{eq:comm_diagram_T_T_red_X_to_Z}
        \begin{tikzcd}
            X_{\red} \ar{r} \ar{d} & T_\red \ar{d} \ar{dr} & \\
            X \ar{r} & T \ar{r} & Z,
        \end{tikzcd}
        \end{equation}
    in the \infcat $\dAnSt_k$. Since $Z_\red \simeq X_\red$ we obtain that $T_\red \to Z$ factors necessarily through the colimit
        \[
            \colim_{S \in \anNil^\cl_{X/ /Z}} S \in \anPreStk.  
        \]
    Assume first that $T$ is bounded, i.e. $T \in \dAnk^{< \infty}$. Then we can construct $T$ out of $T_\red$ via a finite sequence of square-zero extensions,
    as in \cref{prop:filtered_colimit_for_nil-embeddings}. Therefore, in order to construct
    a factorization
        \[
            T \to S \to Z,  
        \]
    in $\anPreStk_{X/}$, we reduce ourselves to the case where the morphism $T_\red \to T$ is itself a square-zero extension.
    In this case, let
        \[
            d \colon \bbL\an_{T_\red} \to \cF[1],  
        \]
    where $\cF \in \Coh^+(T_\red)^{\ge 0}$ be the associated derivation. The existence of the diagram \eqref{eq:comm_diagram_T_T_red_X_to_Z} implies that
    we have a commutative diagram of the form
        \[
        \begin{tikzcd}
                g^*\bbL\an_{T_\red} \ar{r} & \bbL\an_{X_\red} \ar{r} & \bbL\an_{X_\red/ T_\red} \\
                f^* \bbL\an_{Z} \ar{r} \ar{u} & \bbL\an_{X_\red} \ar{u} \ar{r}& \bbL\an_{X_\red/ Z} \ar{u},
        \end{tikzcd}
        \]
    in the \infcat $\pro(\Coh^+(X))$. For this reason, the limit-colimit formula for mapping spaces in pro-\infcats implies that the natural morphism
        \[
            \bbL\an_{X_\red/ Z} \to \bbL\an_{X_\red/ T_\red},
        \]
    in the \infcat $\pro(\Coh^+(X))$ factors necessarily via a morphism of the form
        \[
            \bbL\an_{X_\red/ S} \to \bbL\an_{X_\red/ T_\red},  
        \]
    for some suitable $S \in \anNil^\cl_{X/ /Z}$. Thus the existence problem
        \[
        \begin{tikzcd}
            T_\red \ar{r} \ar{d} & T \arrow[d, dashrightarrow] \ar{rd} & \\
            X_\red = S_\red \ar{r} & S  \ar{r} & Z,
        \end{tikzcd}
        \]
    admits a solution $T \to S \to Z$, as desired. If $T \in \dAnk$ is a general derived $k$-analytic space, we reduce ourselves to the bounded case
    using the fact that $Z$ is nilcomplete.
\end{proof}

\begin{cor}
    The functor
        \[
            F \colon \anFMP_{X/ } \to (\dAnSt_k)_{X/ },  
        \]
    is fully faithful. Moreover, its essential image coincides with those $Z \in (\dAnSt_k)_{X/ }$ which admit deformation theory.
\end{cor}

\begin{proof}
    Let $(\dAnSt_k)^{\mathrm{def}}_{X/ } \subseteq (\dAnSt_k)_{X/ }$ denote the full subcategory spanned by derived $k$-analytic stacks under $X$ admitting
    a deformation theory. It is clear that the natural functor
        \[
            F \colon \anFMP_{X/ } \to (\dAnSt_k)_{X/ },  
        \]
    factors through the full subcategory $(\dAnSt_k)_{X/ }^\mathrm{def}$. Moreover, the restriction functor
        \[
            \mathrm{res} \colon (\dAnSt_k)_{X/ }^\mathrm{def} \to \Fun(\anNil_{X/ }\op, \cS),  
        \]
    factors through $\anFMP_{X/ } \subseteq \Fun(\anNil_{X/ }\op, \cS)$. Moreover,
    \cref{prop:stacks_under_X_with_deformation_theory_are_analytic_FMP} implies that the restriction functor that $F$ and $\mathrm{res}$ are mutually inverse functors,
    proving the claim.
\end{proof}

\subsection{Analytic formal moduli problems over a base}
Let $X \in \dAnk$ denote a derived $k$-analytic space. In \cite[Definition 6.11]{Porta_Yu_NQK}
the authors introduced the \infcat of \emph{analytic formal moduli problems over $X$}, which we shall review:

\begin{notation}
    Let $X \in \dAnk$. We shall denote by $\anNil_{/X}$ the full subcategory of $(\dAnk)_{/X}$ spanned by nil-isomorphisms
        \[
            Z \to X,
        \]
    in the \infcat $\dAnk$.
\end{notation}

\begin{defin} \label{defin:analytic_formal_problems_over_X}
    We denote by $\anFMP_{/ X} \subseteq \Fun(\anNil_{/ X} \op, \cS)$ the full subcategory spanned by those functors $Y \colon \anNil_{/ X} \op \to \cS$
    satisfying the following:
    \begin{enumerate}
        \item $Y(X_\red) \simeq *$;
        \item The natural morphism
            \[
                Y(S) \to \lim_{n \ge 1} Y(\rt_{\le n}(S)),  
            \]
        induced by the natural inclusion morphisms $\rt_{\le n}(S) \to S$ for each $n \ge 0$, is an equivalence in $\cS$;
        \item For each pushout diagram
            \[
            \begin{tikzcd}
                S \ar{r}{f} \ar{d}{g}  & S' \ar{d} \\
                T \ar{r} & T',
            \end{tikzcd}
            \]
        in $\anNil_{/ X}$, for which $f$ has the structure of a square-zero extension, the canonical morphism
            \[
                Y(T') \to Y(T) \times_{Y(S)} Y(S'),
            \]
        is an equivalence in $\cS$.
    \end{enumerate}
\end{defin}

\begin{defin}
    We shall denote by $\anNil_{/X}^\cl \subseteq \anNil_{/X}$ the faithful subcategory in wich we allow morphisms
        \[
            i \colon S \to S'  ,
        \]
    where $i$ is a nil-embedding in $\dAnk$.
\end{defin}

We start with the analogue of \cref{prop:analytic_FMP_under_X_are_ind_inf_schemes} in the setting of analytic formal moduli problems over $X$:

\begin{prop} \label{prop:required_conditions_for_formal_moduli_problems}
    Let $Y \in \anFMP_{/X}$. The following assertions hold:
    \begin{enumerate}
        \item The inclusion functor
            \[
                (\anNil^{\cl}_{/X})_{/Y}  \to (\anNil_{/X})_{/Y},
            \]
        is cofinal.
        \item The natural morphism
            \[
                \colim_{Z \in (\anNil^\cl_{/X})_{/Y}}  Z \to Y,
            \]
        is an equivalence in the \infcat $\anFMP_{/X}$.
        \item The \infcat $\anNil_{/X}^\cl$ is filtered.
    \end{enumerate}
    We shall refer to objects in $\anFMP_{/ X}$ as formal moduli problems over $X$.
\end{prop}

\begin{proof}
    We first prove assertion (1). Let $S \to Z$ be a morphism in $(\anNil^\cl_{/X})_{/Y}$. Consider the pushout diagram
        \begin{equation} \label{eq:diagram_pushout_of_nil_isomorphisms_with_ltv_being_the_reduced_subspace}
        \begin{tikzcd}
            S_\red \ar{r} \ar{d} & S \ar{d} \\
            Z \ar{r} & Z',
        \end{tikzcd}
        \end{equation}
    in the \infcat $\anNil_{/X}$ whose existence is guaranteed by \cref{prop:existence_of_pushouts_along_closed_nil-isomorphisms}. Since
    the upper horizontal morphism in \eqref{eq:diagram_pushout_of_nil_isomorphisms_with_ltv_being_the_reduced_subspace} is a nil-embedding
    combined with the fact that $Y$ is itself convergent, we can reduce
    ourselves via \cref{prop:filtered_colimit_for_nil-embeddings} to the case where the latter is an actual square-zero extension.
    Since $Y$ is assumed to be an analytic formal moduli problem over $X$ we then deduce that the canonical morphism
        \begin{align*}
            Y(Z') & \to Y(Z) \times_{Y(S_\red)} Y(S) \\
            &\simeq Y(Z) \times Y(S),  
        \end{align*}
    is an equivalence (we implicitly used above the fact that $S_\red \simeq X_\red$). As a consequence the object $(Z' \to X)$ in $\anNil_{/X}$ admits an induced morphism
    $Z' \to Y$ making the required diagram commute.
    Thanks to \cref{prop:existence_of_pushouts_along_closed_nil-isomorphisms} we deduce that both
    $S \to Z'$ and $Z \to Z'$ are nil-embeddings. Therefore, we can factor the diagram
        \[
        \begin{tikzcd}
            S \ar{rr} \ar{rd} & & Z \ar{dl} \\
                &               Y       &  
        \end{tikzcd}
        \]
    via a closed nil-isomorphism $Z \to Z'$. As a consequence, we deduce readily that the inclusion functor
        \[(\anNil^\cl_{/X})_{/Y} \to (\anNil_{/X})_{/Y},\]
    is cofinal.
    Assertion (2) is now an immediate consequence of (1). We now prove (3). Let 
        \[\theta \colon K \to (\anNil^\cl_{/X})_{/Y},\]
    be a functor where
    $K$ is a finite \infcat. We must show that $\theta$ can be extended to a functor
        \[\theta^{\rhd} \colon K^{\rhd} \to (\anNil^\cl_{/X})_{/Y}.\]
    Thanks to \cref{prop:filtered_colimit_for_nil-embeddings} we are allowed to reduce ourselves to the case where morphisms indexed by $K$
    are square-zero extensions. The result now follows from the fact that $Y$ being an analytic moduli problem sends finite colimits along square-zero extensions
    to finite limits.
\end{proof}

\begin{lem} \label{formal_moduli_under_induce_formal_moduli_over_via_base_change}
    Let $X \in \dAnk$. Given any $Y \in \anFMP_{X/ }$, then for each $i= 0, 1$ the $i$-th projection morphism
        \[
            p_i \colon X \times_Y X \to X,  
        \]
    computed in the \infcat $\dAnSt_k$ lies in the essential image of $\anFMP_{/X}$ via the canonical functor $\anFMP_{/ X} \to (\dAnSt_k)_{/ X}$.
\end{lem}

\begin{proof}
    Consider the pullback diagram
        \[
        \begin{tikzcd}
            X \times_Y X \ar{r}{p_1} \ar{d}{p_0} & X \ar{d} \\
            X \ar{r} & Y ,
        \end{tikzcd}
        \]
    computed in the \infcat $\dAnSt_k$. Thanks to \cref{prop:analytic_FMP_under_X_are_ind_inf_schemes} together with the fact that fiber products commute with filtered colimis in the \infcat $\dAnSt_k$,
    we deduce that
        \[
            X \times_Y X \simeq \colim_{Z \in \anNil^\cl_{X//Y}} X \times_Z X, 
        \]
    in $\dAnSt_k$. It is clear that $(p_i \colon X \times_Z X \to X)$ lies in the essential image of $ \anFMP_{/X}$, for each $Z \in (\anNil_{/ X} ^\cl)_{/ Y}$ and $i = 0, 1$.
    Thus also the filtered colimit
        \[
            (p_i \colon X \times_Y X \to X) \in \anFMP_{/X}, \quad \mathrm{for \ } i = 0, 1,
        \]  
    as desired.
\end{proof}

Just as in the previous section we deduce that every analytic formal moduli problem over $X$ admits the structure of an \emph{ind}-\emph{inf}-object
in $\anPreStk_k$:

\begin{cor} \label{cor:formal_moduli_problems_over_X_are_ind_inf_objects}
    Let $Y \in (\dAnSt_k)_{/X}$. Then $Y$ is equivalent to an analytic formal moduli problem over $X$ if and only if there exists
    a presentation 
        \[Y \simeq \colim_{i \in I} Z_i,\]
    where $I$ is a filtered \infcat and for every $i \to j$ in $I$, the induced morphism
        \[
          Z_i \to Z_j,  
        \]
    is a closed embedding of derived $k$-affinoid spaces that are nil-isomorphic to $X$.
\end{cor}

\begin{proof}
    It follows immediately from \cref{prop:required_conditions_for_formal_moduli_problems} (2).
\end{proof}

\begin{defin}
    Let $Y \in \anFMP_{/X}$. We define the \infcat of \emph{coherent modules on $Y$}, denoted $\Coh^+(Y)$, as the limit
        \[
            \Coh^+(Y) \coloneqq \lim_{Z \in (\dAnk)_{/Y}}  \Coh^+(Z),
        \]
    computed in the \infcat $\Catst$. We define the \infcat of \emph{pseudo-pro-coherent modules on $Y$}, denoted $\pro^{\mathrm{ps}}(\Coh^+(Y))$, as
        \[
            \pro^\mathrm{ps}(\Coh^+(Y)) \coloneqq \lim_{Z \in (\anNil_{/X}^\cl)_{/Y}} \pro(\Coh^+(Z)),  
        \]
    where the limit is computed in the \infcat $\Catst$.
\end{defin}

\begin{defin} Let $Y \in \anFMP_{/X}$, $Z \in \dAfd_k$ and let $\cF \in \Coh^+(Z)^{\ge 0}$. Suppose furthermore that we are given a morphism $f \colon Z \to Y$.
    We define the \emph{tangent space of $Y$ at $f$ twisted by $\cF$} as the fiber
        \[
            \bbT\an_{Y, Z, \cF,f} \coloneqq \mathrm{fib}_f\big( Y(Z[\cF]) \to Y(Z) \big) 
            \in \cS.  
        \]
    Whenever the morphism $f$ is clear from the context, we shall drop the subscript $f$ above and denote the tangent space
    simply by $\bbT\an_{Y, Z, \cF}$.
\end{defin}

\begin{rem} \label{rem:construction_of_analytic_cotangent_complex_of_analytic_FMP_over_X} Let $Y \in \anFMP_{/X}$.
    The equivalence of ind-objects
        \[
            Y \simeq \colim_{S \in (\anNil^\cl_{/X})_{/Y}} S,
        \]
    in the \infcat $\dAnSt_k$, implies that, for any $Z \in \dAfd_k$, one has an equivalence of mapping spaces
        \[
            \Map_{\dAnSt_k}(Z, Y) \simeq \colim_{S \in (\anNil^\cl_{/X})_{/Y}}  \Map_{\anPreStk}(Z,S ).
        \]
    For this reason, given any morphism $f \colon Z \to Y$ and any
    $\cF \in \Coh^+(Z)^{\ge 0}$, we can identify the tangent space $\bbT\an_{Y, Z, \cF}$
    with the filtered colimit of spaces
        \begin{align} \label{eq:description_of_cotangent_complex_of_a_formal_moduli_problem_in_terms_of_its_tangent_space}
            \bbT\an_{Y, Z, \cF} & \simeq \colim_{S \in (\anNil^\cl_{/X})_{Z/ /Y}} \fib_f \big( S(Z[\cF]) \to S(Z) \big) \\
                                & \simeq \colim_{S \in (\anNil^\cl_{/X})_{Z/ /Y}} \bbT\an_{S, Z , \cF } \\
                                & \simeq \colim_{S \in (\anNil^\cl_{/X})_{Z/ /Y}} \Map_{\Coh^+(Z)} (f_{S,Z}^* (\bbL\an_S), \cF),
        \end{align}
    where $f_{S, Z} \colon Z \to S$ is a transition morphism, in $(\dAnk)_{/X}$, factoring $f \colon Z \to Y$ such that
        \[(S \to X) \in \anNil^\cl_{/X}.\]
    The final equivalence in \eqref{eq:description_of_cotangent_complex_of_a_formal_moduli_problem_in_terms_of_its_tangent_space}, follows from \cite[Lemma 7.7]{Porta_Yu_Representability}.
    Therefore, we deduce that
    the analytic formal moduli problem
    $Y \in \anFMP_{/X}$ admits an \emph{absolute pro-cotangent complex} given as
        \[
           \bbL\an_Y \coloneqq \{ f^*_{S,Z}( \bbL\an_S )  \}_{Z, S \in (\anNil^\cl_{/X})_{/Y}} \in \pro^\mathrm{ps}(\Coh^+(Y)).
        \]
\end{rem}   

\begin{cor}
    Let $Y \in \anFMP_{/X}$. Then its absolute cotangent complex $\bbL\an_Y$ classifies analytic deformations on $Y$. More precisely, given any morphism $Z \to Y$
    where $Z \in \dAfd_k$ and $\cF \in \Coh^+(Z)^{\ge 0}$ one has a natural equivalence of mapping spaces
        \[
            \bbT\an_{Y, Z, \cF} \simeq \Map_{\pro(\Coh^+(Y))}(\bbL\an_Y, \cF).  
        \]
\end{cor}

\begin{proof}
    It follows immediately from the natural equivalences displayed in \eqref{eq:description_of_cotangent_complex_of_a_formal_moduli_problem_in_terms_of_its_tangent_space}
    combined with the description of mapping spaces in \infcats of pro-objects.
\end{proof}

We now introduce the notion of square-zero extensions of analytic formal moduli problems over $X$:

\begin{construction} \label{const:square_zero_extensions_of_analytic_FMP_over_X}
    Let $(f \colon Y \to X) \in \anFMP_{/X}$. Let $d \colon \bbL\an_Y \to \cF[1]$ be an \emph{analytic derivation} in $\pro(\Coh^+(Y))$, where
    $\cF \in \Coh^+(Y)^{\ge 0}$, such that
        \[  
            \cF \simeq f^*(\cF'),
        \]
    for some suitable object $\cF' \in \Coh^+(X)^{\ge 0}$. Thanks to \cref{rem:construction_of_analytic_cotangent_complex_of_analytic_FMP_over_X}
    one has the following sequence of natural equivalences of mapping spaces
        \begin{align} \label{eq:mapping_space_of_derivations}
            \Map_{\pro(\Coh^+(Y))} (\bbL\an_Y , \cF[1])  & \simeq 
            \lim_{S \in (\anNil^\cl_{/X})_{/Y}} \colim_{S' \in (\anNil^\cl_{/X})_{S//Y}} \Map_{\pro(\Coh^+(S))}(f_{S, S'}^*(\bbL\an_{S'}), g_{S}^*(\cF')[1]) \\
                                                         & \simeq
            \lim_{S \in (\anNil^\cl_{/X})_{/Y}} \colim_{S' \in (\anNil^\cl_{/X})_{S//Y}} \Map_{\pro(\Coh^+(S))}(\bbL\an_{S'}, (f_{S, S'})_{ *} g_S^*(\cF')[1]),
        \end{align}
    where $g_{S} \colon S \to X$ denotes the structural morphism in $\anNil^\cl_{/X}$ and $f_{S, S'} \colon S \to S'$ a given transition morphism
    in the \infcat $(\anNil^\cl_{/X})_{/Y}$. For this reason, we can form the filtered colimit
        \[
            Y' \coloneqq \colim_{S \in (\anNil^\cl_{/X})_{/Y}}  \colim_{S' \in (\anNil^\cl_{/X})_{S//Y}} \oS' \in \dAnSt_k,
        \]
    where $S' \to \overline{S}'$ denotes the square-zero extension induced from $d$ together with \eqref{eq:mapping_space_of_derivations}.
    By construction, one has a natural morphism $Y \hookrightarrow Y'$ in the \infcat $(\dAnSt_k)_{/ X}$.
    Moreover, thanks to \cref{prop:sufficient_conditions_for_a_prestack_to_be_equiv_to_an_analytic_FMP} it follows that $Y' \in \anFMP_{/X}$.
\end{construction}

\begin{defin}
    Let $Y \in \anFMP_{/X}$. Suppose we are given an analytic derivation
        \[
            d \colon \bbL\an_Y \to \cF [1],  
        \]
    in $\pro(\Coh^+(Y))$ where $\cF \in \Coh^+(Y)^{\ge 0}$ is such that $\cF \simeq f^*(\cF')$, for some $\cF' \in \Coh^+(X)^{\ge 0}$. We shall
    say that the induced morphism
        \[
            h \colon Y \to Y',  
        \] 
    defined in \cref{const:square_zero_extensions_of_analytic_FMP_over_X}, is a \emph{square-zero extension} of $Y$
    associated to the analytic derivation $d$. 
\end{defin}

\begin{cor} \label{cor:construction_of_square_zero_extensions_for_analytic_FMP_using_univ_property_of_cotangent_complex}
    Let $Y \in \anFMP_{/X}$. Let $h \colon X \hookrightarrow S$ denote a square-zero extension in $\dAnk$.
    Then the space of cartesian squares
        \[
        \begin{tikzcd}
            Y \ar{r}{h'} \ar{d}{f} & Y' \ar{d}{g} \\
            X \ar{r}{h} & S,
        \end{tikzcd}
        \]
    such that $h' \colon Y \to Y'$ is a square-zero extension and $g \colon Y' \to S$ exhibits the former
    as an analytic formal moduli problem over $S$ is naturally equivalent to the space of factorizations
        \[
            f^* \bbL\an_X \to \bbL\an_Y \to f^*( \cF')[1],
        \]
    in $\pro^\mathrm{ps}(\Coh^+(Y))$, of the analytic derivation $d \colon \bbL\an_X \to \cF'[1]$ associated to the morphism $h$ above.
\end{cor}

\begin{proof}
    By the universal property of filtered colimits together with the fact that these preserve fiber products we reduce the statement to the case where $Y \in \anNil_{/X}$ and thus $Y' \in \anNil_{/S}$, in which
    case the statement follows immediately by the universal property of the relative analytic cotangent complex.
\end{proof}

\begin{cor} \label{cor:universal_property_of_relative_cotangent_complex_for_morphisms_between_analytic_FMP}
    Let $f \colon Z \to X$ be a morphism in the \infcat $\dAnk$. Suppose we are given analytic formal moduli problems
        \[
            f \colon Y \to X \quad \mathrm{and} \quad g \colon \tZ \to Z
        \]
    together with a commutative diagram
        \[
        \begin{tikzcd}
            \tZ \ar{r}{s} \ar{d} & Y \ar{d} \\
            Z \ar{r}{f} & X,  
        \end{tikzcd}
        \]
    in the \infcat $\dAnSt_k$. Let $d \colon \bbL\an_Z \to \cF[1]$, where $\cF \in \Coh^+(Z)^{\ge 0}$, be an analytic derivation corresponding to a square-zero extension morphism
    $Z \to Z'$ in the \infcat $\dAnk$. Assume that $d$ induces a square-zero extension
    $\widetilde{d} \colon \bbL\an_{\tZ} \to \cF[1]$
    and let $h \colon \tZ \hookrightarrow \tZ'$ be the induced square-zero extension in $\dAnSt_k$ such that we have a cartesian diagram
        \[
        \begin{tikzcd}
            \tZ \ar{r} \ar{d} & \tZ' \ar{d} \\
            Z \ar{r} & Z'  
        \end{tikzcd}
        \]
    in the \infcat $\dAnSt_k$. Then the space of factorizations
        \[
            s \colon \tZ \to \tZ' \to Y,  
        \]
    is naturally equivalent to the space of factorizations
        \[
            \widetilde{d} \colon \bbL\an_{\tZ} \to  \bbL\an_{\tZ/ Y} \to \cF[1]  ,
        \]
    in the \infcat $\pro(\Coh^+(\tZ))$.
\end{cor}

\begin{proof}
    The statement holds true in the case where $\tZ \in \anNil_{/Z}$ and $Y \in \anNil_{/X}$, by the universal property of the relative cotangent complex.
    The general case is reduced to the previous one by a standard argument with ind-objects in $\dAnSt_k$.
\end{proof}

\subsection{Non-archimedean nil-descent for almost perfect complexes}
In this \S, we prove that the \infcat $\Coh^+(X)$, for $X \in \dAnk$ satisfies nil-descent with respect to morphims $Y \to X$, which exhibit
the former as an analytic formal moduli problem over $X$.

\begin{prop} \label{prop:nil_descent_for_Coh^+}
    Let $f \colon Y \to X$, where $X \in \dAnk$ and $Y \in \anFMP_{/X}$. Consider the \v{C}ech nerve
    $Y^\bullet \colon \mathbf \Delta \op \to \dAnSt_k$ associated to $f$.
    Then the natural functor
        \[
            f_\bullet^* \colon \Coh^+(X) \to \lim_{\bDelta }(\Coh^+(Y^\bullet)),  
        \]
    is an equivalence of \infcats.
\end{prop}

\begin{proof}
    Consider the natural equivalence of derived $k$-analytic stacks
        \[
            Y \simeq \colim_{Z \in (\anNil^\cl_{/X})_{/Y}}  Z.
        \]
    Then, by definition one has a natural equivalence
        \[
            \Coh^+(Y) \simeq \lim_{Z \in (\anNil^\cl_{/X})_{/Y}} \Coh^+(Z),  
        \]
    of \infcats. In particular, since totalizations commute with cofiltered limits in $\Catinf$, it follows that we can suppose from
    the beginning that $Y \simeq Z$ for some $Z \in \anNil_{/X}$. In this case, the morphism $f \colon Y \to X$ is affine. In particular, the fact that
    $\Coh^+(-)$ satisfies descent along admissible open immersions, combined with \cref{lem:affine_morphisms_are_compatible_with_Zariski_localization_on_the_base} we further reduce ourselves
    to the case where both $X$ and $Y$ are derived $k$-affinoid spaces. In this case, by Tate acyclicity
    theorem it follows that letting $A \coloneqq \Gamma(X, \cO_X \alg)$ and $B \coloneqq \Gamma(Y, \cO_Y\alg),$ the pullback functor $f^*$ can be identified with
    the usual base change functor
        \[
            \Coh^+(A) \to \Coh^+(B).  
        \]
    In this case, it follows that $B$ is nil-isomophic to $A$. Moreover, since the latter are noetherian derived $k$-algebras
    the statement of the proposition follows due to \cite[Theorem 3.3.1]{preygel_Leistner_Mapping_stacks_properness}.
\end{proof}

We now deduce \emph{pseudo-pro-nil-descent} for moprhisms of the form $Y \to X$, which exhibit $Y$ as an analytic formal moduli problem over $X$:

\begin{cor} \label{cor:pseudo_nil-descent_for_pro_Coh^+}
    Let $X \in \dAnk$ and $f \colon Y \to X$ a morphism in $\dAnSt_k$ which exhibits $Y$ as an analytic formal moduli problem over $X$.
    Then the natural functor 
        \[
            f_\bullet^* \colon \pro(\Coh^+ (X)) \to \lim_{\bDelta}(\pro^\mathrm{ps}\Coh^+(Y^\bullet/X)),
        \]
    is fully faithful, where $Y^\bullet$ denotes the \v{C}ech nerve associated to the morphism $f$. Moreover, the essential image of the functor $f_\bullet^*$ identifies canonically with the full subcategory
        \[
            \lim_{\bDelta}'\pro^\mathrm{ps}(\Coh^+(X)) \subseteq \lim_{\bDelta} \pro^\mathrm{ps}(\Coh^+(Y^\bullet/ X)),
        \]
    spanned by those $\{\cF_{i, [n]} \}_{i \in I\op, [n]} \in \lim_{\bDelta}(\pro^\mathrm{ps}(\Coh^+(Y^\bullet/X)))$, for some filtered \infcat $I$, which belong to the essential image of the natural
    functor
        \[
            \lim_{\bDelta} \Fun(I\op, \Coh^+(Y^\bullet/X)) \to \lim_{\bDelta} \pro^\mathrm{ps}(\Coh^+(Y^\bullet/X)) . 
        \]
\end{cor}

\begin{proof}
    By the very definition of the \infcat $\pro^\mathrm{ps}(\Coh^+(Y))$, we reduce ourselves as in
    \cref{prop:nil_descent_for_Coh^+} to the case where $Y = S$, for some $S \in \anNil_{/X}$. In this case, it follows readily from
    \cref{prop:nil_descent_for_Coh^+} that the natural functor
        \[
            f_\bullet^* \colon \pro(\Coh^+(X)) \to \lim_{\bDelta} \pro^\mathrm{ps}(\Coh^+(S^\bullet/X)),  
        \]
    is fully faithful. We now proceed to prove the second claim of the corollary. Notice that,
    \cref{lem:f^*_admits_a_right_adjoint_whenever_f_is_nil-iso} implies that there exists a well defined right adjoint
        \[
            f_* \colon \Coh^+(S) \to \Coh^+(X),  
        \]
    to the usual pullback functor $f^* \colon \Coh^+(X) \to \Coh^+(S)$. We can extend the right adjoint $f_*$ to a well defined functor
        \[
            f_* \colon \pro(\Coh^+(S)) \to \pro(\Coh^+(X)),  
        \]
    which commutes with cofiltered limits. For this reason, we have a well defined functor
        \[
            f_{\bullet, *} \colon \lim_{\bDelta}(\pro(\Coh^+(S^\bullet/X))) \to \pro(\Coh^+(X)),  
        \]
    which further commutes with cofiltered limits. Denote by $\cD \coloneqq \lim_{\bDelta}(\pro(\Coh^+(S^\bullet/X)))$.
    We claim that $ f_{\bullet, *}$ is a right adjoint to $f_\bullet^*$ above.
    Indeed, given
    any $\{ \cF_i \}_{i \in I\op}  \in \pro(\Coh^+(X))$ and $\{ \cG_{j, [n]} \}_{j \in J_{[n]} \op, [n] \in \bDelta \op} \in \lim_{\bDelta}(\pro(\Coh^+(S^\bullet/X)))$,
    we compute
        \begin{align*}
            \Map_{\cD}( f_\bullet^*( \{ \cF_i \}_{i \in I\op}), \{ \cG_{j, [n]} \}_{j \in J\op_{[n]}, [n] \in \bDelta \op}) & \simeq \lim_{[n] \in \bDelta} \Map_{\pro^\mathrm{ps}(\Coh^+(S^{[n]}))} \bigg( \{f_{[n]}^\bullet(\cF_i) \}_{i \in I\op},  \{ \cG_{i, [n]} \}_{i \in I\op_{[n]}} \bigg) \\
            \lim_{[n] \in \bDelta \op} \lim_{j \in J\op_{[n]}} \mathrm{colim}_{i \in I} \Map_{\Coh^+(S^{[n]})} ( f_{[n]}^*(\cF_i),   \cG_{i, [n]}) & \simeq \lim_{[n] \in \bDelta} \lim_{j \in J\op_{[n]}} \colim_{i \in I} \Map_{\Coh^+(X)} (\cF_i, f_{[n], *}(\cG_{i, [n]})) \\
            \lim_{[n] \in \bDelta \op} \Map_{\pro(\Coh^+(X))}( \{\cF_i\}_{i \in I\op}, \{f_{[n], *}(\cG_{i, [n]}) \}_{i \in I_{[n]} \op}) & \simeq \Map_{\pro(\Coh^+(X))}( \{\cF_i\}_{i \in I\op}, \lim_{[n] \in \bDelta} \{f_{[n], *}(\cG_{i, [n]}) \}_{i \in I_{[n]} \op}),
        \end{align*}
    as desired. 
    It is clear that the functor $f_{\bullet
    }^*$ above factors through the full subcategory
        \[
            \lim_{\bDelta}'(\pro^\mathrm{ps}( \Coh^+(S^\bullet / X)))  \subseteq \lim_{\bDelta}(\pro(\Coh^+(S^\bullet/X))).
        \]
    For this reason, the pair $(f_\bullet^*, f_{\bullet, *})$ restricts to a well defined adjunction
        \[
            f^*_{\bullet} \colon \pro(\Coh^+(X))  \rightleftarrows  \lim_{\bDelta}'(\pro(\Coh^+(S^\bullet/ X))) \colon f_{\bullet, *}.  
        \]
    In order to conclude, we will show that the functor
        \[
            f_{\bullet, *} \colon  \lim_{\bDelta}'(\pro(\Coh^+(S^\bullet/X))) \to \pro(\Coh^+(X)) ,  
        \]
    is conservative. Since both the \infcats $\pro(\Coh^+(X))$ and $\lim_{\bDelta }'(\pro^\mathrm{ps}(\Coh^+(YS\bullet/X)))$ are stable, we are reduced to prove that given any
        \[
            \{ \cG_{i, [n]} \}_{i \in I\op} \in \lim_{\bDelta}'(\pro^\mathrm{ps}(\Coh^+(S^\bullet/X))),  
        \]
    such that 
        \begin{equation} \label{eq:conservativity_of_lim_f_bullet_*}
            \lim_{[n] \in \bDelta } f_{\bullet, *}( \{ \cG_{i, [n]} \}_{i \in I\op})  \simeq 0,
        \end{equation}
    we necessarily have
        \[
            \{ \cG_{i, [n]} \}_{i \in I\op} \simeq 0  ,
        \]
    in $\lim_{\bDelta}'(\pro^\mathrm{ps}(\Coh^+(S^\bullet / X)))$.
    Assume then \eqref{eq:conservativity_of_lim_f_bullet_*}. Under our hypothesis, for each index $i \in I$, the object
    $\{ \cG_{i, [n]} \}_{ [n] \in \bDelta}$ satisfies descent datum and thanks to \cref{prop:nil_descent_for_Coh^+} it produces a uniquely well defined object
        \[
            \cG_i \in \pro(\Coh^+(X)),  
        \]
    such that for every $[n] \in \bDelta$, one has a natural equivalence of the form
        \begin{align*}
            f_{[n]}^* (\cG_i) & \simeq \cG_{i , [n]} \\
                              & \in \Coh^+(S^{[n]}).
        \end{align*}
    We deduce then that  
        \begin{align*}
            f^*_{\bullet}(\{ \cG_i \}_{i \in I\op}) & \simeq \{\cG_{i, [n]} \}_{i \in I\op, [n]} \\
                                                    & \simeq 0,  
        \end{align*}
    in $\lim'_{\bDelta} \pro^\mathrm{ps}(\Coh^+(S^\bullet/ X))$, as desired.
\end{proof}

We now use the pseudo-pro-nil-descent for $\pro(\Coh^+(X))$ to compute relative analytic cotangent complexes of
analytic formal moduli problems over $X$:

\begin{cor} \label{cor:analytic_relative_cotangent_complex_defines_cartesian_sections_in_the_totalization}
    Let $f \colon Z \to X$ be a morphism in $\dAnk$. Suppose we are given a pullback square
        \[
        \begin{tikzcd}
            \tZ \ar{r}{h} \ar{d} & Y \ar{d} \\
            Z \ar{r} & X,
        \end{tikzcd}
        \]
    in the \infcat $\dAnSt_k$, where $( Y \to X) \in \anFMP_{/X}$ and $(g \colon\tZ \to Z) \in \anFMP_{/Z}$. Then the lax-object
        \[
            \{ \bbL\an_{\tZ^{[n]} / Y^{[n]}} \} \in \underset{\mathrm{lax}, \bDelta}{\lim}(\pro^\mathrm{ps}(\Coh^+(\tZ^\bullet/Z))),  
        \]
    defines a cartesian section
        \[
            \{ \bbL\an_{\tZ^{[n]}/ Y^{[n]}} \} \in \lim_{\bDelta}\pro(\Coh^+((\tZ)^\bullet /Z)),
        \]
    which belongs to the essential image of the natural functor
        \[
            g^*_\bullet \colon \pro(\Coh^+(Z)) \to \lim_{\bDelta}\pro(\Coh^+(\tZ^\bullet / Z)).
        \]
\end{cor}

\begin{proof}
    We first show that the object
        \[\{ \bbL\an_{(\tZ)^{[n]}/ Y^{[n]}} \}  \in \underset{\mathrm{lax}, \bDelta}{\lim}\pro^\mathrm{ps}(\Coh^+((\tZ)^\bullet/ Z)),\]
    defines a cartesian section in
        \[
            \lim_{\bDelta} \pro(\Coh^+((\tZ)^\bullet / Z)).  
        \]
    In order to show this assertion, it is sufficient to prove for every $[n ] \in \bDelta$, that we have a natural equivalence
        \[
            p_{i, n, n+1}^*(\bbL\an_{\tZ^{[n]} / Y^{[n]}}) \simeq \bbL\an_{\tZ^{[n+1]}/ Y^{[n+1]}},  
        \]
    in the \infcat $\pro^\mathrm{ps}(\Coh^+((\tZ)^{[n+1]}))$, for each projection morphism
        \[
            p_{i, n, n+1} \colon \widetilde{Z}^{[n+1]} \to \widetilde{Z}^{[n]},
        \]
    in the \v{C}ech nerve associated to the morphism $g$.
    The latter claim is an immediate consequence of the base change property for the relative analytic
    cotangent complex whenever $Y \in \anNil_{/X}$ (and thus so do $\tZ \in \anNil_{/Z})$), c.f.
    \cite[Proposition 5.12]{Porta_Yu_Representability}. In the general case where $Y \in \anFMP_{/X}$, we reduce to the previous case
    by combining \cref{prop:required_conditions_for_formal_moduli_problems} with the observation that filtered colimits
    commute with finite limits in the \infcat $\dAnSt_k$. 

    We now prove the second assertion of the Corollary: thanks to the characterization of the essential image of natural functor
        \[
            g^*_{\bullet} \colon \pro(\Coh^+(Z)) \to \lim_{\bDelta} \pro^\mathrm{ps}(\Coh^+((\tZ)^\bullet /Z)),  
        \]
    provided in \cref{cor:pseudo_nil-descent_for_pro_Coh^+}, we are reduced to show that for each $[n] \in \bDelta$, we have a natural
    equivalence of pro-objects
        \[
            \bbL\an_{\tZ^{[n]} / Y^{[n]}} \simeq \{ \bbL\an_{\tS^{[n]}/ S^{[n]}} \}_{\tS \in (\anNil_{/ Z})_{/\tZ}, \ S \in (\anNil_{/X})_{/ Y}},
        \]
    where $\tilde{S} \coloneqq S \times_X Z$, for each $S \in (\anNil^\cl_{/ X})_{/ Y}$.The latter statement follows readily from the first part of the proof
    combined with a standard inductive argument.
\end{proof}

\subsection{Non-archimedean formal groupoids} Let $X \in \dAn_k$.
We start with the definition of the notion of \emph{analytic formal groupoids over $X$}:

\begin{defin}We denote by
    $\mathrm{AnFGrpd}(X)$ the full subcategory of the \infcat of simplicial objects
        \[
            \Fun( \bDelta \op, \anFMP_{/X}),
        \]
    spanned by those objects $F \colon \mathbf \Delta \op \to \anFMP_{/X}$ satisfiying the following requirements:
        \begin{enumerate}
            \item $F([0]) \simeq X$ ;
            \item For each $n \ge 1$, the morphism
                \[
                    F([n]) \to F([1]) \times_{F([0])} \dots \times_{F([0])} F([1])  ,
                \]
            induced by the morphisms $s^i \colon [1] \to [n]$ given by $(0,1) \mapsto (i, i+1)$, is an equivalence
            in $\anFMP_{/X}$.
        \end{enumerate}
    We shall refer to objects in $\anFGrpd(X)$ as \emph{analytic formal groupoids over $X$}.
\end{defin}

\begin{rem}
    Note that \cref{prop:required_conditions_for_formal_moduli_problems} implies that fiber products exist in $\anFMP_{/X}$. Therefore, the previous
    definition is reasonable.
\end{rem}

\begin{construction} \label{const:formal_completion_construction_Phi} Thanks to \cref{formal_moduli_under_induce_formal_moduli_over_via_base_change},
there exists a well defined functor $\Phi \colon \anFMP_{X/} \to \anFGrpd(X)$ given by the formula
    \[
        (X \xrightarrow{f} Y ) \in \anFMP_{X/ } \mapsto Y^\wedge_X \in \anFGrpd(X),
    \]
where we shall denote by $Y^\wedge_X \in \anFGrpd(X)$ the \v{C}ech nerve of the morphism $f$ computed in $\anFMP_{/ X}$.
The latter defines a formal groupoid over $X$ admitting
    \[
    \begin{tikzcd}
      \dots \ar[r, shift left=2] \ar[r, shift left=0.75] 
      \ar[r, shift left=-0.75] \ar[r, shift left=-2]
      & X \times_Y X \times_Y X \ar[r, shift left=-1] \ar[r, shift left=1] \ar[r] 
      & X \times_Y X \ar[r, shift left=-0.5ex] \ar[r, shift left=0.5ex] 
      & X 
    \end{tikzcd},
    \] 
as simplicial presentation. We shall denote the later simplicial object simply by $X^{\times_Y \bullet} \in \anFGrpd(X)$.
\end{construction}

\begin{rem}
    Let $(X \to Y) \in \anFMP_{X/ }$ and consider the corresponding analytic formal groupoid $Y^\wedge_X \in \anFMP(X)$. The diagonal morphism
        \[
            \Delta \colon X \to X \times_Y X,   
        \]
    in $\anFMP_{X/ }$ induces a well defined morphism
        \[
            X \to X^{\times_Y \bullet} ,  
        \]
    in  $\Fun(\bDelta \op, \anFMP_{X/ })$, where $X^{\times_Y \bullet}$ denotes the \v{C}ech nerve of the morphism $X \to Y$
    computed in the \infcat $\anFMP_{X/ }$.
\end{rem}

\begin{construction} \label{const:formal_classifying_stack_construction}
    Let $\cG \in \anFGrpd(X)$. Consider the \emph{classifying derived $k$-analytic stack}, $\rB_X(\cG)^{\mathrm{pre}} \in \dAnSt_k$, obtained as the geometric realization
    of the simplicial object $\cG$, regarded naturally as a functor
        \[
            \cG \colon \bDelta \op \to \dAnSt_k,
        \]
    via the natural composite $\anFMP_{/ X} \to (\dAnSt_k)_{/ X} \to \dAnSt_k$.
    Given any $Z \in \dAfd_k$, the \emph{space of $Z$-points of $\rB_X(\cG)^{\mathrm{pre}}$},
        \[
            \rB_X(\cG)^\mathrm{pre}(Z),  
        \]
    can be identified with the space whose objects correspond to the datum of:
        \begin{enumerate}
            \item A morphism $\tZ \to X$, where $\tZ \in \dAnSt_k$, such that
                \[
                    \tZ \simeq Z \times_{\rB_X(\cG)^\mathrm{pre}} X;  
                \]
            \item A morphism of groupoid-objects
                \[\tZ \times_Z \tZ \to \cG,\]
                in the \infcat $\dAnSt_k$.
        \end{enumerate}
    We now define $\rB_X(\cG) \to \rB^\mathrm{pre}_X(\cG)$ as the sub-object spanned by those connected components of $\rB_X(\cG)^\mathrm{pre}$ corresponding to
    morphisms $\tZ \to Z$ in \cref{const:formal_classifying_stack_construction} (1)
    which exhibit $\tZ \in \anFMP_{/Z}$. Denote by
        \[
            \mathrm{can} \colon \rB_X(\cG) \to \rB_X(\cG)^\mathrm{pre},  
        \]
    the canonical morphism. It follows from the construction that the natural morphism    
        \[
            X \to \rB_X(\cG)^\mathrm{pre},  
        \]
    factors as $X \to \rB_X(\cG) \xrightarrow{\mathrm{can}} \rB_X(\cG)^\mathrm{pre}$.
\end{construction}

We are able to prove that the object $\rB_X(\cG)$ admits a deformation theory:

\begin{lem} \label{lem:B_X(G)_is_formal}
    The natural morphism $X \to \rB_X(\cG)$ exhibits the latter as an object in the \infcat $\anFMP_{X/}$ of analytic formal moduli problems
    under $X$.
\end{lem}

\begin{proof}
    Thanks to \cref{prop:sufficient_conditions_for_a_prestack_to_be_equiv_to_an_analytic_FMP} it suffices to prove that $\rB_X(\cG)$ is
    infinitesimally cartesian and it admits furthermore a pro-cotagent complex. The fact that $\rB_X(\cG)$ is infinitesimally cartesian follows from the
    modular description
    of $\rB_X(\cG)$ combined with the fact that $\cG$ is infinitesimally cartesian, as well. Similarly, $\rB_X(\cG)$ being nilcomplete follows again from its
    modular description combined with the fact that analytic formal moduli problems are nilcomplete.

    We are thus required to show that $\rB_X(\cG)$ admits
    a \emph{global} pro-cotangent complex. Let $Z \in \dAnk$ and suppose we are given an arbitrary morphism
        \[
            q \colon Z \to \rB_X(\cG),    
        \]
    in the \infcat $\dAnSt_k$.
    Thanks to \cref{cor:pseudo_nil-descent_for_pro_Coh^+} combined with \cref{cor:analytic_relative_cotangent_complex_defines_cartesian_sections_in_the_totalization}
    it follows that the object
        \[
            \{ \bbL\an_{\tZ^{[n]} / \cG^{[n]}} \}_{[n] \in \bDelta} \in \lim_{\bDelta} \pro^\mathrm{ps}(\Coh^+(\tZ^\bullet / Z)), 
        \]
    defines a well defined object $ (\bbL^{\mathrm{an}}_{Z/ \rB_X(\cG)})' \in \pro(\Coh^+(Z))$. Moreover, it is clear that there exists a natural morphism
        \[
            \theta \colon \bbL\an_Z \to (\bbL^{\mathrm{an} '}_{Z /  \rB_X(\cG)})',
        \]
    in the \infcat $\pro(\Coh^+(Z))$, as this holds in each level of the totalization. Let
        \[
            q^* (\bbL^{\mathrm{an}}_{\rB_X(\cG)})'  \coloneqq \mathrm{fib}(\theta),
        \]
    computed in the \infcat $\pro(\Coh^+(Z))$. We claim that $q^* (\bbL^{\mathrm{an} }_{\rB_X(\cG)})'$ identifies with the analytic cotangent
    complex of $\rB_X(\cG)$ at the point $q \colon Z \to \rB_X(\cG)$. Let
        \[
            Z \hookrightarrow Z',  
        \]
    denote a square-zero extension which corresponds to a certain analytic derivation
        \[
            d \colon \bbL\an_Z \to \cF[1],  
        \]
    for some $\cF \in \Coh^+(Z)^{\ge 0}$. Using \cref{cor:construction_of_square_zero_extensions_for_analytic_FMP_using_univ_property_of_cotangent_complex}
    we deduce that the space of cartesian squares of the form 
        \[
        \begin{tikzcd}
            \tZ \ar{r} \ar{d} & \tZ' \ar{d} \\
            Z \ar{r} & Z'
        \end{tikzcd}
        \]
    where the morphism $\tZ \to \tZ'$ is a square-zero extension in the \infcat $\dAnSt_k$ is equivalent to the space of factorizations
        \[
            d \colon g^*  \bbL\an_Z \to \bbL\an_{\tZ} \xrightarrow{{d'}} g^*(\cF)[1],
        \]
    in the \infcat $\pro(\Coh^+(\tZ))$. Apply the same reasoning to the each object in the \v{C}ech nerve
        \[
            \tZ^\bullet \to Z.  
        \]
    Furthermore, \cref{cor:universal_property_of_relative_cotangent_complex_for_morphisms_between_analytic_FMP}
    implies that the space of factorizations
        \[
            \tZ^\bullet \to (\tZ')^\bullet \to \cG^\bullet,  
        \]
    identifies with the space of factorizations
        \[
            d' \colon \bbL\an_{\tZ^\bullet} \to \bbL\an_{\tZ^\bullet/ \cG} \to g^*(\cF)[1],
        \]  
    in the \infcat $\pro(\Coh^+(\tZ))$. \cref{cor:analytic_relative_cotangent_complex_defines_cartesian_sections_in_the_totalization} implies that the above factorization space can be identified with the space
    of factorizations
        \[
            d \colon \bbL\an_Z \to (\bbL^{\mathrm{an}}_{Z/ \rB_X(\cG)})' \to \cF[1].
        \]
    This implies that $(\bbL^{\mathrm{an}}_{Z/ \rB_X(\cG)})'$ satisfies the universal property of the relative analytic pro-cotangent complex
    associated to $X \to \rB_X(\cG)$, as desired.
\end{proof}

\begin{thm} \label{thm:Phi_is_an_equivalence}
    The functor $\Phi \colon \anFMP_{/X} \to \anFGrpd(X)$ of \cref{const:formal_completion_construction_Phi} is an equivalence of \infcats.
\end{thm}

\begin{proof} Let $\cG \in \anFGrpd(X)$. Thanks to \cref{lem:B_X(G)_is_formal}, we have a well defined functor
        \[
            \rB_X(-) \colon \anFGrpd(X) \to \anFMP_{X/ },  
        \]
    given on objects by the association $\cG \in \anFGrpd(X) \mapsto \rB_X(\cG) \in \anFMP_{X/ }$.
    Thanks to (1) in \cref{const:formal_classifying_stack_construction} it follows that one has a canonical equivalence
        \[
            X \times_{\rB_X(\cG)} X \simeq \cG,
        \]
    in $\dAnSt_k$. This shows that the construction
        \[
            \rB_X(\cG) \colon \anFGrpd(X) \to \anFMP_{X/},  
        \]
    is a right inverse to $\Phi$. As a consequence the functor $\Phi$ is essentially surjective. Thanks to \cref{prop:conservativity_of_relative_an_cot_complex}
    we are reduced to show that the canonical morphism
        \[
            Y \to \rB_X(X \times_Y X),  
        \]
    induces an equivalence on the associated relative analytic cotangent complexes. The claim is an immediate consequence of
    the description of $\bbL\an_{X/ \rB_X(X \times_Y X)}$ provided in \cref{lem:B_X(G)_is_formal} together with
    \cref{cor:analytic_relative_cotangent_complex_defines_cartesian_sections_in_the_totalization}.
\end{proof}

\begin{rem}
    It follows from \cref{thm:Phi_is_an_equivalence} that given $(X \to Y) \in \anFMP_{X/ }$, we can identify the latter with the
    geometric realization of the associated \v{C}ech nerve, the latter regarded as a simplicial object in $\anFMP_{X/ }$ via the
    diagonal morphisms.
\end{rem}

\subsection{The affinoid case} Let $X \in \dAfdk$ denote a derived $k$-affinoid space.
Thanks to derived Tate aciclycity theorem, cf. \cite[Theorem 3.1]{Porta_Yu_Derived_Hom_spaces}
the \emph{global sections functor}
    \[
        \Gamma \colon \Coh^+(X) \to  \Coh^+(A),
    \]
where $A \coloneqq \Gamma(X, \cO_X \alg) \in \CAlg_k$, is an equivalence of \infcats. Since ordinary $k$-affinoid algebras are Noetherian, we deduce that
$A \in \CAlg_k$ is a Noetherian derived $k$-algebra.

\begin{notation}
    Let $X \in \dAfd_k$ and $A \coloneqq \Gamma(X, \cO_X)$. We denote by
        \[
            \mathrm{FMP}_{/\Spec A} \in \Catinf,  
        \]
    the \infcat of \emph{algebraic formal moduli problems over $\Spec A$}, (c.f. \cite[Definition 6.11]{Porta_Yu_NQK}).
\end{notation}

\begin{thm}$\mathrm{(}$\cite[Theorem 6.12]{Porta_Yu_NQK}$\mathrm )$ \label{thm:equivalence_between_analytic_and_algebraic_formal_moduli_problems_over}
    Let $X \in \dAfdk$ and $A \coloneqq \Gamma(X, \cO_X\alg) \in \CAlg_k$. Then the induced functor
        \[
            (-) \an \colon \mathrm{FMP}_{/ \Spec A} \to \anFMP_{/X},
        \]
    is an equivalence of \infcats.
\end{thm}

As an immediate consequence, we obtain the following result:

\begin{cor} \label{cor:equivalence_between_pointed_analytic_and_algebraic_formal_moduli_problems_over}
    Let $X \in \dAfd_k$ and $A \coloneqq \Gamma(X, \cO_X\alg)$. Then one has an equivalence of \infcats
        \[
            \mathrm{FMP}_{\Spec A/ /\Spec A} \to \anFMP_{X/ /X} ,
        \]
    of pointed algebraic formal moduli problems over $\Spec A$ and pointed analytic formal moduli problems over $X$, respectively.
\end{cor}

\begin{proof}
    It is an immediate consequence of \cref{thm:equivalence_between_analytic_and_algebraic_formal_moduli_problems_over}. Indeed, equivalences of \infcats with final objects
    induce natural equivalences on the associated \infcats of pointed objects.
\end{proof}

\begin{lem} \label{lem:Fgroupoids_to_pointed_objects_is_conservative_and_commutes_with_sifted_colimits}
    Consider the natural functor
        \[
            F \colon \anFGrpd(X) \to \anFMP_{X/ /X},  
        \]
    given on objects by the formula $\cG \in \anFGrpd(X) \mapsto \cG([1]) \in \mathrm{Ptd}(\anFMP_{/ X})$. Then $F$ is conservative and
    commutes with sifted colimits.
\end{lem}

\begin{proof}
    The proof of \cite[Corollary 5.2.2.4]{Gaitsgory_Study_II} applies in our setting
    to show that the functor $F \colon \anFGrpd(X) \to \mathrm{Ptd}(\anFMP_{/ X})$
    commutes with sifted colimits. We are reduced to prove that $F$ is also conservative.
    Let $f \colon \cG \to \cG'$ be a morphism in $\anFGrpd(X)$. The equivalence of \infcats 
        \[\rB_X(\bullet) \colon \anFGrpd(X) \to \anFMP_{X/ },\]
    of \cref{thm:Phi_is_an_equivalence} implies that
    the morphism $f$ can be obtained as the \v{C}ech nerve of a morphism
        \[
            \widetilde{f} \colon \rB_X(\cG) \to \rB_X(\cG'),  
        \]
    in $\anFMP_{X/ }$. For this reason, for every $[n] \in \bDelta$, the morphism
        \[f_{[n]} \colon \cG([n]) \to \cG'([n]),\]
    is obtained as (an iterated) pullback of
        \[f_{[1]} \colon \cG([1]) \to \cG'([1]).\]
    Therefore, under the assumption
    that $F(F) \simeq f_{[1]}$ is an equivalence in $\anFMP_{X/ /X}$ we deduce that $f$ itself
    must be an equivalence of analytic formal groupoids, as it is each of its components.
\end{proof}



\begin{rem} Let $(X \xrightarrow{f} Y) \in \anFMP_{X/ }$.
Consider the diagram
    \[
    \begin{tikzcd}
        X \ar{r}{\Delta}  \ar[rd, equal] & X \times_Y X \arrow[d, "p_0", shift left=2] \arrow{d}[swap]{p_1} \\
        &   X  
    \end{tikzcd}
    \]
in the \infcat $\dAnSt_k$, where $\Delta \colon X \to X \times_Y X$ denotes the usual \emph{diagonal embedding}. We then obtain a natural
fiber sequence associated to the above diagram of the form
    \begin{equation} \label{eq:fiber_sequence_for_rel_cot_complexes_along_diagonal_and_projections}
        \Delta^* \bbL\an_{X \times_Y X/ X} \to \bbL\an_{X/X} \to \bbL\an_{X/ X \times_Y X}.
    \end{equation}
Notice further that by \cite[Proposition 5.12]{Porta_Yu_Representability} one has an equivalence
    \[
        \bbL\an_{X \times_Y X/ X} \simeq p_i^* \bbL\an_{X/Y},  
    \]
for $i =0, 1$. We further deduce that
    \[\Delta^* \bbL\an_{X \times_Y X/ X} \simeq \bbL\an_{X/Y},\]
in the \infcat $\pro(\Coh^+(X))$. Moreover, since $\bbL\an_{X/ X} \simeq 0$, we obtain from the fiber sequence \eqref{eq:fiber_sequence_for_rel_cot_complexes_along_diagonal_and_projections}
a natural equivalence
    \[
        \bbL\an_{X/ X \times_Y X} \simeq \bbL\an_{X / Y} [1],    
    \]
in $\pro(\Coh^+(X))$. Moreover, we can identify the $\bbL\an_{X \times_Y X / X}$ with the pro-object  
    \[
        \bbL\an_{X/ X \times_Y X } \simeq \{ \bbL\an_{X/ X \times_S X} \}_{S \in \anNil_{X//Y}^\cl} ,
    \]
where $\bbL\an_{X/ X\times_S X} \in \Coh^+(X)$ denotes the relative analytic cotangent complex associated to the closed embedding
    \[
        X \to X \times_S X,  
    \]
for $S \in \anNil_{X/Y}^\cl$. Thanks to \cite[Corollary 5.33]{Porta_Yu_Representability} we deduce that
    \[
        \bbL\an_{X/X\times_SX} \simeq \bbL_{A/ A \otimes_{B_S} A} ,
    \]
in $\Coh^+(A)$, where $B_S = \Gamma(S, \cO\alg_S)$, where we have a natural identification
    \[
        A \widehat{\otimes}_{B_S} A \simeq A \otimes_{B_S} A,
    \]
since the morphism $A \to B_S$ is finite.
\end{rem}

\begin{lem}
    Let $X \in \dAfd_k$ be a derived $k$-affinoid space. Then the derived $k$-algebra of global sections
        \[
            A \coloneqq \Gamma(X, \cO_X\alg),  
        \]
    admits a dualizing module (see \cite[Definition 4.2.5]{DAG-XIV} for the definition of the latter notion).
\end{lem}

\begin{proof}
    This is an immediate consequence of the following facts:
    \begin{enumerate}
        \item Every regular $k$-algebra $R$ admits a dualizing module, c.f. \cite[\href{https://stacks.math.columbia.edu/tag/0AWX}{Tag 0AWX}]{stacks-project};
        \item Every quotient of an algebra $R$ which admits a dualizing module admits itself a dualizing module, c.f. \cite[\href{https://stacks.math.columbia.edu/tag/0A7I}{Tag 0A7I}]{stacks-project}.
        \item The ordinary affinoid $k$-algebra $\pi_0(A)$ can be realized as a quotient of a Tate algebra on $n$-generators, $k \langle T_1, \dots, T_n \rangle$.
        The latter being regular, we deduce from the previous items that $\pi_0(A)$ itself admits a dualizing module (cf. \cite[\href{https://stacks.math.columbia.edu/tag/0AWX}{Tag 0AWX}]{stacks-project}).
        \item Thanks to \cite[Theorem 4.3.5]{DAG-XIV} it follows that $A$ itself admits a dualizing module.
    \end{enumerate}
    The proof is thus concluded.
\end{proof}

\begin{lem}
    Let $X \in \dAfd_k$ be a bounded derived $k$-affinoid space and let $\omega_A \in \Mod_A$ denote a dualizing module
    for $A \coloneqq \Gamma(X, \cO_X\alg)$. Then $\omega_A \in \Coh^b(A)$, where the latter denotes the
    full subcategory of (bounded) \emph{coherent} $A$-modules.
\end{lem}

\begin{proof}
    By definition it follows that for every $i \in \bbZ$, $\pi_i(\omega_A)$ is finitely generated as a discrete $\pi_0(A)$-module.
    We are thus reduced to show that $\omega_A$ is bounded. Thanks to \cite[Theorem 4.2.7]{DAG-XIV} it follows that
        \[
            \omega_A \simeq \underline{\Map}_{\Mod_A}(A, \omega_A),  
        \]
    is truncated. Moreover, by definition it follows that $\omega_A \in \Mod_A$ is of finite injective dimension. In particular, for every
    $j \ge 0$, we have that 
        \[
            \underline{\Map}_{\Mod_A}(\pi_j(A), \omega_A),
        \]
    is bounded. Under our assumption that $X \in \dAfd_k$ is bounded it follows that $A$ itself is a bounded derived $k$-algebra and
    therefore by a standard argument on Postnikov towers, we deduce that
        \[
            \underline{\Map}_{\Mod_A}(A, \omega_A) \simeq \omega_A,
        \]
    is bounded as well.
\end{proof}

\begin{rem} \label{rem:dualizing_module_on_truncations}
    Let $X \in \dAfd_k$ be a derived $k$-affinoid algebra. If $X$ is not bounded then neither it is $A \in \CAlg_k$. For this reason,
    a dualizing module $\omega_A \in \Mod_A$ is not in general bounded. Nonetheless, the latter is always truncated.
    Let $m \in \bbZ$ be such that $\omega_A \in \Mod_A$ is $m$-truncated. Thanks to the proof of \cite[Theorem 4.3.5]{DAG-XIV},
    we deduce that we have a natural
    equivalence
        \[
            \omega_A \simeq \colim_{n \ge 0} \omega_{\tau_{\le n}(A)},  
        \]
    in the \infcat $\Mod_A$, where for each $n \ge 0$, $\omega_{\tau_{\le n}(A)} \in \Mod_{\tau_{\le n}(A)}$ is a suitable $m$-truncated dualizing module for
    the $n$-th truncation $\tau_{\le n}(A)$.
\end{rem}

\begin{defin}
    Let $X \in \dAfd_k$ we define its \infcat of \emph{ind-coherent sheaves} as
        \[
            \Ind\Coh(X) \coloneqq \Ind(\Coh^b(A)).  
        \]
\end{defin}

\begin{defin}
    Let $X \in \dAfd_k$ be a derived $k$-affinoid space. Whenever $X$ is bounded we shall denote
    $\omega_X \in \Ind\Coh(X)$ the image of a fixed dualizing module for $A \coloneqq \Gamma(X, \cO_X\alg)$ under the natural inclusion
        \[
            \Coh^b(X) \subseteq \Ind\Coh(X).  
        \]
    In the unbounded case, we shall define 
        \[
            \omega_X \coloneqq \colim_{n \ge 0} \omega_{\rt_{\le n}(X)} \in \Ind\Coh(X),  
        \]
    where the $\{ \omega_{\rt_{\le n}(X)} \}_{n \ge 0}$ denotes a compatible sequence of
    dualizing modules for the consecutive truncations of $X$.
\end{defin}

\begin{notation}Let $X \in \dAfd_k$ be a derived $k$-affinoid space together with a dualizing module $\omega_X \in \Ind\Coh(X)$.
    We shall denote by 
        \[
            \bD^\mathrm{Serre}_{X} \colon \Ind(\Coh^b(X)\op) \to \Ind\Coh(X),
        \]
    the associated \emph{Serre duality functor}, see \cite[\S 2.4]{Antonio_Koszul}.
\end{notation}

\begin{rem}
    Let $(X \to Y) \in \anFMP_{X/ }$. The pro-cotangent complex $\bbL_{X/ Y} \an \in \pro(\Coh^+(X))$ can be naturally
    regarded as an object in the full subcategory $\pro(\Coh^b(X))$ of pro-objects of bounded almost perfect $\cO_X$-modules,
    c.f. \cite[Lemma 4.6]{Antonio_Koszul}.
\end{rem}

Notice that we have an equivalence of \infcats $\pro(\Coh^b(X)) \op \simeq \Ind(\Coh^b(X) \op)$. We now introduce the central notion
of the \emph{Serre tangent complex}:

\begin{defin}[Serre Tangent complex] Let $Y \in \anFMP_{X/}$. We define the \emph{relative analytic Serre tangent complex} of $X \to Y$ as the object
    \begin{align*}
        \bbT\an_{X/ Y}  & \coloneqq \bD^\mathrm{Serre}_{X}(\bbL\an_{X/Y}) \\
                        & \simeq \colim_{S \in \anNil^\cl_{X/ /Y}} \bD^\mathrm{Serre}_X(\bbL\an_{X/S}),
    \end{align*}
    in $\Ind\Coh(X)$.
\end{defin}

\begin{rem}
    The previous definition depends on the choice of a dualizing module for $X$. Nonetheless, given two different choices, these
    differ only
    by an invertible $A$-module, c.f. \cite[Proposition 2.4]{Antonio_Koszul}.
\end{rem}

The following result will play a major role in the study of the deformation to the normal bundle:

\begin{prop} \label{prop:conservativity_and_preservation_of_sifted_colimits_of_tangent_complex}
    The functor $\bbT\an_{X/ \bullet} \colon \anFMP_{X/} \to \QCoh(X)$ given on objects by the formula
        \[
            (X \to Y) \in \anFMP_{X/} \mapsto \bbT\an_{X/Y} \in \Ind\Coh(X), 
        \]
    is conservative and commutes with sifted colimits.
\end{prop}

\begin{proof}
    Observe that
    \cref{lem:Fgroupoids_to_pointed_objects_is_conservative_and_commutes_with_sifted_colimits} combined with
    \eqref{eq:fiber_sequence_for_rel_cot_complexes_along_diagonal_and_projections} imply that it suffices to prove that
    the right vertical functor in the diagram
        \begin{equation} \label{eq:comm_pointed_formal_moduli_problems_Tangent_and_analytification}
        \begin{tikzcd}
            \mathrm{FMP}_{\Spec A/ / \Spec A} \ar{r}{(-)\an_X}\ar{d}{\bbT_{\Spec A/ \bullet}} & \anFMP_{X/ / X} \ar{d}{\bbT\an_{X/ \bullet}} \\
            \Ind\Coh(A) \ar{r}{(-)\an} & \Ind\Coh(X),
        \end{tikzcd}
        \end{equation}
    commutes with sifted colimits and is further conservative. The diagram is commutative as a consequence of \cite[Lemma 6.9 (2)]{Porta_Yu_NQK}
    combined with the definition of Serre duality in $\Ind\Coh(X)$. Moreover, \cite[Corollary 4.31]{Antonio_Koszul} implies that
        \[
            \bbT_{\Spec A/ \bullet} \colon     \mathrm{FMP}_{\Spec A/ / \Spec A} \to \Ind\Coh(A),
        \]
    is both conservative and commutes with sifted colimits. Furthermore, the bottom horizontal functor in \eqref{eq:comm_pointed_formal_moduli_problems_Tangent_and_analytification}
    is an equivalence of \infcats, see for instance \cite[Theorem 4.5]{Porta_Yu_NQK}. Thanks to \cref{cor:equivalence_between_pointed_analytic_and_algebraic_formal_moduli_problems_over}, we further deduce
    that the upper horizontal functor in \eqref{eq:comm_pointed_formal_moduli_problems_Tangent_and_analytification} is an equivalence, and the result follows.
\end{proof}

\section{Non-archimedean Deformation to the normal bundle}

In this \S, we introduce the construction of the deformation to the normal cone in the setting of derived $k$-analytic geometry. This construction has
already been performed in the literature in the particular case of the natural inclusion morphism
    \[
        i \colon \trunc(X) \to X,  \quad X \in \dAnk,
    \]
c.f \cite{Porta_Yu_NQK}.
On the other hand, the algebraic situation is largely understood mainly due to \cite{Gaitsgory_Study_II}.

\subsection{General construction in the algebraic case} We start by recalling the definition of the simplicial object
    \[
        \mathrm{B}^\bullet_{\mathrm{scaled}} \in \Fun(\bDelta \op, \mathrm{dSt}_k),
    \]
introduced in
\cite[\S 9.2.2]{Gaitsgory_Study_II}. Namely, for each $[n] \in \bDelta$, the object $\mathrm{B}^{n}_\mathrm{scaled}$ is obtained by gluing $n+1$ copies of $\bbA^1_k$ together along
$0 \in \bbA^1_k$. Moreover, the transition morphisms
    \[
        p_{i, n} \colon  \mathrm{B}^{n+1}_\mathrm{scaled} \to \mathrm{B}^n_{\mathrm{scaled}}, \quad \mathrm{for} \ i \in \{0, \dots, n\} ,
    \]
collapse two given different irreducible components of $\rB^{n+1}_\mathrm{scaled}$ into $\rB^n_{\mathrm{scaled}}$. More explicitly, for $n=0$, we have
    \[
        \rB_\mathrm{scaled}^0 = \bbA^1_k,  
    \]
and, for $n = 1$, $\rB_\mathrm{scaled}^1 = \Spec k[x, y]/ (x^2 - y^2)$.

\begin{construction} \label{construction:def_to_the_normal_bundle}
    Let $f \colon X \to Y$ denote a morphism between locally geometric derived $k$-stacks in $\mathrm{dSt}^\laft_k$. Consider the formal completion of $Y$ on $X$ along
    the morphism $f$
        \[
            Y^\wedge_X \in \dSt_k  
        \]
    In \cite[\S 9.3]{Gaitsgory_Study_II}, the authors introduced the \emph{parametrized deformation to the normal bundle associated to $f \colon X \to Y$}, as the pullback
        \begin{equation} \label{eq:parametrized_deformation}
        \begin{tikzcd}
            \cD^\bullet_{X/Y} \ar{r} \ar{d} &   Y^\wedge_X \times \bbA^1_k \ar{d} \\
            \Map_{/ \bbA^1_k}(\mathrm{B}^\bullet_{\mathrm{scaled}}, X \times \bbA^1_k) \ar{r} & \Map_{/ \bbA^1_k}(\mathrm{B}^\bullet_{\mathrm{scaled}}, Y^\wedge_X \times \bbA^1_k),
        \end{tikzcd}
        \end{equation}
    where both $X \times \bbA^1_k$ and $Y^\wedge_X \times \bbA^1_k$ are considered as constant simplicial objects in the \infcat $\dSt_k$.
    As a consequence of \cite[Theorem 9.2.3.4]{Gaitsgory_Study_II}, it follows that
    each component of the simplicial object
    $\cD_{X/Y}^\bullet$ admits a deformation theory relative to $X$.
    
    In particular, the object $\cD^\bullet_{X/ Y}$ defines a formal groupoid over the stack
    $X \times \bbA^1_k$. Moreover, Theorem 5.2.3.4 in loc. cit. allows us to associate to $\cD^\bullet_{X/Y}$ a
    formal moduli problem under $X \times \bbA^1_k$
           \[
               \cD_{X/Y} \in \mathrm{FMP}_{X \times \bbA^1_k/ / Y^\wedge_X \times \bbA^1_k},
           \]
    obtained via the construction provided in \S 5.2.4 in loc. cit.
    More explicitly, the object $\cD_{X/ Y} \in \mathrm{FMP}_{X \times \bbA^1_k/ }$ is computed as the sifted colimit of the simplicial diagram
        \[
            \cD_{X/ Y}^\bullet \in \Fun(\bDelta \op, \mathrm{FMP}_{X \times \bbA^1_k/ })  ,
        \]
    the colimit being computed in $\mathrm{FMP}_{X \times \bbA^1_k/ }$. We can furthermore consider $\cD_{X/Y}$
    naturally as an object in $(\dSt^\laft_k)_{X \times \bbA^1_k/ / Y^\wedge_X \times \bbA^1_k}$.
\end{construction}

\begin{construction}
    Let $f \colon X \to Y$ a morphism of locally geometric derived $k$-stacks in $\dSt_k^\laft$. In \cite[\S 9.2.5]{Gaitsgory_Study_II} the authors
    constructed an explicit left-lax action of the monoid-scheme object $\bbA^1_k$ (with respect to \emph{multiplication})
    on the deformation to the normal bundle $\cD_{X/ Y}$. The above action
    implies that the structure sheaf of $\cD_{X/ Y}$ is equipped with a (negatively indexed) filtration. 
\end{construction}

\begin{notation}
    We shall denote by $p \colon \cD_{X/ Y} \to \bbA^1_k$ the natural composite morphism 
        \[
            \cD_{X/ Y} \to Y^\wedge_X \times \bbA^1_k \to \bbA^1_k,  
        \]
    in the \infcat $\dSt_k^\laft$.
\end{notation}

We shall now describe formal geometric properties of the deformation to the normal bundle $\cD_{X/Y} \in (\dSt^\laft_k)_{X \times \bbA^1_k/ /Y^\wedge_X \times \bbA^1_k}$:

\begin{prop} \label{prop:algebraic_properties_of_deformation} Let $f \colon X \to Y$ denote a morphism of locally geometric derived $k$-stacks in the \infcat $\dSt_k^\laft$.
    The following assertions hold:
    \begin{enumerate}
        \item The fiber of the morphism $p \colon \cD_{X/Y} \to \bbA^1_k$ at the fiber $\{0 \} \subseteq \bbA^1_k$ identifies naturally with
        the formal completion 
            \[\rT_{X/Y}[1]^\wedge \in \mathrm{FMP}_{X/ },\]
        of the shifted tangent bundle $\rT_{X/Y}[1] \to X$ along the
        zero section
            \[
                s_0 \colon X \to \rT_{X/Y}[1].  
            \]
        \item For $\lambda \in \bbA^1_k$ with $\lambda \neq 0$, the fiber $(\cD_{X/Y})_{\lambda}$ canonically identifies with the formal completion
            \[
                Y^\wedge_X \in \dSt^\laft_k,  
            \]
        of $X$ in $Y$ along the morphism $f$.
        \item (Hodge Filtration) There exists a natural sequence of morphisms
            \[
                X \times \bbA^1_k = X^{(0)} \to X^{(1)}  \to  \dots \to X^{(n)} \to \dots \to Y,
            \]
        admitting a deformation theory, in $(\dSt_k^\laft)_{X\times\bbA^1_k/ /Y\times\bbA^1_k}$. For each $n \ge 0$,
        the relative tangent complex
            \[
                \bbT_{X^{(n)} / Y }   \in \Ind\Coh(X^{(n)}),
            \]
        has first non-zero associated graded piece in degree $n$ and the latter identifies naturally with
            \[
                i_{n, *}^{\Ind\Coh} ( \Sym^{n+1} (\bbT_{X/ Y} [1])[-1])) \in \Ind\Coh(X^{(n)}).
            \]
        Dually, the pro-relative cotangent complex $\bbL_{X/ Y} \in \pro(\Coh(X^{(n)}))$ is equipped with a decreasing filtration, starting in degree $n$,
        and the corresponding associated $n$-th graded piece identifies naturally with
            \[
                i_{n, *}(\Sym^{n+1}(\bbL_{X/ Y}[-1])[1])  \in \pro(\Coh^b(X^{(n)})).
            \]
        \item Assume that $f \colon X \to Y$
        is a closed immersion of derived schemes locally almost of finite presentation then, for each $n \ge 0$, the morphism
            \[
                X^{(n)} \to X^{(n+1)},  
            \]
        has the structure of a square-zero extension associated to a canonical morphism
            \[d \colon \bbL_{X^{(n)}} \to \bbL_{X^{(n)}/ Y} \to i_{n, *} (\Sym^{n+1}(\bbL_{X/ Y}[-1])[1]),\]
        in the \infcat $\Coh^+(X^{(n)})$, where 
            \[
                i_n \colon X \times \bbA^1_k \to X^{(n)}
            \]
        denotes the structural morphism.
        In particular, if $X$ is a derived scheme almost of finite presentation, then so are the $X^{(n)}$, for each $n \ge 0$.
        \item There exists a natural morphism
            \[\colim_{n} X^{(n)} \to \cD_{X/Y},\]
        in the \infcat $\mathrm{FMP}_{X \times \bbA^1_k/ /Y^\wedge_X \times\bbA^1_k}$. Moreover, the latter is an equivalence
        of derived $k$-stacks almost of finite presentation.
    \end{enumerate}
\end{prop}

\begin{proof}
    Assertion (1) follows formally from \cite[\S 9, Proposition 2.3.6]{Gaitsgory_Study_II}. Assertion (2) follows from the observation that
    the fiber of the simplicial object $\cD_{X/ Y}^\bullet$ at $\lambda \neq 0$ can be identified with the \v{C}ech nerve of the (completion) of the morphism $f$
        \[
            X^{\times_{Y^\wedge_X} \bullet} \in \Fun(\bDelta \op, \dSt_k^\laft).
        \]
    Therefore, its colimit taken in $\mathrm{FMP}_{X \times \bbA^1_k/}$ agrees with the formal completion $Y^\wedge_X \in \mathrm{FMP}_{X \times \bbA^1_k/ }$. The first part of assertion (3) follows
    formally from \cite[\S 9, Theorem 5.1.3]{Gaitsgory_Study_II}. For each $n \ge 0$, denote by 
    \[i_n \colon X \times \bbA^1_k \to X^{(n)},\]
    the structural morphism. The latter is a proper morphism due to \cite[\href{https://stacks.math.columbia.edu/tag/0CYK}{Tag 0CYK}]{stacks-project}.

    The statement concerning the relative cotangent complex follows formally from
    \cite[\S 9, Theorem 5.1.3]{Gaitsgory_Study_II} by applying the Serre duality functor and combining
    \cite[Corollary 9.5.9 (b)]{Gaitsgory_IndCoh} with \cite[Proposition 3.1.3]{Gaitsgory_IndCoh} and \cite[Corollary 1.4.4.2]{Gaitsgory_Study_II}
    to produce a natural identification
        \[
            \bD^\mathrm{Serre}_X \big((i_{n, *}^{\Ind\Coh}(\Sym(\bbT_{X/ Y}[1])[-1]) \big)   \simeq i_{n, *} \big(\Sym(\bbL_{X/ Y}[-1])[1] \big),
        \]
    in the \infcat $\pro(\Coh^b(X))$.
    Assertion (5) is an immediate consequence of \cite[\S 9, Proposition 5.2.2]{Gaitsgory_Study_II}.
    We shall now deduce claim (4) of the Proposition:
    it follows from
    Theorem 9.5.1.3, in loc. cit.,
    that the (Serre) tangent complex
        \[
            \bbT_{X^{(n)}/ Y } \in \Ind\Coh(X^{(n)}),  
        \]
    admits a natural (increasing) filtration whose $n$-th piece identifies with 
        \[i_{n, *}^\mathrm{IndCoh}(\Sym^{n+1}(\bbT_{X/Y}[1]))[-1] \in \Ind\Coh(X^{(n)}).\]
    For this reason, we have a structural morphism
        \[
            i_{n, *}^\mathrm{IndCoh}(\Sym^{n+1}(\bbT_{X/Y}[1]))[-1] \to \bbT_{X^{(n)}/ Y},
        \]
    in the \infcat $\Ind\Coh(X^{(n)})$. Consider then the natural composite
        \[
            i_{n, *}^\mathrm{IndCoh}(  \Sym^{n+1}(\bbT_{X/ Y}[1])[-1])) \to \bbT_{X^{(n)}/ Y} \to \bbT_{X^{(n)}},
        \]
    in the \infcat $\Ind\Coh(X)$. Using the equivalence of \infcats between Lie algebroids in $\Ind\Coh(X)$ and the \infcat
    $\mathrm{FMP}_{X/ }$, c.f. \cite[\S 8.5]{Gaitsgory_Study_II}, one produces then a canonical morphism
        \[
            X^{(n)} \to X^{(n+1)},
        \]
    in the \infcat $\mathrm{FMP}_{X \times \bbA^1_k/ }$.
    
    In the case where $f \colon X \to Y$ is a closed immersion
    of derived schemes locally almost of finite presentation the latter construction can be adapted in terms of the
    (usual) cotangent complex
    formalism: thanks to \cite[Corollary 8.4.3.2]{Lurie_Higher_algebra}, the relative cotangent complex
        \[
            \bbL_{X/ Y} \in \Coh^+(X),  
        \]
    is $1$-connective. In particular, the shift $\bbL_{X/ Y}[-1]$ is $0$-connective.
    By applying the usual Serre duality functor (c.f. \cite[\S 9]{Gaitsgory_IndCoh}) we obtain a well defined fiber sequence
        \begin{equation} \label{eq:fiber_seq_cotangent_complex_for_Hodge_filtration2}
            \bbL_{X^{(n)}} \to \bbL_{X^{(n)}/ Y} \to i_{n,*}\big(  \Sym^{n+1}(\bbL_{X/ Y}[1])[-1]) \big) ,
        \end{equation}
    in the \infcat $\pro(\Coh^b(X))$. Moreover, under our assumptions, it follows that \eqref{eq:fiber_seq_cotangent_complex_for_Hodge_filtration}
    lies in the full subcategory 
        \[\Coh^+(X) \subseteq \pro(\Coh^b(X)),\]
    via the equivalence of \infcats provided in \cite[Corollary 1.4.4.2]{Gaitsgory_Study_II}.
    Since the right hand side of \eqref{eq:fiber_seq_cotangent_complex_for_Hodge_filtration2}
    is $1$-connective, it follows from \cite[\S 8, 5.5]{Gaitsgory_Study_II} that the extension
        \[
            X^{(n)} \to X^{(n+1)}, \quad \mathrm{for} \ n \ge 0,  
        \]
    can be identified with the (usual) square-zero extension of $X^{(n)}$ by means of the derivation \eqref{eq:fiber_seq_cotangent_complex_for_Hodge_filtration2}. This proves the second part
    of claim (3) and the proof of the proposition is thus concluded.
\end{proof}

\begin{rem} \label{rem:simplification_of_deformation_in_the_closed_immersion}
    Assume that $f \colon X \to Y$ is a closed immersion of derived schemes locally almost of finite presentation.
    Then \cref{construction:def_to_the_normal_bundle} can be simplified by taking directly $Y$ instead of the formal completion $Y^\wedge_X$
    in the defining pullback square for \eqref{eq:parametrized_deformation}. Indeed, we claim that the simplicial object
    obtained as the fiber product
        \[
        \begin{tikzcd}
            (\cD_{X/ Y}^\bullet) ' \ar{r} \ar{d} & Y \times \bbA^1_k \ar{d}\\
            \Map_{/ \bbA^1_k}(\rB^\bullet_\mathrm{scaled}, X\times \bbA^1_k) \ar{r} & \Map_{/ \bbA^1_k} (\rB^\bullet_\mathrm{scaled}, Y \times \bbA^1_k),
        \end{tikzcd}
        \]
    in $\Fun(\bDelta \op, \dSt_k^\laft)$, agrees with $\cD^\bullet_{X/ Y}$. Observe that we have a natural identification
        \[
            (\cD_{X/ Y}^\bullet)' \times_{\bbA^1_k} \bbG_m \simeq X^{\times_Y \bullet} ,
        \]
    and the latter agrees with the formal groupoid associated with $(X \to Y^\wedge_X) \in \mathrm{FMP}_{X \times \bbA^1_k}$. Moreover, we
    have a natural identification
        \[
            \cD_{X/ Y}' \times_{\bbA^1_k} \{0 \} \simeq \rT^\bullet_{X/ Y},
        \]
    where $\rT^\bullet_{X/ Y}$ denotes the groupoid object associated to the usual shifted relative tangent bundle $\rT_{X/ Y}[1]$.
    Under our assumption that $f \colon X \to Y$ is a closed immersion it follows from \cite[Corollary 8.4.3.2]{Lurie_Higher_algebra} that the relative cotangent complex
        \[  
            \bbL_{X/ Y} \in \Coh^+(X),
        \]
    is $1$-connective. As a consequence, we deduce that the natural morphism $X \to \rT^\bullet_{X/Y}$ is a termwise nil-isomorphism. As a consequence,
    we conclude immediately that $\rT^\bullet_{X/ Y}$
    agrees with the formal groupoid,
    over $X$, associated to the formal moduli problem
        \[  
            (X \to \bbT_{X/ Y}[1]^\wedge) \in \mathrm{FMP}_{X/ }.
        \]
\end{rem}

\begin{rem}
    Consider the sequence of deformations in \cref{prop:algebraic_properties_of_deformation}, $\{ X^{(n)} \}_{n \ge 0}$.
    The latter defines a filtration on the global sections of the formal completion $Y^\wedge_X$, which we shall
    refer to as the \emph{Hodge filtration associated to the morphism $f$}. Therefore, \cref{prop:algebraic_properties_of_deformation}
    can be interpreted as a spreading out result of the Hodge filtration
    from de Rham cohomology to the global sections of the formal completion.
\end{rem}

\begin{rem}
    The left-lax action of the multiplicative monoid $\bbA^1_k$ can be made explicit at the fiber of the morphism $p \colon \cD_{X/ Y} \to \bbA^1_k$
    at $\{ 0\} \subseteq \bbA^1_k$. Indeed, under the natural identification
        \[
            (\cD_{X/ Y})_0 \simeq \rT_{X/ Y}^\wedge[1] ,
        \]
    we obtain that the filtration on $\cD_{X/ Y}$ induces the natural filtration on $\rT_{X/ Y}^\wedge[1]$ which is obtained by considering
        \[
            \bbT_{X/ Y}  \in \Ind\Coh(X)^{\mathrm{gr}, =1} \subseteq \Ind\Coh(X)^{\mathrm{gr}, \ge 0},
        \]
    using notations as in \cite[\S 9.2.5.2]{Gaitsgory_Study_II}. Moreover, the construction of the Hodge filtration in \cref{prop:algebraic_properties_of_deformation} (3)
    is naturally $\bbA^1_k$ left-lax equivariant and the equivalence in \cref{prop:algebraic_properties_of_deformation} (4) is $\bbA^1_k$ left-lax equivariant, as well.
\end{rem}

Our goal now is to identify the Hodge filtration on $Y^\wedge_X$, in the case where $f \colon X \to Y$ is a closed immersion:

\begin{prop} \label{lem:identification_of_Hodge_and_adic_filtrations}
    Let $f \colon X \to Y$ be a morphism of affine derived schemes. Denote by $A = \Gamma(X, \cO_X)$ and $B = \Gamma(Y, \cO_Y)$, the corresponding derived
    global sections and let $I \coloneqq \fib(B \to A)$ denote the fiber of the induced morphism of derived $k$-algebras $B \to A$. The following assertions hold:
    \begin{enumerate}
        \item If $f \colon X \to Y$ be a locally of complete intersection closed morphism in the \infcat $(\dAff_k)^\laft$,
        then the Hodge filtration on the formal completion $Y^\wedge_X$ constructed in
        \cref{prop:algebraic_properties_of_deformation} identifies with the usual $I$-adic filtration on $B^\wedge_I$, by taking derived global sections.
        \item Assume that $f \colon X \to Y$ is a closed immersion. Then the Hodge filtration, $\{ \mathrm{Fil}^n_H \}$,
        on the derived global sections 
            \[\Gamma(Y^\wedge_X,\cO_{Y^\wedge_X}) \simeq B^\wedge_I,\]
        produces, for each $n \ge 1$, natural equivalences
            \[
                B/ \mathrm{Fil}_H^n \simeq \dR_{A/ B}/ \mathrm{Fil}_H^n,
            \]
        of derived $k$-algebras.
        In particular, we obtain a natural equivalence of derived $k$-algebras
            \[
                B^\wedge_I \simeq \lim_{n \ge 1} \dR_{A/ B} / \mathrm{Fil}_H^n.  
            \]
    \end{enumerate}
\end{prop}

\begin{proof} We first prove claim (2) of the Proposition.
    Since $f \colon X \to Y$ is a closed immersion between affines we deduce that the relative cotangent complex $\bbL_{X/ Y} \in \Coh^+(X)$ is $1$-connective.
    For this reason, the sequence morphisms
        \[
            X \times \bbA^1_k = X^{(0)} \hookrightarrow X^{(1)} \hookrightarrow \dots \hookrightarrow X^{(n)} \hookrightarrow \dots,  
        \]
    of \cref{prop:algebraic_properties_of_deformation} (4) correspond to successive square-zero extensions.
    Moreover,
    it follows from \cite[Corollary 9.5.2.5]{Gaitsgory_Study_II} that, after applying the Serre duality functor on $\Ind\Coh$, one has an equivalence
        \[
            \cO_{Y^\wedge_X} \simeq \lim_{n \ge 0} \cO_{(X^{(n)})_\lambda},  
        \]
    in $\CAlg_k$, for any
    for $\lambda \neq 0$. Henceforth,
    at the level of derived global sections, one obtains a natural identification   
        \[
            B^\wedge_I \simeq \lim_{n \ge 0} \Gamma((X^{(n)})_\lambda, (\cO_{X^{(n)}})_\lambda),  
        \]
    in the \infcat $\CAlg_k$. Similarly, thanks to
    the proof of \cite[\S 9, Theorem 5.1.3]{Gaitsgory_Study_II} in the case of vector groups (c.f. \cite[\S 9.5.5]{Gaitsgory_Study_II})
    we deduce a natural equivalence
        \[
            \dR_{A/ B} \simeq \lim_{n \ge 0}\Gamma((X^{(n)})_0, \cO_{(X^{(n)})_0}),
        \]
    in $\CAlg_k$. By the naturality of the construction of \cite[\S 9.5.1]{Gaitsgory_Study_II} combined with the argument provided in \cite[\S 9.5.5]{Gaitsgory_Study_II} to the reduction to the case of vector groups,
    we deduce that
        \[
            p \colon X^{(n)} \to   \bbA^1_k,
        \]
    has constant fibers. Indeed, for $n = 0$, the natural morphism
        \[
            X^{(0)} \to X^{(1)},  
        \]
    corresponds, by construction, to the natural square-zero extension
        \[
            \bbL_A \to \bbL_{A/ B} ,  
        \]
    classifying the extension
        \[
            \bbL_{A/ B}[-1] \to \Sym^{\le 1}(\bbL_{A/ B}[-1])  \to A,
        \]
    of $A$-modules.
    Assuming the result for $n \ge 0$, we have that
        \[
            \bbL_{A^{(n)}/ B} \in \Coh^+(A^{(n)}),    
        \]
    where $A^{(n)} \coloneqq \Gamma(X^{(n)}, \cO_{X^{(n)}})$,
    identifies with the relative cotangent complex of the natural morphism
        \[
            B \to \Sym^{\le n}(\bbL_{A/ B}[-1]),  
        \]
    produced by induction and the naturality of the construction. Moreover, thanks to \cite[\S 9.5.5]{Gaitsgory_Study_II}
    the relative cotangent complex
        \[
            \bbL_{\Sym^{\le n}(\bbL_{A/ B}[-1])/ B} \in \Coh^+(\Sym^{\le n}(\bbL_{A/ B}[-1])),
        \]  
    is equipped with a decreasing filtration, starting in degree $n$, whose $n$-th piece identifies with
        \[
            \Sym^{n+1}(\bbL_{A/ B}[-1])[1] \in \Coh^+(\Sym^{\le n}(\bbL_{A/ B}[-1])).  
        \]
    For this reason, the derivation
        \[
            \bbL_{A^{(n)}} \to \bbL_{A^{(n)}/ B} \to \Sym^{n+1}(\bbL_{A/ B}[-1])[1] ,
        \]
    identifies with the natural derivation
        \[
            \bbL_{\Sym^{\le n}(\bbL_{A/ B}[-1])[1])/ B} \to \Sym^{n+1}(\bbL_{A/ B} [1])  ,
        \]
    classifying the extension
        \[
            \Sym^{n+1}(\bbL_{A/ B}[-1]) \to \Sym^{\le n+1}(\bbL_{A/ B}[-1])   \to \Sym^{\le n}(\bbL_{A/ B}[-1]).
        \]
    The claim now follows by induction on $n \ge 0$.
    As a consequence of the previous considerations we deduce that
        \[
            \Gamma((X^{(n)})_\lambda, (\cO_{X^{(n)}})_\lambda) \simeq \Gamma((X^{(n)})_0, (\cO_{X^{(n)}})_0),
        \]
    for every $\lambda \in \bbA^1_k$, and therefore we have natural equivalences
        \[
            B^\wedge_I \simeq \dR_{A/ B} \quad \mathrm{and} \quad B/ \mathrm{Fil}_H^n \simeq \dR_{A/ B}/ \mathrm{Fil}^n_H,
        \]
    in $\CAlg_k$, as desired.
    We now prove statement (1) of the Proposition.
    Assume first that $A$ and $B$ are ordinary $k$-algebras and the (classical) ideal $I \subseteq B$
    is generated by a regular sequence.
    In this case, we have an identification
        \[
            \bbL_{A/ B} \simeq I/ I^2 [1],
        \]
    see for instance \cite[\href{https://stacks.math.columbia.edu/tag/08SJ}{Tag 08SJ}]{stacks-project}.
    In particular, the $n$-th graded piece of the Hodge filtration in \cref{prop:algebraic_properties_of_deformation} (3) can be identified with
        \[
            \Sym^n(\bbL_{A/ B}[-1]) \simeq I^n/ I^{n+1}, 
        \]
    where the latter equivalence follows again by the fact that $I$ is generated by a regular sequence in $\pi_0(B)$. In particular, we deduce, by induction on $n \ge 1$, a natural identification between
    the fiber sequence
        \[
            \Sym^{n+1}(\bbL_{A/ B}[-1]) \to \Sym^{\le n+1}(\bbL_{A/ B}[-1]) \to \Sym^{\le n}(\bbL_{A/ B}[-1]),
        \]
    and the the $I$-adic one
        \[  
            I^n/ I^{n+1} \to B/ I^{n+1} \to B/ I^n,
        \]
    as $B$-modules. By naturality of \cite[Theorem 9.5.1.3]{Gaitsgory_Study_II} and induction on $n \ge 1$ it follows that we have natural identifications
        \[
            X^{(n)} \simeq \Spec (A/ I^{n+1}) ,
        \]
     of affine derived schemes, as desired. The assertion in general follows along the same reasoning noticing that
        \[
            \bbL_{A/ B} \in \Mod_A,  
        \]
    is concentrated in homological degrees $[1, 0]$.
%
\end{proof}     

\begin{rem} It follows from \cite[\S 9, Theorem 5.5.4]{Gaitsgory_Study_II}, that in the situation of \cref{lem:identification_of_Hodge_and_adic_filtrations}, the natural composite morphism
        \[
            \Sym^n(\bbL_{A/ B} [-1]) \to \Sym^{ \le n}(\bbL_{A/ B}[-1]) \to \Sym^{n+1}(\bbL_{A/ B}[-1]), 
        \]
    corresponds to the usual de Rham differential,
    where $\Sym^{ \le n}(\bbL_{A/ B}[-1]) \to \Sym^{n+1}(\bbL_{A/ B}[-1])$ classifies the natural extension
        \[
            \Sym^{n+1}(\bbL_{A/ B} [-1]) \to \Sym^{ \le n +1 }(\bbL_{A/ B}[-1]) \to \Sym^{\le n}(\bbL_{A/ B}[-1]),
        \]
    associated to the $n$-th piece of the Hodge filtration,
\end{rem}

\begin{rem}
    Let $ f\colon X \to Y$ be a closed immersion of affine derived schemes. Let $A \coloneqq \Gamma(X, \cO_X)$ and $B \coloneqq \Gamma(Y, \cO_Y)$.
    In \cite[Proposition 4.16]{Bhatt_Derived_Completions}, the author proves a similar statement to \cref{lem:identification_of_Hodge_and_adic_filtrations}.
    It turns out that both filtrations on $B^\wedge_I$ agree in the case of \cite{Bhatt_Derived_Completions} or in the situation of \cite[\S 9]{Gaitsgory_Study_II}.
    A comparison can be stated as follows: it is immediate that both the filtrations of Bhatt and Gaitsgory-Rozemblyum on $B^\wedge_I$, denoted respectively
        \[
            \mathrm{Fil}^n_B \quad \mathrm{and} \quad \mathrm{Fil}^n_H,  
        \]
    agree in the case where $f$ is a local complete intersection morphism.
    Moreover, it follows from the fact that the cotangent complex commutes with sifted colimits that, for each $n \ge 0$, the natural morphisms
        \begin{equation} \label{eq:de_Rham_differential}
            \Sym^{\le n}(\bbL_{A/ B}[-1]) \to \Sym^{n+1} (\bbL_{A/ B}[-1])  [1]
        \end{equation}
    are stable under sifted colimits as well. Consider the simplicial model structure in the simplicial category of pairs $(A, I)$
    where $A$ is a simplicial $k$-algebra and $I \subseteq A$ a simplicial ideal given by the description:
    a pair $(A, I)$ is a fibration (resp., trivial fibration) if and only if both $A$ and $I$ are simultaneously fibrations (resp., trivial fibrations)
    of simplicial sets. Moreover, the cofibrant objects correspond to pairs $(P, J)$ where, for each $[n] \in \bDelta$, $P_n$ is a polynomial $k$-algebra on a set $X_n$
    and $J_n$ an ideal defined by a subset $Y_n \subseteq X_n$, both preserved by degeneracies and the relative situation. The $\infty$-categorical
    localization at the class of weak equivalences agrees furthermore with the \infcat $\Fun_\mathrm{closed}(\Delta^1, \CAlg_k)$. In particular,
    considering a simplicial resolution $(P_\bullet, I_\bullet)$ for the morphism $B \to A$ consisting of polynomial algebras and regular ideals
    we deduce that
        \[
            \mathrm{Fil}^n_H \simeq \mathrm{Fil}^n_B,  \quad \mathrm{for} \ n \ge 0,
        \]  
    in the general case, as well.
\end{rem}

\subsection{The construction of the deformation in the affinoid case} Consider the object
\[
    \mathbf{B}^{\mathrm{an}, \bullet}_{\mathrm{scaled}} \colon \bDelta \op \to (\dAnSt_k)_{/ \mathbf A^1_k},
\]
obtained as the analytification of the simplicial object $\mathrm{B}^\bullet_{\mathrm{scaled}} \in \Fun(\bDelta \op, \mathrm{dSt}^\laft_k)$ described in the previous section.

Our goal in this section, is to construct the deformation to the normal cone in the non-archimedean
setting, in the case where $Y$ is derived $k$-affinoid. We shall further assume that
    \[
        f \colon X \to Y,
    \]  
exhibits $X$ as an analytic formal moduli problem over $Y$. In particular, $X$ admits a deformation theory and we can
consider the relative cotangent complex
    \[
        \bbL\an_{X/ Y} \in \pro^\mathrm{ps}(\Coh^+(X)).  
    \]

\begin{defin}
    We define the \emph{parametrized non-archimedean deformation to the normal bundle} via the pullback diagram
    \[
    \begin{tikzcd}
        \cD_{X/Y}^{\mathrm{an}, \bullet} \ar{r} \ar{d} & Y \times \bA^1_k \ar{d} \\
        \bfMap_{/ \bA^1_k}(\mathbf{B}_\mathrm{scaled}^{\mathrm{an}, \bullet}, X \times \bA^1_k) \ar{r} & \bfMap_{/ \bA^1_k}(\mathbf{B}_{\mathrm{scaled}}^{\mathrm{an, \bullet}}, Y \times \bA^1_k)
    \end{tikzcd},
    \]
    computed in the \infcat $\mathrm{Fun}(\bDelta \op, \dAnSt_k)$.
\end{defin}

\begin{rem}
    As in the algebraic case, the object $\mathbf{B}^{\mathrm{an}, \bullet}_{\mathrm{scaled}}$ admits a left-lax action of the multplicative monoid $\bA^1_k \in \dAnSt_k$.
    This action induces a natural left-lax action of $\bA^1_k$ on the deformation $\cD_{X/ Y}^{\mathrm{an}, \bullet}$, as well.
\end{rem}

\begin{rem} \label{rem:reduction_step_closed_ims}
    It follows immediately from the definition of $\cD_{X/ Y}^{\mathrm{an}, \bullet}$ that the latter commutes with filtered colimits in $X \in \dAnSt_k$.
    Moreover, \cref{prop:required_conditions_for_formal_moduli_problems} implies that the natural morphism
        \[  
            \colim_{S \in (\anNil^\cl_{/ X})_{/ Y}} \cD_{S/ Y}^{\mathrm{an}, \bullet} \to \cD_{X/ Y}^{\mathrm{an}, \bullet},
        \]
    is an equivalence in $\dAnSt_k$. For this reason, we will often reduce the construction to the case where $f \colon X \to Y$ is a nil-isomorphism of derived $k$-affinoid
    spaces.
\end{rem}

\begin{rem}
    Let $f\colon X \to Y$ be a nil-isomorphism in the \infcat $\dAfd_k$. Then we can consider the diagram
        \[
            \mathrm{id}_Y \colon Y \to X \bigsqcup_{X_\red} Y \to Y,  
        \]
    which exhibits $(X \xrightarrow{f} Y)$ as a retrat of the induced morphism $(X \xrightarrow{g} X \underset{X_\red}{\sqcup}Y)$, both regarded as objects in $\anFMP_{X/ }$.
    In this case, it follows by the naturality of the parametrized Hodge filtration that we have a retract diagram
        \[
            \cD^\bullet_{X/ Y} \to \cD^\bullet_{X/ X \underset{X_\red}{\bigsqcup} Y} \to \cD^\bullet_{X/ Y},   
        \]
    of simplicial objects in $\anFMP_{X \times \bA^1_k/ }$. This observation will allow us to perform reduction arguments by replacing the nil-isomorphism
    $f\colon X \to Y$ by the nil-embedding $g \colon X \to X \underset{X_\red}{\sqcup} Y$.
\end{rem}

Let $f \colon X \to Y$ be a nil-isomorphism in the \infcat $\dAfd_k$.
Thanks to \cite[Lemma 6.9.]{Porta_Yu_NQK}, we deduce that the induced morphism
    \[
        f\alg \colon X\alg \to Y\alg,  
    \]
is again a nil-isomorphism of affine derived schemes. We shall need a few preparations:

\begin{lem} \label{lem:relative_anlaytification_of_closed_immersions_are_compatible_with_alg_construction} Let $Y \in \dAfd_k$ and
    $(f \colon X \to Y) \in \anFMP_{/Y}$. Then one has a natural equivalence
        \[
            (f \alg)\an_Y \simeq f,  
        \]
    in the \infcat $\Fun(\bDelta\op, \dAfd_k)$.
\end{lem}

\begin{proof} 
    The statement of the Lemma is a direct consequence of \cite[Theorem 6.12]{Porta_Yu_NQK}.
\end{proof}

\begin{lem} \label{lem:Noetherian_approximation_of_nil_isomorphisms}
    Let $f\alg \colon X\alg \to Y\alg$ be a closed embedding of derived schemes. There exists a filtered \infcat $\cI$ and a diagram
        \[
            F \colon \cI \op \to \Fun(\Delta^1, \dAff^\laft_k),  
        \]
   such that, for each $\alpha \in \cI$, we have that
        \[
            F(\alpha) \coloneqq (f_\alpha \colon X_\alpha \to Y_\alpha),  
        \]
    is a closed immersion in the \infcat $\dAff_k^\laft$ and such that
        \[
            f \simeq \lim_{\cI\op} F,  
       \]
    in the \infcat $\Fun(\Delta^1, \dAff_k)$.
\end{lem}

\begin{proof}
    Consider the full subcategory $\Fun_\mathrm{closed}(\Delta^1, \CAlg_k) \subseteq \Fun (\Delta^1, \CAlg_k)$ spanned by morphisms $A \to B$
    such that the induced homomorphism of ordinary rings
        \[
            \pi_0(A) \to \pi_0(B),  
        \]
    is surjective. As in the proof of \cref{prop:algebraic_properties_of_deformation} the \infcat $\Fun_\mathrm{closed}(\Delta^1, \CAlg_k)$
    is obtained as an $\infty$-categorical localization at weak equivalences of the simplicial model category whose objects correspond to pairs $(C, I)$ where $C$ is a simplicial $k$-algebra
    and $I \subseteq C$ a simplicial ring. Moreover, thanks to \cite[\S 2 IV, Theorem 4]{Quillen_Homotopical_1967} it follows that
    finite projective generators in the latter correspond to pairs of the form $(P, I)$ where $P$ is some polynomial algebra in a finite number of generators
    (concentrated in non-negatively homologically degrees)
    and $I$ an ideal spanned by a regular sequence. 
    
    In particular, finite projective generators in $\Fun_\mathrm{closed}(\Delta^1, \CAlg_k)$ correspond
    to morphisms $g \colon P \to Q$ where $P$ is a finite free derived $k$-algebra and the morphism $g$ is a complete intersection morphism.
    As in the proof of \cite[Proposition 8.2.5.27]{Lurie_Higher_algebra} we then deduce that the \infcat $\mathrm{Fun}_\mathrm{closed}(\Delta^1, \CAlg_k)$ is compactly generated by
    morphisms $g \colon P \to Q$ which can be obtained from finite projective generators via a finite sequence of retracts and finite colimits.

    Arguing as in \cite[Proposition 8.2.5.27]{Lurie_Higher_algebra}, we are able to describe the set of compact generators of $\Fun_\mathrm{closed}(\Delta^1, \CAlg_k)$: they correspond to
    locally of finite presentation morphisms between locally finite presentation derived $k$-algebras. In particular, we can find a presentation
    of
        \[
            f\alg \colon X\alg \to Y\alg,  
        \]
    by closed immersions between almost finite presentation affine derived schemes
        \[
            f_\alpha \colon X_\alpha \to Y_\alpha, \quad \mathrm{for \ } \alpha \in \cI\op,  
        \]
    where $\cI$ is a filtered \infcat, as desired.
\end{proof}

\begin{rem} Let $f \colon X \to Y$ be a nil-embedding in the \infcat $\dAfd_k$. Find a presentation of $f\alg$ by closed immersions between almost of finite presentation affine derived schemes of the form
        \[
            \lim_{\alpha \in I\op}(f_\alpha \colon X_\alpha \to Y_\alpha).  
        \]
    Consider the object $\cD_{X\alg/ Y\alg}^\bullet \in \Fun(\bDelta \op, \dAff_k)$ as in \cref{construction:def_to_the_normal_bundle}.
    It is clear from the descritption
    of $\cD_{X\alg/Y\alg}^\bullet$ given in \cref{construction:def_to_the_normal_bundle} combined with \cref{rem:simplification_of_deformation_in_the_closed_immersion}
    that one has a natural equivalence of derived stacks
        \[
            \cD_{X\alg/Y\alg}^\bullet \simeq \lim_{\alpha \in I\op} \cD_{X_\alpha\alg/ Y_\alpha\alg}^\bullet,
        \]
    in the \infcat $\mathrm{Fun}(\bDelta \op, \dSt_k)$.
\end{rem}

\begin{lem} \label{lem:analytification_of_simplicial_object_deformation}
    Let $Y \in \dAfd_k$. Consider a morphism $f \colon X \to Y$ in $\dAnSt_k$ which exhibits $X$ as an analytic formal moduli problem
    over $Y$. Then one has a natural equivalence
        \[
            \theta_{X/ Y} \colon \cD_{X/Y}^{\mathrm{an}, \bullet} \simeq (\cD_{X\alg/ Y\alg}^\bullet)\an_Y  ,
        \]
    where $(-)_Y\an$ denotes the relative analytification with respect to $Y$.
\end{lem}

\begin{proof}
    We first observe that both $\cD_{X/Y}^{\mathrm{an}, \bullet}$ and $(\cD_{X\alg/ Y\alg}^\bullet)\an_Y$ are stable under filtered
    colimits in $X$. For this reason, we reduce ourselves to the case where $f \colon X \to Y$ itself is a nil-isomorphism in the \infcat
    $\dAfd_k$. Similarly, $\theta_{X/ Y}$ is stable under retracts. Therefore, we are allowed to replace the morphism $X \to Y$
    by
        \[  
            X \to Y \underset{X_\red}{\bigsqcup} X,
        \]
    and therefore assume that $f \colon X \to Y$ is a nil-embedding.
    As in \cref{lem:Noetherian_approximation_of_nil_isomorphisms}, write the nil-embedding between affine derived schemes
        \[
            f \colon X\alg \to Y\alg,  
        \]
    as a cofiltered limit of closed immersions
        \[\lim_{\alpha \in I\op} f_\alpha \colon \lim_{\alpha \in I\op} X_\alpha\alg \to \lim_{\alpha \in I\op} Y_\alpha\alg,\]
    where for each index $\alpha \in I$, both $X_\alpha$
    and $Y_\alpha$ are affine derived schemes almost of finite presentation. Since the natural projection morphism
        \[
            \mathrm{B}_\mathrm{scaled}^\bullet \to \bbA^1_k, 
        \]
    where $\bbA^1_k$ is considered as a constant simplicial object in $\dSt_k$, is a proper morphism, we deduce from \cite[Theorem 6.13]{Holstein_Analytification_of_mapping_stacks}
    that for every $\alpha$, the natural morphisms
        \begin{align*}
            \Map_{/\bbA^1_k}(\mathrm{B}^\bullet_\mathrm{scaled}, X_\alpha\alg \times \bbA^1_k)\an & \to \bfMap_{/ \bA^1_k}(\mathbf{B}^{\bullet, \mathrm{an}}_\mathrm{scaled} , (X_\alpha\alg)\an \times \bA^1_k) \\
            \Map_{/\bbA^1_k}(\mathrm{B}^\bullet_\mathrm{scaled}, Y_\alpha\alg \times \bbA^1_k)\an & \to \bfMap_{/ \bA^1_k}(\mathbf{B}^{\bullet, \mathrm{an}}_\mathrm{scaled} , (Y_\alpha\alg)\an \times \bA^1_k)
        \end{align*}
    are equivalences in the \infcat $\dAnSt_k$. Therefore, the natural morphism
        \[
            (\cD_{X\alg_\alpha/Y\alg_\alpha}^\bullet )\an  \to \cD_{(X_\alpha \alg)\an/ (Y_\alpha \alg)\an}^{\bullet, \mathrm{an}},
        \]
    is an equivalence in the \infcat $\dAnSt_k$. Observe further that for every $\alpha \in A$ the mapping stacks
        \[
            \Map_{/ \bbA^1_k}(\mathrm{B}^\bullet_\mathrm{scaled}, X_\alpha\alg \times \bbA^1_k) \quad \mathrm{and} \quad \Map_{/ \bbA^1_k}(\mathrm{B}^\bullet_\mathrm{scaled}, Y_\alpha\alg \times \bbA^1_k)  
        \]
    are affine schemes and therefore it follows by the construction of the analytification functor
        \[
            (-)\an \colon \dAff_k \to \dAnSt_k,  
        \]
    as a right Kan extension of the usual analytification functor
        \[
            (-)\an \colon \dAff^\laft_k \to \dAn_k,  
        \]
    that we have natural equivalences
        \begin{align*}
           ( \lim_\alpha \Map_{/\bbA^1_k}(\mathrm{B}^\bullet_\mathrm{scaled}, X_\alpha\alg \times \bbA^1_k)) \an & \simeq \bfMap_{/\bbA^1_k}(\mathrm{B}^{\bullet, \mathrm{an}}_\mathrm{scaled}, (X\alg)\an_Y \times \bA^1_k) \\
           ( \lim_\alpha \Map_{/\bbA^1_k}(\mathrm{B}^\bullet_\mathrm{scaled}, Y_\alpha\alg \times \bbA^1_k)) \an & \simeq \bfMap_{/\bbA^1_k}(\mathrm{B}^{\bullet, \mathrm{an}}_\mathrm{scaled}, (Y\alg)\an_Y \times \bA^1_k).
        \end{align*}
    in the \infcat $(\dAnSt_k)_{X \times \bbA^1/ /Y \times \bbA^1_k}$.
    The result now follows from the existence of a commutative cube
        \[
        \begin{tikzcd}[column sep=0.125in,row sep=0.125in]
            (\cD_{X\alg/ Y\alg}^\bullet)\an_Y \ar{rr} \ar{dd} \ar{rd} &  & Y \times \bA^1_k \ar{rd} \ar{dd} & \\
            & (\cD_{X\alg/ Y\alg}^\bullet)\an \ar{rr} \ar{dd} &  & (Y\alg \times \bbA^1_k)\an \ar{dd} \\
            \bfMap_{/\bA^1_k}(\mathbf{B}^\bullet, X \times \bA^1_k) \ar{rr} \ar{rd} & & \bfMap_{/\bA^1_k}(\mathbf{B}^\bullet, Y \times \bA^1_k) \ar{rd} \\
            & \bfMap_{/\bA^1_k}(\mathbf{B}^\bullet, (X\alg)\an \times \bA^1_k) \ar{rr} &  & \bfMap_{/\bA^1_k}(\mathbf{B}^\bullet, (Y\alg)\an \times \bA^1_k)
        \end{tikzcd}
        \]
    whose top and front squares are fiber products in $\dAnSt_k$, and thus so it is the back square, as desired.
\end{proof}

\begin{cor} \label{cor:analytic_def_fiber_at_0_is_tangent} With notations as above, one has a canonical equivalence
        \[
            (\cD^{\bullet, \mathrm{an}}_{X/Y})_0 \simeq (\bT^{\mathrm{an}, \bullet}_{X/Y})^\wedge ,
        \]
    where the latter denotes the commutative group object associated to the formal completion of the tangent bundle of $f$ along the zero section $s_0 \colon
    X \to \bT^{\mathrm{an}, \bullet}_{X/Y}$.
\end{cor}

\begin{proof}
    The result follows from \cite[Proposition 9.2.3.6]{Gaitsgory_Study_II} combined with the fact that relative analytification
    commutes with tangent bundles and formal completions, see the proof of \cite[Lemma 5.31]{Holstein_Analytification_of_mapping_stacks}
    and \cite[Corollary 5.20]{Holstein_Analytification_of_mapping_stacks}.
\end{proof}

The following result is an immediate consequence of \cref{lem:analytification_of_simplicial_object_deformation}:

\begin{lem} \label{lem:deformation_theory_for_D^an_bullet_X/Y}
    For each $[n] \in \bDelta$, the object $\cD_{X/Y}^{\mathrm{an}, [n]} \in (\dAnSt_k)_{X \times \bA^1_k/ /Y \times \bA^1_k}$ admits a deformation theory.
\end{lem}

\begin{proof}
    We shall prove that each component of the simplicial object
        \[\cD^{ \mathrm{an}, \bullet}_{X/ Y} \in \Fun(\bDelta \op, \anFMP_{X \times \bA^1_k/ /Y\times\bA^1_k}),\]
    admits a deformation theory.
    By \cite[Lemma 2.3.2]{Gaitsgory_Study_II}, it follows
    that $\cD^\bullet_{X\alg/ /Y\alg}$ is an object in $\mathrm{FMP}_{X\alg/ /Y\alg}$. The result now follows by applying the relative analytification functor $(-)_Y\an$ together with
    \cite[Proposition 6.10]{Porta_Yu_NQK}.
\end{proof}

\begin{rem} \label{rem:D_X/Y_lives_over_and_under_XxA^1_as_anFMP} Let $(X \xrightarrow{f} Y) \in \anFMP_{/ Y}$, with $Y \in \dAfd_k$.
    It follows immediately from \cref{lem:deformation_theory_for_D^an_bullet_X/Y} that the datum
        \[
            (X \times \bA^1_k \to \cD^{\bullet, \mathrm{an}}_{X/ Y} \to Y \times \bA^1_k) \in (\dAnSt_k)_{X \times \bA^1_k/ /Y\times\bA^1_k},  
        \]
    belongs to the full subcategory $\anFMP_{X\times\bA^1_k/ /Y\times\bA^1_k}$.
\end{rem}

\begin{construction}
    The \infcat $\anFMP_{X\times\bA^1_k/ /Y\times\bA^1_k}$ is presentable and in particular it
    admits sifted colimits. Thanks to \cref{lem:deformation_theory_for_D^an_bullet_X/Y},
    we can consider the object
        \[
            \cD\an_{X/Y} \coloneqq \colim_{\bDelta \op} \cD^{\mathrm{an}, \bullet}_{X/Y} \in \anFMP_{X\times\bA^1_k/ /Y\times\bA^1_k}.  
        \]
    Similarly, we consider the sifted colimit
        \[
            \cD_{X\alg/ Y\alg} \coloneqq \colim_{\bDelta \op} \cD^\bullet_{X\alg/ Y\alg} \in \mathrm{FMP}_{X\alg \times \bbA^1_k/ /Y\alg \times \bbA^1_k} ,
        \]
    computed in the \infcat $\mathrm{FMP}_{X\alg\times\bbA^1_k/ /Y\alg\times\bbA^1_k}$. We can consider the latter as an object in
        \[
            \cD_{X\alg/ Y\alg} \in \dSt_{X\alg \times \bbA^1_k/ /Y\alg \times \bbA^1_k},  
        \]
    and therefore consider its relative analytification $(\cD_{X\alg/ Y\alg})\an_Y \in \dAnSt_{X\alg \times \bA^1_k/ /Y\alg \times \bA^1_k}$. Thanks to
    \cref{lem:analytification_of_simplicial_object_deformation} it follows that we have a natural morphism
        \[
            \theta_{X/ Y} \colon \cD\an_{X/ Y} \to (\cD_{X\alg/ Y\alg})\an_Y,  
        \]
    in the \infcat $(\dAnSt_k)_{X \times \bA^1_k/ /Y \times \bA^1_k}$.
\end{construction}

Our next goal is to prove that the morphism $\theta_{X/ Y}$ is an equivalence of deformations. In order to prove the latter statement, we shall
need a preliminary lemma:

\begin{lem} \label{lem:Noetherian_approximation_of_the_deformation_to_the_normal_bundle}
    Let $f \colon X \to Y$ be a nil-embedding in the \infcat $\dAfd_k$. Consider approximation of the nil-embedding morphism $f\alg$
        \[
            \lim_{\alpha \in I\op} f_\alpha \colon \lim_{\alpha \in I\op} X_\alpha \to \lim_{\alpha \in I\op} Y_\alpha,
        \]
    as in \cref{lem:Noetherian_approximation_of_nil_isomorphisms}. Then, we have a natural morphism of the form
        \begin{equation} \label{eq:morphism_theta}
            \theta \colon \cD_{X\alg/ Y\alg}  \to \lim_{\alpha \in I\op} \cD_{X_\alpha/ Y_\alpha},
        \end{equation}
    which is furthermore an equivalence in the \infcat $\anFMP_{X \times \bbA^1_k / Y \alg \times \bbA^1_k}$
\end{lem}

\begin{proof}
    It is clear that the natural morphisms $X_\alpha \to \cD_{X_\alpha/ Y_\alpha}$ assemble into a morphism
        \[
            X \alg \simeq \lim_{\alpha \in I\op} X_\alpha \to \lim_{\alpha \in I\op} \cD_{X_\alpha/ Y_\alpha},
        \]
    which admits a deformation theory, therefore the morphism 
        \[X\alg \to \lim_{\alpha \in I\op} \cD_{X_\alpha/ Y_\alpha},\]
    exhibits $\lim_{\alpha \in I\op} \cD_{X_\alpha/ Y_\alpha}$ as an object
    in the \infcat $\mathrm{FMP}_{X\alg \times \bbA^1_k/ }$. Notice that \cref{thm:Phi_is_an_equivalence} holds true in this setting, proved in
    an analogous way, to provide us
    with an equivalence of \infcats
        \[  
            \rB_X(\bullet) \colon \mathrm{FGrpd}(X\alg \times \bbA^1_k) \to \mathrm{FMP}_{X\alg \times \bbA^1_k/ },
        \]
    where $\mathrm{FGrp}(X\alg \times \bbA^1_k)$ denotes the \infcat of formal groupoids over $X\alg \times \bbA^1_k \in \dSt_k$. For this reason, we are reduced to show
    the existence of a natural equivalence between formal groupoids over $X\alg \times \bbA^1_k$,
        \[
            \cD_{X\alg/ Y\alg}^\bullet \simeq X^{\times_{\lim_\alpha \cD_{X_\alpha/ Y_\alpha}} \bullet},
        \]
    where the right hand side denotes the \v{C}ech nerve of the morphism $X \to \lim_\alpha \cD_{X_\alpha/ Y_\alpha}$.
    Since limits
    commute with limits in the \infcat $\dAnSt_k$, we obtain a chain of natural equivalences of the form
        \begin{align*}
            (X\alg)^{\times_{\lim_{\alpha} \cD_{X_\alpha/ Y_\alpha}}\bullet}    & \simeq \lim_{\alpha } \big(X_\alpha^{\times_{\cD_{X_\alpha/ Y_\alpha}} \bullet} \big) \\
                                                                                & \simeq \lim_{\alpha } \cD_{X_\alpha/ Y_\alpha}^\bullet \\
                                                                                & \simeq \cD_{X\alg/ Y\alg}^\bullet,
        \end{align*}
    in the \infcat $\mathrm{FGrp}(X)$. Only the the second equivalence needs a justification:
    for each $\alpha \in I$, we have natural identifications
        \[
            \cD_{X_\alpha/ Y_\alpha} \simeq \rB_{X_\alpha}( \cD_{X_\alpha/ Y_\alpha}^\bullet) , 
        \]
    by construction, where $\rB_{X_\alpha} (-)$ is as in \cite[\S 5.2]{Gaitsgory_Study_II}. Therefore, we deduce by adjunction that
        \[
            X^{\times_{\cD_{X_\alpha/ Y_\alpha} }\bullet} \simeq \cD_{X_\alpha/ Y_\alpha}^\bullet,
        \]
    is an equivalence in the \infcat $\mathrm{FGrpd}(X_\alpha)$, as well. The proof is thus concluded.
\end{proof}

The following result implies that the relative analytification commutes with the deformation to the normal bundle:

\begin{prop} \label{prop:rel_analytification_preserves_the_deformation}
    The natural morphism of \eqref{eq:morphism_theta},
        \[
            \theta_{X/ Y}\colon\cD\an_{X/Y}\to(\cD_{X\alg/ Y\alg})\an_Y,
        \]
    is an equivalence of derived $k$-analytic stacks.
\end{prop}

\begin{proof} Since the morphism $\theta_{X/ Y}$ commutes with filtered colimits in $X$, we reduce the statement of the Proposition to the
    case where $X \in \anNil_{/ Y}$.
    In this case, we have that $Y \times \bA^1_k \in \anFMP_{X \times \bA^1_k/ }$. By naturality of the morphism $\theta_{X/ Y}$
    we are allowed to prove the statement of the Proposition up to a retract. For this reason,
    we can replace $Y$ by $Y \underset{X_\red}{\sqcup} X$ and therefore assume that $f \colon X\to Y$ is a nil-embedding.
    Consider now the relative analytification functor
        \[
            (-)\an_Y \colon \mathrm{FMP}_{/ Y\alg \times \bbA^1_k}   \to \anFMP_{/ Y \times \bA^1_k}, 
        \]
    introduced in \cite[\S 6.1]{Porta_Yu_NQK}. Thanks to \cite[Theorem 6.12]{Porta_Yu_NQK} the natural functor displayed above is an equivalence of \infcats.
    Observe further that we have a well defined functor
        \begin{equation} \label{eq:rel_analytification_functor_as}
            (-)\an_Y \colon \mathrm{FMP}_{X\alg \times \bbA^1_k/ /Y\alg \times \bbA^1_k} \to \anFMP_{X \times \bA^1_k/ /Y \times \bA^1_k},
        \end{equation}
    induced by the usual relative analytification functor. Indeed, we have that 
        \[
            X \times \bA^1_k \simeq (X\alg)\an_Y \times \bA^1_k,   
        \]
    as the pair$((-)\alg, (-)\an_Y)$ forms an equivalence itself. It is further clear that \eqref{eq:rel_analytification_functor_as}
    is an equivalence of \infcats as well. In particular, it commutes with all small colimits. Therefore, we have a chain of natural equivalences
        \begin{align*}
            (\cD_{X\alg/ Y\alg})\an_Y & \simeq (\colim_{\bDelta \op} \cD_{X\alg/Y\alg}^\bullet) \an_Y \\
                                      & \simeq \colim_{\bDelta \op}(\cD_{X\alg/ Y\alg}^\bullet)\an_Y \\
                                      & \simeq \colim_{\bDelta \op} \cD_{X/ Y}^{\mathrm{an}, \bullet},
        \end{align*}
    in $\anFMP_{X \times \bA^1_k/ / Y\times \bA^1_k}$, thanks to \cref{prop:rel_analytification_preserves_the_deformation}. The conclusion now follows from the observation that the forgetful functor
        \[
            \anFMP_{X \times \bA^1_k/ /Y \times \bA^1_k} \to \anFMP_{X \times \bA^1_k/ }  ,
        \]
    commutes with colimits and it is moreover conservative, thus it reflects sifted colimits.
\end{proof}

\begin{lem} \label{lem:identification_of_non-archimedean_fibers}
    Consider the natural projection morphism
        \[
            q \colon \cD_{X/ Y}\an \to \bA^1_k.
        \]
    Then its fiber at $\lambda \neq 0$ coincides with the formal completion
        \[
            Y^\wedge_X,  
        \]
    and its fiber at $\{ 0 \} \subseteq \bA^1_k$ with the completion of the shifted tangent bundle $\bT\an_{X/ Y}[1]$ along the zero section $s_0 \colon X \to \bT\an_{X/ Y}[1]$.
\end{lem}

\begin{proof}
    The above assertions follow immediately from \cref{prop:rel_analytification_preserves_the_deformation}, \cref{cor:analytic_def_fiber_at_0_is_tangent} and \cref{prop:algebraic_properties_of_deformation}.
\end{proof}

\begin{construction} \label{const:pullbacks_of_deformations} Let $Y \in \dAfd_k$ and $(X \xrightarrow{f} Y) \in \anFMP_{/ Y}$.
    Denote by $g \colon U \to Y$ a morphism in $\dAfd_k$. Consider the pullback diagram
        \[
        \begin{tikzcd}
            X_U \ar{r} \ar{d} & U \ar{d} \\
            X \ar{r}{f} & Y,  
        \end{tikzcd}
        \]
    computed in the \infcat $\dAfd_k$. It follows from the definitions that the morphism $X_U \to U$ admits a deformation theory and
    that we have a natural pullback square of simplicial objects
        \[
        \begin{tikzcd}
            \cD_{X_U/ U}^{\mathrm{an}, \bullet} \ar{r} \ar{d} & \cD_{X/Y}^{\mathrm{an}, \bullet} \ar{d} \\
            U \ar{r} & Y,
        \end{tikzcd}
        \]
    in the \infcat $\dAnSt_k$. For this reason, we obtain a natural commutative diagram
        \[
        \begin{tikzcd}
            \cD_{X_U/ U}\an \ar{r} \ar{d} & \cD_{X/ Y} \an \ar{d} \\
            U \ar{r} & Y,
        \end{tikzcd}
        \]
    in the \infcat $\dAnSt_k$. Similarly, \cite[Proposition 3.17]{Porta_Yu_Derived_non-archimedean_analytic_spaces}
    implies that we have a pullback diagram
        \[
        \begin{tikzcd}
            X_U\alg \ar{r} \ar{d} & U \alg \ar{d} \\
            X\alg \ar{r}{f\alg} & Y\alg
        \end{tikzcd},
        \]
    in the \infcat $\dAff_k$. Reasoning as above, we have a natural commutative square
      \[
        \begin{tikzcd}
        \cD_{X_U\alg/ U\alg} \ar{r} \ar{d} & \cD_{X\alg/ Y\alg} \ar{d} \\
        U\alg \ar{r} & Y \alg,  
       \end{tikzcd}
      \]
    in the \infcat $\dSt_k$.
\end{construction}

\begin{prop} \label{prop:gluing_the_deformation}
    The commutative square
        \[
        \begin{tikzcd}
            \cD_{X_U/ U}\an \ar{r} \ar{d} & \cD_{X/ Y} \an \ar{d} \\
            U \ar{r} & Y, 
        \end{tikzcd}
        \]
    of \cref{const:pullbacks_of_deformations}, is a pullback square in the \infcat $\dAnSt_k$.
\end{prop}

\begin{proof} Since the pullback diagram above is preserved under filtered colimits in $X$, we can assume without loss of generality that
    $f \colon X\to Y$ is a nil-isomorphism. Similarly, we are allowed to replace the retract $ \colon X \to Y$ by the nil-embedding
        \[
            g \colon \to X \to Y \underset{X_\red}{\bigsqcup} X.  
        \]
    In order to show the assertion of the proposition, we are reduced to show that the natural morphism
        \[
            \cD_{X_U/ U}\an \to \cD\an_{X/ Y} \times_Y U,  
        \]
    is an equivalence in the \infcat $\anFMP_{X_U/ /U}$. Moreover, \cref{prop:rel_analytification_preserves_the_deformation} reduce us
    to show that the natural morphism
        \[
            \cD_{X_U\alg/ U\alg} \to \cD_{X\alg/ Y\alg} \times_{Y\alg}U\alg,
        \]
    is an equivalence in $\dSt_k$ (see also \cref{prop:rel_analytification_preserves_the_deformation}). 
    Moreover, \cref{lem:Noetherian_approximation_of_the_deformation_to_the_normal_bundle} allows us to perform Noetherian approximation on the
    nil-isomorphism $f \alg \colon X\alg \to Y\alg$.
    For this reason, we can find an approximation of the morphism $f \colon X \alg \to Y\alg$
    an approximation of the form
        \[
            \lim_{\alpha \in I\op} ( f_\alpha \colon X_\alpha \to Y_\alpha ),  
        \]
    where $I$ is a filtered \infcat and for each $\alpha \in I$, the morphism $f_\alpha$ is a closed immersion of affine derived schemes
    almost of finite presentation over $k$. Similarly, find an almost of finite presentation approximation
        \[
            U \alg \simeq \lim_{\beta \in J\op} U_\beta.  
        \]
    By almost finite presentation we deduce that for every $\alpha \in I$ there exists an index $\beta(\alpha) \in J$ such that the composite
        \begin{equation} \label{eq:factorization}
             U\alg \to Y\alg \to Y_\alpha,
        \end{equation}
    factors as 
        \[U\alg \to U_{\beta(\alpha)} \to Y_\alpha,\]
    for a well defined morphism $U_{\beta(\alpha)} \to Y_\alpha$ in $\dAff_k^\laft$.
    Let $\cK$ denote the filtered subcategory of the product $I \times J$ spanned by those pairs $(\alpha, \beta) \in I \times J$
    such that the composite displayed in \eqref{eq:factorization} factors as $U_\beta \to Y_\alpha$. We thus have that the morphism
        \[
            U \alg \to Y\alg ,  
        \]
    can be written as $\lim_{(\gamma_1, \gamma_2) \in \cK} (U_{\gamma_1} \to Y_{\gamma_2})$. For each $(\gamma_1, \gamma_2) \in \cK$ consider the deformations to the normal bundle
        \[
            \cD_{X_{\gamma_2}/ Y_{\gamma_2}} \quad \mathrm{and} \quad \cD_{X_{U_{\gamma_2}}/ U_{\gamma_2}},
        \]
    where $X_{U_{\gamma_2}}$ is defined as the pullback of the diagram
        \[
            U_{\gamma_1} \to Y_{\gamma_2} \leftarrow X_{\gamma_2}.  
        \]
    We claim that for each $(\gamma_1, \gamma_2) \in \cK$, we have a natural equivalence
        \[
            \cD_{X_{U_{\gamma_1}}/ U_{\gamma_1}} \simeq \cD_{X_{\gamma_2}/ Y_{\gamma_2}} \times_{Y_{\gamma_2}} U_{\gamma_1},  
        \]
    of derived $k$-stacks almost of finite presentation. Indeed, by conservativity of the relative tangent complex, c.f. \cite[\S 5, Theorem 2.3.5]{Gaitsgory_Study_II}, it suffices to show that 
        \[
            \bbT_{X_{U_{\gamma_1}}/ \cD_{X_{U_{\gamma_1}}/ U_{\gamma_1}}} \to \bbT_{X_{U_{\gamma_1}}/ \cD_{X_{\gamma_2}/ Y_{\gamma_2}} \times_{Y_{\gamma_2}} U_{\gamma_1}},
        \]
    is an equivalence in the \infcat $\Ind\Coh(X_{U_{\gamma_1}})$. Thanks to \cite[\S 5, Corollary 2.3.6]{Gaitsgory_Study_II}, we are reduced to show that
    the natural morphism of simplicial objects
        \[
            \{ \bbT_{X_{U_{\gamma_1}}/ \cD_{X_{U_{\gamma_1}}/ U_{\gamma_1}}^{\bullet}} \} \to \{ \bbT_{X_{U_{\gamma_1}}/ \cD_{X_{\gamma_2}/ Y_{\gamma_2}}^{\bullet} \times_{Y_{\gamma_2}} U_{\gamma_1}} \},
        \]
    is an equivalence in $\Fun(\bDelta\op, \Ind\Coh(X_{U_{\gamma_1}}))$. The latter assertion further reduce us to show that the commutative
    diagram
        \[
        \begin{tikzcd}
            \cD_{X_{U_{\gamma_1}}/ U_{\gamma_1}}^{ \bullet} \ar{r} \ar{d} & \cD_{X_{\gamma_2}/ Y_{\gamma_2}}^{\bullet} \ar{d} \\
            U_{\gamma_1} \ar{r} & Y_{\gamma_2} 
        \end{tikzcd}
        \]
    is a pullback diagram component-wise in $\dAff_k^\laft$. By unwinding the definitions, it suffices to prove that the commutative diagram
    of simplicial objects
        \[
        \begin{tikzcd}
            \bfMap_{/ \bA^1_k}(\mathbf B_\mathrm{scaled}^\bullet,  X_{U_{\gamma_1}} \times \bA^1_k) \ar{r} \ar{d} & \bfMap_{/ \bA^1_k}(\mathbf B_\mathrm{scaled}^\bullet, U_{\gamma_1} \times \bA^1_k) )  \ar{d} \\
            \bfMap_{/ \bA^1_k}(\mathbf B_\mathrm{scaled}^\bullet, X_{{\gamma_2}} \times \bA^1_k) \ar{r} & \bfMap_{/ \bA^1_k}(\mathbf B_\mathrm{scaled}^\bullet, Y_{\gamma_2} \times \bA^1_k) )
        \end{tikzcd}
        \]
    is a pullback square. But the latter assertion is obvious as derived $k$-analytic mapping stacks commute with fiber products in $\dAnSt_k$.
\end{proof}

\subsection{Gluing the Deformation} In this \S, we globalize the results proved so far in \S 3.2.
Let $f \colon X \to Y$ be a morphism of locally geometric derived $k$-analytic stacks.
Consider as before the \emph{deformation to the normal bundle of the morphism $f$} constructed via the pullback diagram
    \[
    \begin{tikzcd}
        \cD^{\mathrm{an}, \bullet}_{X/ Y} \ar{r} \ar{d} & Y^\wedge_X \times \bA^1_k \ar{d} \\
        \bfMap_{/ \bA^1_k} (\bB^{\mathrm{an}, \bullet}_\mathrm{scaled}, X \times \bA^1_k) \ar{r} & \bfMap_{/ \bA^1_k}(\bB^{\mathrm{an},\bullet}_{\mathrm{scaled}}, Y^\wedge_X \times \bA^1_k),
    \end{tikzcd}
    \]
in the \infcat $\dAnSt_k$. As in the previous \S, one can show that the simplicial object 
    \[
        \cD^{\mathrm{an}, \bullet}_{X/ Y} \in \anFMP_{X \times \bA^1_k/ /Y^\wedge_X \times \bA^1_k}.
    \]
Let now $g \colon U \to Y^\wedge_X$ be a morphism in $\dAnSt_k$, where $U \in \dAfd_k$. Consider the pullback diagram
    \[
    \begin{tikzcd}
        X_U \ar{r}{g'} \ar{d} & X \ar{d}{f} \\
        U \ar{r}{g} & Y^\wedge_X,
    \end{tikzcd}
    \]
in the \infcat $\dAnSt_k$. We have:

\begin{lem}
    The morphism $X_U \to U$ admits a deformation theory, and thus exhibits $X_U$ as an object in $\anFMP_{/U}$.
\end{lem}

\begin{proof}
    It is clear that $X_U$ is both nilcomplete and infinitesimally cartesian as a derived $k$-analytic stack. It suffices thus to show that
    the structural morphism $X_U \to U$ admits a cotangent complex. We claim that $\bbL_{X_U/ U}$ identifies with
        \[
            (g')^* \bbL_{X/ Y} \in \Coh^+(X_U).
        \]
    Indeed, let $Z \to X_U$ be any morphism with $h \colon Z \in \dAfd_k$. Then we obtain that
        \[
            h^* (g')^* \bbL_{X/ Y} \simeq (g' \circ h)^* (\bbL_{X/ Y}),  
        \]
    in $\Coh^+(Z)$, which corepresents the functor
        \[
            \mathrm{Der}_{X/ Y}(Z, -) \colon \Coh^+(Z) \to \cS,  
        \]
    see \cite[Definition 7.6]{Porta_Yu_Representability} for the definition of the latter. Since fiber products commute with fibers, we have a natural identification
        \[
            \mathrm{Der}_{X_U}(Z, -) \simeq \mathrm{Der}_{X} (Z, -) \times_{\mathrm{Der}_{Y}(Z, -)} \mathrm{Der}_U(Z, -).
        \]
    We thus conclude that 
        \[
            \mathrm{Der}_{X_U/ U}(Z, -) \simeq \fib \big(\mathrm{Der}_{X_U}(Z, -) \to \mathrm{Der}_{U}(Z, -)\big)  ,
        \]
    identifies naturally with
        \[
            \mathrm{Der}_{X/ Y}(Z, -) \simeq \fib \big(\mathrm{Der}_{X}(Z, -) \to \mathrm{Der}_{Y}(Z, -)\big)  .          
        \]
    The claim of the lemma now follows from the fact that $\mathrm{Der}_{X/ Y}(Z, -) \colon \Coh^+(Z) \to \cS$ is corepresentable
    by $h^* (g')^*\bbL_{X/ Y}$, as desired. 
\end{proof}

Thanks to the above lemma we conclude that the canonical map $X_U \to U$ exhibits $X_U$ as an analytic formal moduli problem
over $U$. We are thus in the case of the previous section. The following result implies that the deformation to the normal bundle in the non-archimedean
setting glues:

\begin{prop} \label{prop:gluing_def}
    The simplicial object
        \[
            \cD_{X/ Y}^{\mathrm{an}, \bullet} \colon \bDelta \op \to \anFMP_{X/ /Y},  
        \]
    admits a colimit $\cD_{X/ Y}\an \in \anFMP_{X/ }$.
\end{prop}

\begin{proof} Let $(Y^\wedge_X)^\mathrm{afd}$ denote the \infcat consisting of morphisms
        \[
            U \to Y^\wedge_X,
        \] 
    where $U$ is a derived $k$-affinoid space. We have a natural equivalence of \infcats
        \[
            \Psi \colon (\dAnSt_k)_{/ Y}  \to \lim_{U \in (Y^\wedge_X)^\mathrm{afd}} (\dAnSt_k)_{/U}.
        \]
    We deduce from the construction of $\cD^{\mathrm{an}, \bullet}_{X/ Y}$ that the simplicial object $\cD^{\mathrm{an}, \bullet}_{X/ Y}$ satisfies
        \begin{align*}
            \Psi(\cD^{\mathrm{an}, \bullet}_{X/ Y}) & \simeq \{ \cD^{\mathrm{an}, \bullet}_{X_{U}/ U} \} \\
                                                    & \in \lim_{U \in (Y^\wedge_X)^\mathrm{afd}} (\dAnSt_k)_{/ U}.
        \end{align*}
    For each $U \in (Y^\wedge_X)^\mathrm{afd}$ the \infcats $\anFMP_{X_U/ /U}$ are presentable. For this reason, we can consider the geometric realization
        \[
            \cD_{X_U/ U} \an \in \anFMP_{X_U / U},
        \]
    for every $[n] \in \bDelta$. Thanks to
    \cref{prop:gluing_the_deformation} we deduce that the above deformations glue to form a uniquely defined object
        \[
            \cD\an_{X/ Y} \in \dAnSt_k,  
        \]
    satisfiying the relation 
        \begin{align*}
            \Psi(\cD\an_{X/ Y}) & \simeq \{ \cD\an_{X_{U}/ U} \} \\
                                & \in \lim_{U \in (Y^\wedge_X)^\mathrm{afd}} (\dAnSt_k)_{/ U}.
        \end{align*}
    It is clear from the definitions that $\cD\an_{X/ Y} \in \anFMP_{X/ /Y}$. We claim that the latter is a colimit of the diagram
        \[  
            \cD^{\mathrm{an}, \bullet}_{X/ Y}.
        \]
    We need to show that for every $Z \in \anFMP_{X/ /Y}$ together with a morphism
        \[
            \cD^{\mathrm{an}, \bullet}_{X/ Y} \to Z,   
        \]
    then there exists a uniquely defined (up to a contractible space of choices) morphism
        \[
            \cD\an_{X/ Y} \to Z,  
        \]
    in the \infcat $\anFMP_{X/ /Y}$. Moreover, we are reduced to check this property after base change along any $U \to Y^\wedge_X$, in which case the assertion follows immediately
    from the construction.
\end{proof}

\begin{defin}
    Let $f\colon X \to Y$ be a morphism of locally geometric derived $k$-analytic stacks. Then the \emph{deformation to the normal bundle associated to $f$}
    is by definition the object $\cD_{X/ Y} \in \anFMP_{X\times \bA^1_k/ /Y\times\bA^1_k}$, as in \cref{prop:gluing_def}. 
\end{defin}

\subsection{The Hodge filtration} Let $f \colon X \to Y$ be a morphism between locally geometric derived $k$-analytic stacks.
In this \S, we will describe the construction of the Hodge filtration on the deformation to the normal bundle $\cD_{X/ Y}\an$ associated to $f$.
We first deal with the $k$-affinoid case:

\begin{construction} \label{const:construction_of_Hodge_filtration_in_the_lafp_case}
    Let $f \colon X \to Y$ be a nil-embedding in the \infcat $\dAfd_k$.
    Consider the induced morphism
        \[
            f \alg \colon X\alg \to Y\alg,   
        \]
    where $X \alg = \Spec A$ and $Y\alg = \Spec B$, with
    \[A \coloneqq \Gamma(X, \cO_X\alg) \quad \mathrm{and} \quad B \coloneqq \Gamma(Y, \cO_Y \alg).\]
    The morphism $f\alg$ is a nil-embedding of affine derived schemes (combined \cite[Lemma 6.9]{Porta_Yu_NQK} and
    \cite[Proposition 3.17]{Porta_Yu_Derived_non-archimedean_analytic_spaces}). Consider then the nil-embedding
        \[
            g \coloneqq f \alg \times \id_{\bbA^1_k} \colon X \times \bbA^1_k \to Y \times \bbA^1_k,  
        \]
    in $\dAff_k$. By Noetherian approximation, we can write $g$ as an inverse limit of the form
        \[
            \lim_{\alpha \in I \op}  g_\alpha \colon \lim_{\alpha \in I \op} X_\alpha \times \bbA^1_k \to \lim_{\alpha \in I \op} Y_\alpha \times \bbA^1_k,  
        \]
    where $I$ is a filtered \infcat and for each index $\alpha \in I$, we have that
        \[
            g_\alpha \colon X_\alpha \times \bbA^1_k \to Y_\alpha \times \bbA^1_k,  
        \]
    is a closed immersion in the \infcat $\dAff_k^\laft$. Fix some $\alpha \in I$. Thanks to \cref{prop:algebraic_properties_of_deformation},
    there exists a sequence of square-zero extensions of the form
        \[
            X_\alpha \times \bbA^1_k = X^{(0)}_\alpha \hookrightarrow X_\alpha^{(1)} \hookrightarrow \dots \hookrightarrow X_\alpha^{(n)} \hookrightarrow \dots \to Y_\alpha \times \bbA^1_k,  
        \]
    such that each term comes equipped with a natural left-lax action of multiplicative monoid $\bbA^1_k \in \dSt_k$.
\end{construction}

\begin{lem} \label{lem:naturality_of_Hodge_filtration_w.r.t_Noetherian_approximation}
    Let  $n \ge 0$, and let $\alpha \to \beta$ be a morphism in $I \op$. Then the transition morphism
        \[
            X_\alpha \times \bbA^1_k \to X_\beta \times \bbA^1_k,  
        \]
    lifts to a well defined induced morphism
        \[
            X_\alpha^{(n)} \to X_\beta^{(n)},
        \]
    in $\dAff_k^\laft$.
\end{lem}

\begin{proof}
    The result follows immediately from the naturality of the construction in \cite[\S 9.5.1]{Gaitsgory_Study_II}.
\end{proof}

We now introduce the (algebraic) Hodge filtration associated with the closed immersion $f\alg \colon X\alg \to Y\alg$:

\begin{defin} \label{defin:algebraic_Hodge_filtration} Let $f \colon X \to Y$ be a nil-isomorphism in the \infcat $\dAfd_k$.
    Consider the morphism $g \coloneqq f \times \mathrm{id} \colon X\alg \times \bbA^1_k \to Y\alg \times \bbA^1_k$ above. \cref{lem:naturality_of_Hodge_filtration_w.r.t_Noetherian_approximation}
    implies that for each $n \ge 0$, we have a well defined object
        \[
            X^{\mathrm{alg}, (n)} \coloneqq \lim_{\alpha \in I\op} X_\alpha^{(n)} \in \dAff_k,  
        \]
    which fits into a sequence of square-zero extensions
        \[
            X\times \bbA^1_k = X^{ \mathrm{alg}, (0)} \hookrightarrow X^{\mathrm{alg}, (1)}  \hookrightarrow \dots \hookrightarrow X^{\mathrm{alg}, (n)} \hookrightarrow \dots \to Y\alg,
        \]
    in $\dAff_k$. We shall refer to the sequence of the $X^{\mathrm{alg}, (n)}$ as the \emph{algebraic Hodge filtration} associated to the morphism $f\alg$. 
    In the more general case, where $f \colon X \to Y$ exhibits $X$ as an analytic formal moduli problem we define, for each $n \ge 0$,
    the \emph{$n$-th piece of the Hodge filtration} as 
        \[
            X^{(n)} \coloneqq \colim_{S \in (\anNil^\cl_{X/ /Y})} S^{\mathrm{alg}, (n)}, 
        \]
    the filtered colimit being computed in the \infcat $\anFMP_{X \times \bA^1_k/ }$.
\end{defin}

\begin{construction} Let $f \colon X \to Y$ be a nil-embedding in the \infcat $\dAfd_k$. 
    It follows from \cref{prop:algebraic_properties_of_deformation} (4) that we have a natural morphism
        \begin{align} \label{eq:map_Hodge_filtration_to_deformation}
            \theta_{X/ Y} \colon \colim_{n \ge 0} X^{\mathrm{alg}, (n)} & \to \lim_{\alpha \in I\op} \cD_{X_\alpha/ Y_\alpha} \\
                                                                        & \simeq \cD_{X\alg/ Y\alg},
        \end{align}
    in the \infcat $\mathrm{FMP}_{X \times \bbA^1_k/ /Y \times \bbA^1_k}$. Passing to filtered colimits we produce a natural
    morphism as in \eqref{eq:map_Hodge_filtration_to_deformation} in the case where $(X \xrightarrow{f} Y) \in \anFMP_{/ Y}$ and $Y \in \dAfd_k$.
\end{construction}

The following result implies that the Hodge filtration associated to $f\alg$ does not depend on choices:

\begin{prop} \label{cor:commutation_between_limit_and_colimit_on_Hodge_filtration}
    Let $(X \xrightarrow{f} Y) \in \anFMP_{/ Y}$ and $Y \in \dAfd_k$. Then the natural morphism
        \[
            \theta_{X/ Y} \colon \colim_{n \ge 0} X^{\mathrm{alg}, (n)} \to \cD_{X\alg/ Y\alg},
        \] 
    is an equivalence of derived $k$-stacks.
\end{prop}

\begin{proof}
    Since both sides of $\theta_{X/ Y}$ are stable under filtered colimits it suffices to treat the case where $f \colon X \to Y$
    is a nil-isomorphism of derived $k$-affinoid spaces. We observe that $X \to Y$ exhibits $Y$ as a retract of the morphism
        \[
            X \to X \underset{X_\red}{\bigsqcup} Y.
        \]
    Since $\theta_{X/ Y}$ is stable under retracts, we are further reduced to the case where $f$ is a nil-embedding. In this case,
    the relative cotangent complex
        \[
            \bbL_{X\alg/ Y\alg} \simeq \bbL\an_{X/ Y},  
        \]
    is $1$-connective, see \cite[Lemma 6.9]{Porta_Yu_NQK}. The proof of \cref{lem:Noetherian_approximation_of_nil_isomorphisms} allow us to produce a presentation
        \[
            \lim_{\alpha \in I\op} (f_\alpha \colon X_\alpha \to Y_\alpha),  
        \]
    of the morphism $f\alg$ by closed immersions $f_\alpha \colon X_\alpha \to Y_\alpha$ of affine
    derived schemes almost of finite presentation parametrized by a filtered \infcat $I$. The proof of \cref{lem:identification_of_Hodge_and_adic_filtrations},
    implies that, for each $\alpha \in I$, the natural morphism
        \[
            X_\alpha^{(n)} \to X_\alpha^{(n+1)},
        \]
    exhibits $X_\alpha^{(n+1)}$ as a square-zero extension of $X_\alpha^{(n)}$ via a natural morphism
        \[
            \bbL_{X_\alpha^{(n)}} \to i_{n, *} \Sym^{n+1}(\bbL_{X_\alpha/ Y_\alpha}[-1])[1],  
        \]
    in the \infcat $\Coh^+(X_\alpha^{(n)})$. In particular, we deduce that, for each $n \ge 0$, the natural morphisms
        \[
            X_\alpha \times \bbA^1_k \to X_\alpha^{(n)},  
        \]
    are nil-embeddings of affine derived schemes and so are the natural morphisms
        \[
            X \alg \times \bbA^1_k \to X^{\mathrm{alg}, (n)}.  
        \]
    We are now able to show that $\theta_{X/ Y}$ is an equivalence in the \infcat $\dSt_k$.
    Thanks to \cite[Corollary 4.4.1.3]{Lurie_SAG} combined with \cite[Corollary 4.5.1.3]{Lurie_SAG}, the relative cotangent complex satisfies
        \[
            \bbL_{X\alg \times \bbA^1_k/ X^{(n)}} \simeq \colim_{\alpha \in I} h_\alpha^* \bbL_{X_\alpha \times \bbA^1_k/ X^{(n)}_\alpha} ,
        \]
    in the \infcat $\QCoh(X\alg \times \bbA^1_k)$, where $h_\alpha \colon X \to X_\alpha$ denotes the corresponding transition morphisms. Similarly, we have that
        \[
            \bbL_{X\alg \times \bbA^1_k/ Y \times \bbA^1_k} \simeq \colim_{\alpha \in I} h_\alpha^* \bbL_{X_\alpha \times \bbA^1_k / Y_\alpha \times \bbA^1_k} ,
        \]
    in the \infcat $\QCoh(X\alg \times \bbA^1_k)$. Moreover, the latter equivalences are compatible with the left-lax actions of the multiplicative monoid $\bbA^1_k \in \dSt_k$
    on
        \[
            X^{\mathrm{alg}, (n)} \quad \mathrm{and} \quad X^{(n)}_\alpha, \quad \mathrm{for \ each} \ \alpha \in I,  
        \]
    over $X\alg \times \bbA^1_k$ (resp. $X_\alpha \times \bbA^1_k$). Let $\lambda \in \bbA^1_k$ be such that $\lambda \neq 0$.
    By passing to Serre duals, we deduce from \cref{lem:identification_of_Hodge_and_adic_filtrations}
    that, for each $\alpha \in I$ and $n \ge 0$, the morphism
        \begin{equation} \label{eq:cot_complex_for_alpha}
            \bbL_{ (X_\alpha^{(n)})_\lambda} \to \Sym^{n+1}(\bbL_{X_\alpha / Y_\alpha}[-1])[1],
        \end{equation}
    classifies the square-zero extension associated to the fiber sequence
        \[
            \Sym^{n+1}(\bbL_{X_\alpha/ Y_\alpha}[-1]) \to \Sym^{\le n+1} (\bbL_{X_\alpha/ Y_\alpha}[-1]) \to \Sym^{\le n} ( \bbL_{X_\alpha/ Y_\alpha}[-1]).
        \]
    Since \eqref{eq:cot_complex_for_alpha} are stable under filtered colimits, we deduce, for each $n \ge 0$, that the morphism
        \[
            \bbL_{X^{\mathrm{alg}, (n)}} \to \Sym^{n+1}(\bbL_{X\alg/ Y\alg}[-1])[1],
        \]
    classifies the square-zero extension given by
        \[
            \Sym^{n+1}(\bbL_{X\alg/ Y\alg}[-1]) \to \Sym^{\le n+1}(\bbL_{X\alg/ Y\alg}[-1]) \to \Sym^{\le n}(\bbL_{X\alg/ Y\alg} [-1]).
        \]
    Thanks to \cite[\S 9.5.5.2]{Gaitsgory_Study_II} combined with an analogous Noetherian approximation reasoning as the one employed
    before we obtain a natural equivalence
        \[
            \bbL_{\colim_{n \ge 0}X^{\mathrm{alg}, (n)}_\lambda} \simeq \bbL_{X\alg/ Y\alg}.
        \]
    In particular, thanks to the algebraic version of \cref{prop:conservativity_of_relative_an_cot_complex} we deduce that the natural morphism
        \[
            \colim_{n \ge 0} X^{\mathrm{alg}, (n)}  \to \cD_{X\alg/ Y\alg}, 
        \]
    is an equivalence in $\mathrm{FMP}_{X \alg}$, for any $\lambda \bbA^1_k$ such that $\lambda \neq 0$. For $\lambda = 0$, the
    precise same reasoning applies using \cite[\S 9, Theorem 5.5.4]{Gaitsgory_Study_II} for the case of vector groups. The proof is thus concluded.
\end{proof}

We now introduce the \emph{non-archimedean Hodge filtration:}

\begin{defin}
    Let $Y \in \dAfd_k$ and $(X \xrightarrow{f} Y) \in \anFMP_{/Y}$. For each $n \ge 0$, we define the square-zero extension
        \[
            X \times \bA^1_k \hookrightarrow X^{(n)},  
        \]
    as the relative analytification, $(-) \an_Y$, of the square-zero extension
        \[
            X \alg \times \bbA^1_k \hookrightarrow X^{\mathrm{alg}, (n)},
        \]
    introduced in \cref{defin:algebraic_Hodge_filtration}.
    By construction, for each $n \ge 0$, we have structural morphisms
        \[
            X^{(n)} \to Y \times \bA^1_k.  
        \]
    We shall refer to the sequence
        \[
            X \times \bA^1_k \coloneqq X^{(0)} \hookrightarrow X^{(1)} \hookrightarrow \dots \hookrightarrow X^{(n)} \hookrightarrow \dots \to Y \times \bA^1_k,  
        \]
    as the \emph{non-archimedean Hodge filtration associated to the morphism $f$}.
\end{defin}

Putting together the above results we can easily deduce:

\begin{cor} \label{cor:Hodge_filtration_in_the_non-archimedean_setting}
    There exists a natural morphism
        \[
            \colim_{n \ge 0} X^{(n)} \to \cD\an_{X/ Y},  
        \]
    which is furthermore an equivalence in the \infcat $\anFMP_{X \times \bA^1_k/ /Y \times \bA^1_k}$.
\end{cor}

\begin{proof}
    Since the natural functor
        \[
            \anFMP_{X \times \bA^1_k/ /Y \times \bA^1_k} \to (\dAnSt_k)_{X \times \bA^1_k/ /Y \times \bA^1_k},
        \]
    commutes with filtered colimits, we are reduced to prove the statement of the Corollary in the \infcat $(\dAnSt_k)_{X \times \bA^1_k/ /Y \times \bA^1_k}$.
    The result is now an immediate consequence of \cref{prop:rel_analytification_preserves_the_deformation}
    and the fact that relative analytification commutes with filtered colimits (as it is defined as left Kan extension).
\end{proof}

\begin{cor}
    Let $f \colon X \to Y$ be a closed immersion of derived $k$-affinoid spaces. Denote by $A \coloneqq \Gamma(X, \cO_X\alg)$
    and $B \coloneqq \Gamma(Y, \cO_Y\alg)$ and $I \coloneqq \fib(B \to A)$. Then the Hodge filtration on $\cD_{X/ Y}\an$
    induces a natural filtration on $B^\wedge_I$, \{$\mathrm{Fil}_H^{n, \mathrm{an}} \}_{n \ge 0}$, together with natural equivalences
        \[
            B/ \mathrm{Fil}^n_H  \simeq \Sym^{\le n, \mathrm{an}}(\bbL_{X/ Y}\an[-1]),  
        \]
    in the \infcat $\CAlg_k$.
\end{cor}

\begin{proof}
    The claim of the proof follows from \cref{lem:identification_of_Hodge_and_adic_filtrations} (2) combined with Noetherian approximation
    and the fact that the relative analytification functor sends the relative cotangent complex $\bbL_{X/ Y}$ to $\bbL_{X/ Y}\an$, c.f. \cite[Theorem 5.21]{Porta_Yu_Representability}.
\end{proof}
We now globalize the Hodge filtration to geometric derived $k$-analytic stacks:

\begin{construction}
    Let $f \colon X \to Y$ be a morphism of locally geometric derived $k$-analytic stacks. Consider the formal completion
        \[
            X \to Y^{\wedge}_X \to Y,  
        \] 
    of the morphism $f$. Let 
        \[
            g \colon U \to Y^\wedge_X,  
        \]
    be a morphism in the \infcat $\dAnSt_k$, such that $U \in \dAfd_k$. Form the pullback diagram
        \[
        \begin{tikzcd}
            X_U \ar{r} \ar{d} & U \ar{d} \\
            X \ar{r}{f} & Y^\wedge_X
        \end{tikzcd}
        \]
    in $\dAnSt_k$. Let $V \to U$ be a morphism of derived $k$-affinoid spaces. Consider the pullback diagram
        \[
        \begin{tikzcd}
            X_V \ar{r} \ar{d} & X_U \ar{d} \\
            V \ar{r} & U,  
        \end{tikzcd}
        \]
    in the \infcat $\dAfd_k$. By naturality of the Hodge filtration we deduce that, for each $n \ge 0$, we have natural commutative diagrams
        \begin{equation} \label{eq:gluing_diagram_of_Hodge_filtrations}
        \begin{tikzcd}
            X_V^{(n)} \ar{r} \ar{d} & X_U^{(n)} \ar{d} \\
            V \ar{r} & U,
        \end{tikzcd}
        \end{equation}
    of the $n$-th pieces of the Hodge filtrations on $\cD_{X_U/ U}$ and $\cD_{X_V/ V}$, respectively.
\end{construction}

\begin{lem} \label{lem:preservation_of_Hodge_filtration_under_pullback}
    The commutative square in \eqref{eq:gluing_diagram_of_Hodge_filtrations} is a pullback square in the \infcat $\dAfd_k$.
\end{lem}

\begin{proof} As in the proof of \cref{cor:commutation_between_limit_and_colimit_on_Hodge_filtration}, we can reduce the statement to the case where $X_U \to U$ is a nil-embedding in the \infcat $\dAfd_k$
    in which case so it is $X_V \to V$.
    Thanks to \cref{lem:relative_anlaytification_of_closed_immersions_are_compatible_with_alg_construction}
    it suffices to prove that for each $n \ge 0$, the induced diagram
        \[
        \begin{tikzcd}
            X_V^{\mathrm{alg}, (n)} \ar{r} \ar{d} & X_U^{\mathrm{alg}, (n)} \ar{d} \\
            V \alg \ar{r} & U\alg  
        \end{tikzcd},
        \]
    is a pullback square in $\dAff_k$.
    By a standard argument of Noetherian approximation, we might assume that $U' \coloneqq U\alg $ and $V' \coloneqq V\alg$ are
    both affine derived schemes almost of finite presentation. We observe that for each $n \ge 0$, the morphisms
        \[
            g_n \colon X_{V'}^{(n)} \to X_{U'}^{(n)},
        \]
    are left-lax $\bbA^1_k$-equivariant. For this reason, they induce morphisms
        \[
            g_{n, *}^\mathrm{IndCoh}( \bbT_{X^{(n)}_{V'}/ V'})) \to \bbT_{X^{(n)}_{U'} / U'}  ,
        \]
    in the \infcat $\Ind\Coh(X_{V'}^{(n)})^\mathrm{fil}$. In particular, we have an induced morphism
    at the $n$-th piece of the filtration
        \[
                g_{n, *}^\mathrm{IndCoh}\big( \mathrm{Sym}^{n+1}(\bbT_{X^{(n)}_{U'}/ U'}[-1])[1]\big) \to  \Sym^{n+1}\big( \bbT_{X^{(n)}_{V'} / V'}[-1]\big)[1],
        \]
    (c.f. \cite[Theorem 9.5.1.3]{Gaitsgory_Study_II}). 
    The result is now a direct consequence of \cref{prop:algebraic_properties_of_deformation} (3)
    combined with \cite[Proposition 5.12]{Porta_Yu_Representability} and \cite[Proposition 2.5.4.5]{Lurie_SAG}.
\end{proof}

\begin{construction} Let $f \colon X \to Y$ be a locally geometric derived $k$-analytic stack. Denote by $(Y^\wedge_X)^\mathrm{afd}$ the \infcat
    of morphisms in $\dAnSt_k$
        \[
            U \to Y^\wedge_X,  
        \]
    where $U$ is derived $k$-affinoid.
    Consider the relative analytification functor
        \begin{align*}
            \lim_{U \in (Y^\wedge_X)^\mathrm{afd}} (-)\an_U \colon\lim_{U \in (Y^\wedge_X)^{\mathrm{afd}}} (\dSt_k)_{/ U\alg \times \bbA^1_k} & \to \lim_{U \in (Y^\wedge_X)^{\mathrm{afd}}} (\dAnSt_k)_{/ U \times \bA^1_k}  \\
                                                                                                                                                                            & \simeq (\dAnSt_k)_{/ Y^\wedge_X \times \bA^1_k}.
        \end{align*}
    Thanks to \cref{lem:preservation_of_Hodge_filtration_under_pullback}, for each $n \ge 0$, the algebraic Hodge filtration defines a well defined object in the limit
        \[
            \{ X^{\mathrm{alg}, (n)}_{U} \}_{U \in (Y^\wedge_X)^\mathrm{afd}}  \in \lim_{n \ge 0} (\dSt_k)_{/ U\alg}.
        \]
    Therefore, after taking its relative analytification we produce well defined objects 
        \[X^{(n)} \in (\dAnSt_k)_{Y^\wedge_X \times \bA^1_k},\]
    which we shall refer to as the \emph{Hodge filtration associated to the morphism $f$}.
    Moreover, 
    it follows essentially by construction that the latter restricts to the usual
    Hodge filtration, for each
        \[
            g \colon U \to Y^\wedge_X,  
        \]
    in $(Y^\wedge_X)^\mathrm{afd}$. Moreover, by construction we have a natural morphism
        \[
            \colim_{n \ge 0} X^{(n)} \to \cD_{X/ Y} \an,  
        \]
    which is an equivalence in the \infcat $\dAnSt_k$.
\end{construction}

We can thus summarize the previous results in the form of a main theorem:

\begin{thm} \label{thm:main_thm}
    Let $f \colon X \to Y$ be a morphism of locally geometric derived $k$-analytic stacks. Then the following assertions hold:
    \begin{enumerate}
        \item There exists a deformation to the normal bundle $\cD_{X/ Y}\an \in (\dAnSt_k)_{\bA^1_k}$ such that its fiber at $\{0 \} \subseteq \bA^1_k$
        coincides with the shifted $k$-analytic tangent bundle
            \[
                \bbT_{X/Y}\an[1]^\wedge  ,
            \]
        completed at the zero section $s_0 \colon X \to \bbT\an_{X/ Y}[1]$. Moreover, its fiber at $\lambda \neq 0$ coincides with the formal completion $Y^\wedge_X$ of $Y$ at $X$ along $f$;
        \item The object $\cD_{X/ Y}\an$ admits a left-lax equivariant action of the (multiplicative) monoid-object $\bA^1_k$ in $\dAnSt_k$;
        \item There exists a sequence of left-lax $\bA^1_k$-equivariant morphisms admitting a deformation theory
            \[
                X \times \bA^1_k = X^{(0)} \hookrightarrow X^{(1)} \hookrightarrow \dots \hookrightarrow X^{(n)} \hookrightarrow \dots \to Y \times \bA^1_k.
            \]
        Assume further that $f\colon X \to Y$ is a closed immersion of derived $k$-analytic spaces then, for each $n \ge 0$,
        the relative $k$-analytic cotangent complex
            \[
                \bbL\an_{X^{(n)}/ Y} \in \Coh^+(X^{(n)}),  
            \]
        is equipped with a (decreasing) filtration such that its $n$-th piece identifies canonically with $\Sym^{n+1}(\bbL\an_{X/ Y}[-1])[1]$.
        Moreover, for each $n \ge 0$, the morphisms
        $X^{(n)} \to X^{(n+1)}$ are square-zero extensions obtained via a canonical composite
            \[
                \bbL\an_{X^{(n)}} \to \bbL\an_{X^{(n)}/ Y} \to   \Sym^{n+1}(\bbL_{X/ Y}\an[-1])[1],
            \]
        in $\Coh^+(X^{(n)})$;
        \item We have a natural morphism
            \[
                \colim_{n \ge 0} X^{(n)} \to \cD_{X/ Y}\an,   
            \]
        which is furthermore an equivalence in the \infcat $\dAnSt_k$;
        \item Assume that $f \colon X \to Y$ is a locally complete intersection morphism of derived $k$-affinoid spaces with $A \coloneqq \Gamma(X, \cO_X\alg)$
        and $B \coloneqq \Gamma(Y, \cO_Y\alg)$. Then the corresponding filtration on (derived)
        global sections of the formal completion $Y^{\wedge}_X$ (via points (1) and (3)) identifies canonically with the $I$-adic filtration on
        $B$, where $I \coloneqq \fib(B \to A)$.
    \end{enumerate}
\end{thm}

\begin{proof}
    Claim (1) follows immediately from \cref{lem:identification_of_non-archimedean_fibers}. Assertion (2) follows immediately from \cref{prop:rel_analytification_preserves_the_deformation}, as
    in the algebraic case, the deformation $\cD_{X\alg/ Y}$ admits a natural left-lax $\bbA^1_k$-equivariant action. The statement in point (4) is essentially the content
    of \cref{cor:Hodge_filtration_in_the_non-archimedean_setting} and point (3) follows from the construction of the Hodge filtration in the non-archimedean setting combined with Noetherian approximation (for the relative cotangent complex) and
    \cref{prop:algebraic_properties_of_deformation} (3). Finally, claim (4) follows readily from \cref{lem:identification_of_Hodge_and_adic_filtrations} and the fact that
    both the Hodge filtration and the $I$-adic filtrations are compatible via relative analytification.
\end{proof}

\bibliographystyle{alpha}
\bibliography{dahema}

\end{document}